\documentclass[11pt]{article}
\usepackage{epsfig}
\usepackage{authblk}
\usepackage{epstopdf}
\usepackage{graphics}
\usepackage{amsmath}
\usepackage{amsthm}
\usepackage{amssymb}
\usepackage{array}
\usepackage{subfig} 
\usepackage{bbm}
\usepackage[font={footnotesize}]{caption}
\usepackage{relsize}
\usepackage[nottoc]{tocbibind}
\usepackage{hyperref}
\usepackage{float}

\usepackage{chngcntr}
\usepackage{wrapfig}

\counterwithin*{equation}{section}
  
\newtheorem{thm}{Theorem}[section] 
\newtheorem{lem}[thm]{Lemma}  
\newtheorem{cor}[thm]{Corollary} 
\newtheorem{prop}[thm]{Proposition}  

\newtheorem{hyp}{Hypothesis}

\newtheoremstyle{named}{}{}{\itshape}{}{\bfseries}{.}{.5em}{\thmnote{#3's }#1}
\theoremstyle{named}

\theoremstyle{definition}
\newtheorem{rmk}{Remark}

\title{Stable Periodic Solutions to Lambda-Omega Lattice Dynamical Systems}

\author{Jason J. Bramburger\\ Division of Applied Mathematics\\ Brown University \\ Providence, Rhode Island 02912\\ USA}

\date{} 

\begin{document} 

\maketitle

\abstract{In this manuscript we consider the stability of periodic solutions to Lambda-Omega lattice dynamical systems. In particular, we show that an appropriate ansatz casts the lattice dynamical system as an infinite-dimensional fast-slow differential equation. In a neighbourhood of the periodic solution an invariant slow manifold is proven to exist, and that this slow manifold is uniformly exponentially attracting. The dynamics of solutions on the slow manifold become significantly more complicated and require a more delicate treatment. We present sufficient conditions to guarantee convergence on the slow manifold which is algebraic, as opposed to exponential, in the slow-time variable. Of particular interest to our work in this manuscript is the stability of a rotating wave solution, recently found to exist in the Lambda-Omega systems studied herein.}

\section{Introduction} 

At their most general a lattice dynamical system (LDS) takes the form
\begin{equation} \label{GeneralLattice}
	\dot{u}_{\xi} = g_{\xi}(\{u_{\zeta}\}_{\zeta\in\Lambda}),\quad \xi \in \Lambda. \\
\end{equation}
Here $\Lambda$ is a discrete subset of $\mathbb{R}^n$, commonly referred to as a lattice, and the $u_\xi = u_\xi(t)$ are time-dependent functions indexed by the lattice. We consider the variables $u_\xi$ for each $\xi \in \Lambda$ to be coordinates in the state vector $u = \{u_\xi\}_{\xi\in\Lambda}$ and $g_\xi$ a function on these coordinates which governs the flow of each $u_\xi$. Lattice models of the form (\ref{GeneralLattice}) have been shown to be well-suited in such areas as chemical reaction theory $\cite{Erneux,Erneux2}$, quantum mechanics $\cite{Kevrekidis1,Kevrekidis2,Temam}$, models of neural networks $\cite{Ermentrout,ErmentroutKopell2}$, optics $\cite{Firth}$ and material science \cite{Cahn,Cook}. In many of these applications one considers the elements $u_\xi$ to be interacting particles whose behaviour influences, and is influenced by, a finite subset of elements of the state vector $u = \{u_\xi\}_{\xi\in\Lambda}$. Despite spatially continuous models such as partial differential equations (PDEs) being the traditional means by which such situations are characterized mathematically, it appears that the discrete nature of the lattice model is better suited to reflect the discrete nature of the physical setting which the models are attempting to describe.  

In many interesting mathematical investigations of LDSs one typically considers the index set $\Lambda$ to be countably infinite, thus making (\ref{GeneralLattice}) an infinite system of coupled ordinary differential equations. This infinite-dimensional setting of course necessitates more abstract and technical analytical tools to investigate the behaviour of solutions to systems of the form (\ref{GeneralLattice}). In particular, the problem of determining local asymptotic stability of a known solution to (\ref{GeneralLattice}) greatly increases in complexity when moving from the finite-dimensional to the infinite-dimensional setting. In the case when $\Lambda = \mathbb{Z}$ some authors have circumvented this difficulty through the use of comparison principles to obtain stability of traveling waves solutions to LDSs \cite{Chen,Guo,Hsu,Ma,Zhang,ZinnerStable}. The problem with this method is that it requires a number of assumptions on the model and can only capture a specific class of initial conditions that converge back to the given solution. This was exactly the point made in \cite{Hoffman} where the authors examine the local asymptotic stability of traveling waves with $\Lambda = \mathbb{Z}^2$ in a more traditional dynamical systems context, which in turn is used to inform our present analysis.   

In this work we aim to continue the discussion of local asymptotic stability of solutions to LDSs on higher-dimensional lattices by investigating the stability of periodic solutions to the so-called Lambda-Omega system
\begin{equation} \label{IntroLDS} 
	\dot{z}_{i,j} = \alpha(z_{i+1,j} + z_{i-1,j} + z_{i,j+1} + z_{i,j-1} - 4z_{i,j}) + z_{i,j}[\lambda(|z_{i,j}|) + {\rm i}\omega(|z_{i,j}|,\alpha)], \quad (i,j) \in \mathbb{Z}^2,	
\end{equation} 
 which we write in terms of the time-dependent complex variables $z_{i,j} = z_{i,j}(t)$. Here $\alpha \geq 0$ is often referred to as the {\em coupling coefficient} and it describes the strength of interaction between neighbouring elements of the lattice. The inclusion of the terms $z_{i\pm1,j}$ and $z_{i,j\pm1}$ on the righthand side of (\ref{IntroLDS}) represents a local coupling over the two-dimensional lattice $\mathbb{Z}^2$, referred to as nearest-neighbour coupling. Such a coupling can be derived as a leading order approximation of the the typical five point discretization of the Laplacian differential operator, leading one to view system (\ref{IntroLDS}) as a spatially discretized reaction-diffusion PDE (see \cite[Section~1.1]{Bramburger2} for full details). When $\alpha \geq 0$ is taken to be a small parameter in the system, one may alternatively view (\ref{IntroLDS}) as an infinite system of weakly coupled oscillators, since when $\alpha = 0$ the $z_{i,j}$ act independently of each other. The specifics of the functions $\lambda$ and $\omega$ will be detailed later in this manuscript, but the important characteristic to consider at this point is that they only depend on the modulus of the complex variable $z_{i,j}$, endowing (\ref{IntroLDS}) with a rotating symmetry invariance. In this way, the choices for $\lambda$ and $\omega$ are meant to mimic the behaviour of the normal form of a Hopf bifurcation in the ordinary differential equation setting. 
 
Since their introduction by Howard and Kopell, Lambda-Omega PDEs have long been studied as an archetype for oscillatory behaviour in reaction-diffusion systems \cite{Kopell}. In particular, Lambda-Omega PDEs are presented as generalizations of the complex Ginzburg-Landau equation, which is well-known to be the truncated normal form of a reaction-diffusion equation undergoing a Hopf bifurcation \cite{Cohen}. To date Ginzburg-Landau PDEs have been shown to manifest themselves as the dominant leading order perturbation in a wide class of partial differential equations, including for example the Swift-Hohenberg equation, therefore testifying to the universality of these equations \cite{Melbourne}. Hence, Lambda-Omega systems should not be viewed as a specific model of a physical phenomenon, but as a paradigm for understanding periodic behaviour. Furthermore, investigations of solutions on the regular structure of the square lattice provide the mathematical community with significant insight into the dynamics of more general systems posed on more general lattices.  

From a mathematical point of view, many investigations have demonstrated that Lambda-Omega PDEs exhibit rotating wave solutions \cite{Cohen,Troy,Greenberg,Kopell2}, thus prompting the recent investigation into rotating wave solutions to LDSs of the type (\ref{IntroLDS}) \cite{Bramburger1,Bramburger2}. Hence, this work provides a natural follow-up investigation to \cite{Bramburger1,Bramburger2} by providing a series of sufficient conditions which can determine the stability of rotating wave solutions to (\ref{IntroLDS}). In this manuscript we dedicate a significant portion of the section on applications to discussing how these results should apply to rotating waves. The reason for this is that rotating waves have been an intense area of mathematical investigation for decades now, and this manuscript aims at continue that investigation within the framework of lattice dynamical systems. In the scientific literature, the study of spiral waves, which are a particular example of a rotating waves, dates back at least to the work of Winfree who observed spiral wave concentration patterns in chemical reactions \cite{Winfree}. Since Winfree's work spiral waves have been observed in a number of excitable systems ranging from retinal spreading depression, to fertilizing {\em Xenopus} oocyte calcium waves, and ventricular fibrillation $\cite{Beaumont,Gorelova,Huang,Cardiac,KeenerSneyd,Lechleiter,Cortical}$. With spiral waves abounding in nature, it follows that they remain an active area of investigation throughout the physical sciences.  

The importance of our focus on the stability of rotating wave solutions to (\ref{IntroLDS}) is that even in the PDE setting little is known about the stability of rotating waves. Therefore any insight into the stability of rotating waves should be of value to a wide variety of researchers due to their prevalence in chemical and biological systems. Numerical evidence and heuristic arguments in \cite{Hagan} appear to indicate that reaction-diffusion PDE should be able to exhibit asymptotically stable single-armed rotating waves, but this remains to be verified rigorously. Here we add to this mounting evidence by providing numerical simulations and heuristic arguments that indicate that rotating wave solutions should in fact satisfy the sufficient conditions for local asymptotic stability laid out in this manuscript. Moreover, the approach taken in this manuscript could influence future investigations of nonlinear stability of rotating waves in the spatially continuous setting of PDEs since we provide a method of extending linear stability to nonlinear stability while demonstrating how we lose some decay of the solutions due to the nonlinear terms.          

In this work we show that after introducing an appropriate ansatz into system (\ref{IntroLDS}) it can be interpreted as an infinite-dimensional fast-slow dynamical system when $0 \leq \alpha \ll 1$. The fast-slow nature of the resulting dynamical system requires the understanding of the asymptotic behaviour of solutions on two different time-scales: one fast and one slow. We show that an invariant manifold in the form of an infinite-dimensional torus persists for sufficiently small $\alpha \geq 0$, and that this invariant manifold is locally asymptotically stable with an exponential rate of decay. Solutions on the invariant manifold evolve in the slow time variable, leading to the nomenclature that the invariant manifold is a slow manifold. Upon reducing to the slow manifold we are able to extend previous methods presented in \cite{Bramburger3} to investigate the stability of solutions on the manifold. We find that on the slow manifold solutions decay back to equilibrium at an algebraic (as opposed to exponential) rate provided they start sufficiently close to that equilibrium. The notion of closeness is one that requires extra attention in this infinite-dimensional setting since it depends on the Banach space in which one is measuring distance. Hence, we pay special attention to describing the Banach spaces in which initial conditions belong to, as well as how solutions decay with respect to a variety of different norms.    

Although the analysis of this manuscript would be significantly eased by considered the Lambda-Omega system (\ref{IntroLDS}) on a finite-dimensional truncation of the integer lattice $\mathbb{Z}^2$, the goal of this manuscript is to discuss stability in the most generality possible. That is, it is entirely possible that one may provide similar asymptotic stability results on a finite lattice, but this work would most likely fail to discuss how stability changes as the size of the lattice increases. A common feature of increasing the size of the lattice, and in turn increasing the dimensionality of the phase space, is that the spectral gap of the linearization about an equilibrium shrinks to zero as the number of variables tends to infinity. In the context of PDEs this is most easily seen by inspecting the spectrum of the heat equation in appropriate Hilbert spaces on an infinite domain (all of the real nonpositive numbers) versus a finite domain (a discrete decreasing sequence of nonpositive numbers). This is a common feature of investigating PDEs on truncated domains \cite{SandstedeScheel} which leads one to believe that studying a finite-dimensional truncation of (\ref{IntroLDS}) would be insufficient to gain an understanding of the stability in general. Moreover, it is entirely possible that the truncation from infinite to finite could in turn stabilize otherwise unstable solutions. Hence, this investigation aims to understand the stability of periodic solutions of (\ref{IntroLDS}) in the absence of boundary conditions by posing the problem on an infinite lattice.  

The election to study the infinite-dimensional lattice dynamical system (\ref{IntroLDS}) is further motivated by the idea that studying LDSs obtained by discretizing PDEs leads to a greater understanding of the potentially more complicated PDE itself. One such example of this is that the atomization of continuous space to discrete space leading to (\ref{IntroLDS}) leads to a bounded coupling operator which displays similar semigroup properties to the unbounded Laplacian operator of a Lambda-Omega PDE. It is in this way that working with LDSs can slightly ease the analysis while also inform the behaviour of a related PDE.   

This paper is organized as follows. We begin with a discussion of the relevant Banach spaces which are used throughout, as well as present a series of results pertaining to a family of semi-norms which will be integral to this work. In Section~\ref{sec:Main} we detail the assumptions made on system (\ref{IntroLDS}) and provide the major results of this work. Prior to proving our main results of this work we provide a number of applications of our main theorems in Section~\ref{sec:Applications}. In particular, we discuss how these results should apply to the rotating wave solution found in \cite{Bramburger1,Bramburger2}. Following this discussion of the applications of this work, the entirety of Section~\ref{sec:InvMan} is dedicated to proving the existence and local asymptotic stability of the invariant slow manifold discussed above. Then, in Section~\ref{sec:Thm2Proof} we prove the local asymptotic stability of solutions on the slow manifold. Finally, this paper concludes with Section~\ref{sec:Discussion} which provides concluding remarks on the work undertaken in this paper, as well as details avenues for future work.

\section{Spatial Settings} \label{sec:SpatialSettings} 

Prior to presenting the main hypotheses and results of this work, we provide the following discussion regarding the appropriate spatial settings for solutions of (\ref{LambdaOmegaLDS}). To begin, the Banach spaces which will be of primary interest throughout this work will be sequence spaces indexed by a countably infinite index set $\mathbb{Z}^2$. In particular, our attention will be focussed on the spaces $\ell^p(\mathbb{Z}^2)$ with $p \in [1,\infty]$. They are defined as follows:
\begin{equation}\label{ellp}
	\ell^p(\mathbb{Z}^2) = \bigg\{x = \{x_{i,j}\}_{(i,j) \in \mathbb{Z}^2}:\ \sum_{(i,j) \in \mathbb{Z}^2} |x_{i,j}|^p < \infty\bigg\},
\end{equation}  
for all $p \in [1,\infty)$ and 
\begin{equation}
	\ell^\infty(\mathbb{Z}^2) = \bigg\{x = \{x_{i,j}\}_{(i,j) \in \mathbb{Z}^2}:\ \sup_{n \in \mathbb{Z}^2} |x_{i,j}| < \infty\bigg\}.
\end{equation}  
It is well-known that $\ell^p(\mathbb{Z}^2)$ is complete (and therefore a Banach space) under the norm
\begin{equation}
	\|x\|_p = \bigg(\sum_{(i,j) \in \mathbb{Z}^2} |x_{i,j}|^p\bigg)^\frac{1}{p},
\end{equation}
and similarly the norm associated to $\ell^\infty(\mathbb{Z}^2)$ is given by
\begin{equation} \label{InfNorm}
	\|x\|_\infty = \sup_{(i,j) \in \mathbb{Z}^2} |x_{i,j}|.
\end{equation}
Since the index set will always be $\mathbb{Z}^2$, for the ease of notation we will write $\ell^p(\mathbb{Z}^2)$ as $\ell^p$. It should be noted that $\ell^1 \subset \ell^p$ for every $p > 1$, and hence posing solutions to (\ref{LambdaOmegaLDS}) in $\ell^1$ will allow for the discussion of the behaviour of solutions with respect to all $\ell^p$ norms. Due to the subscript indices of elements in $\ell^p$, throughout this manuscript we will write initial conditions as $x^0 = \{x^0_{i,j}\}_{(i,j)\in\mathbb{Z}^2}$. This avoids the confusion created by using the traditional notation of $x_0$, where the meaning of the subscript could be ambiguous to the reader since $x_0$ could represent an element of $\ell^p$ or a single element of the sequence of an element in $\ell^p$.  

For convenience, throughout this manuscript we will introduce the shorthand
\begin{equation}\label{SumShorthand}
	\sum_{i',j'} (x_{i',j'} - x_{i,j}) := (x_{i+1,j} - x_{i,j}) + (x_{i-1,j} - x_{i,j}) + (x_{i,j+1} - x_{i,j}) + (x_{i,j-1} - x_{i,j}).
\end{equation}
Then, using this shorthand we will consider the real-valued functions on $\ell^p$ given as
\begin{equation}\label{Qp}
	Q_p(x) := \bigg(\sum_{(i,j)\in\mathbb{Z}^2}\sum_{(i',j')} |x_{i',j'} - x_{i,j}|^p\bigg)^\frac{1}{p},
\end{equation}
for all $x \in \ell^p$, $p \in [1,\infty)$. Similarly, define $Q_\infty$ as
\[
	Q_\infty(x) := \sup_{(i,j)\in\mathbb{Z}^2} \sum_{(i',j')} |x_{i',j'} - x_{i,j}|, 
\] 
for all $x \in \ell^\infty$. The terms $(x_{i',j'} - x_{i,j})$ can be interpreted as discrete directional derivatives along the horizontal and vertical directions of the lattice. Hence, the functions $Q_p$ can be understood as the discrete analogue of the norm of the gradient of a function in the Lesbesgue measure spaces. We now provide the following lemma to show that the functions $Q_p$ are indeed well-defined.

\begin{lem} \label{lem:NormBnds} 
	For every $1 \leq p \leq p' \leq \infty$ and $x \in \ell^p$ we have:
	\[
		Q_{p'}(x) \leq Q_p(x) \leq 8\|x\|_p.
	\]
\end{lem}

\begin{proof}
	The proof of the left inequality follows in exactly the same way as showing that $\|x\|_{p'} \leq \|x\|_p$, and therefore we are left to prove the rightmost inequality. Begin by noting that a simple application of the triangle inequality gives
	\[
		 |x_{i',j'} - x_{i,j}| \leq |x_{i',j'}| + |x_{i,j}|,	
	\]	
	for every $(i,j)\in\mathbb{Z}^2$ and a nearest-neighbour $(i',j')$. Then, for $1\leq p < \infty$, the triangle inequality on the $\ell^p$ spaces imply that 
	\[
	\begin{split}
		Q_p(x) &\leq \bigg(\sum_{(i,j)\in\mathbb{Z}^2} |x_{i+1,j}|^p\bigg)^\frac{1}{p} + \bigg(\sum_{(i,j)\in\mathbb{Z}^2} |x_{i-1,j}|^p\bigg)^\frac{1}{p} + \bigg(\sum_{(i,j)\in\mathbb{Z}^2} |x_{i,j+1}|^p\bigg)^\frac{1}{p} \\ 
		&+ \bigg(\sum_{(i,j)\in\mathbb{Z}^2} |x_{i,j-1}|^p\bigg)^\frac{1}{p} + 4\bigg(\sum_{(i,j)\in\mathbb{Z}^2} |x_{i,j}|^p\bigg)^\frac{1}{p} \\
		&= 8\|x\|_p,   
	\end{split}
	\]
	as claimed. The case when $p = \infty$ follows in a similar way, and is omitted.
\end{proof} 

Throughout this work we will make use of the following simple bound.

\begin{lem}\label{lem:AltGradient} 
	For every $p \in [1,\infty)$ and $x \in \ell^p$ we have 
	\[
		\bigg(\sum_{(i,j)\in\mathbb{Z}^2}\bigg(\sum_{(i',j')} |x_{i',j'} - x_{i,j}|\bigg)^p\bigg)^\frac{1}{p} \leq 4Q_p(x).
	\]
\end{lem} 

\begin{proof}
	Begin by fixing some $p \in [1,\infty)$ and let $q \in [1,\infty)$ be its H\"older conjugate. Then, for all $x \in \ell^p$ H\"older's inequality gives
	\[
		\sum_{(i',j')} |x_{i',j'} - x_{i,j}| \leq 4^\frac{1}{q}\bigg(\sum_{(i',j')} |x_{i',j'} - x_{i,j}|^p\bigg)^\frac{1}{p} \leq 4\bigg(\sum_{(i',j')} |x_{i',j'} - x_{i,j}|^p\bigg)^\frac{1}{p},	
	\]
	uniformly in $(i,j)\in\mathbb{Z}^2$, since $q \geq 1$. The stated inequality now follows since
	\[
		\bigg(\sum_{(i,j)\in\mathbb{Z}^2}\bigg(\sum_{(i',j')} |x_{i',j'} - x_{i,j}|\bigg)^p\bigg)^\frac{1}{p} \leq 4\bigg(\sum_{(i,j)\in\mathbb{Z}^2}\sum_{(i',j')} |x_{i',j'} - x_{i,j}|^p\bigg)^\frac{1}{p} = 4Q_p(x).	
	\]
\end{proof} 

We note that the functions $Q_p$ should be interpreted as semi-norms since they annihilate constant sequences\footnote{In fact, one can easily prove that $Q_p(x) = 0$ if, and only if, $x$ is a constant sequence.}. However, since the constant sequences only belong to $\ell^\infty$, this will not pose a problem to our analysis since we primarily focus on elements in $\ell^1$, which does not include the constant sequences. Furthermore, we will see that it is indeed an understanding of the behaviour of solutions with respect to the $Q_p$ semi-norms that influences our understanding of the behaviour of solutions with respect to the $\ell^p$ norms. Hence, this work aims to convince the reader that a complete discussion of stability in (\ref{LambdaOmegaLDS}) requires the introduction of the $Q_p$ semi-norms. Having now done so, we are now able to present the hypotheses and main results of this work.

\section{Hypotheses and Main Results}\label{sec:Main} 

To begin, using the shorthand (\ref{SumShorthand}) we find that (\ref{IntroLDS}) can be written compactly as
\begin{equation}\label{LambdaOmegaLDS} 
	\dot{z}_{i,j} = \alpha\sum_{i',j'}(z_{i',j'} - z_{i,j}) + z_{i,j}[\lambda(|z_{i,j}|) + {\rm i}\omega(|z_{i,j}|,\alpha)], \quad (i,j) \in \mathbb{Z}^2.	
\end{equation}
We make the following hypothesis on the functions $\lambda$ and $\omega$ in the differential equation (\ref{LambdaOmegaLDS}).

\begin{hyp} \label{hyp:LambdaOmega} 
	The functions $\lambda$ and $\omega$ in $(\ref{LambdaOmegaLDS})$ satisfy the following:
	\begin{itemize}
	\item[{\rm(1)}] $\lambda : [0,\infty) \to \mathbb{R}$ is continuously differentiable and there exists some $a> 0$, with the property that $\lambda(a) = 0$ and $\lambda'(a) < 0$.
	\item[{\rm(2)}] $\omega = \omega(R,\alpha)$ is of the form
	\begin{equation}
		\omega(R,\alpha) = \omega_0(\alpha) + \alpha \omega_1(R,\alpha),
	\end{equation} 
	for some function $\omega_0:\mathbb{R} \to \mathbb{R}$ and $\omega_1(R,\alpha): [0,\infty) \times \mathbb{R} \to \mathbb{R}$ twice continuously differentiable with $\omega_1(a,\alpha) = 0$ for all $\alpha \in \mathbb{R}$. 
\end{itemize}
\end{hyp} 

These conditions are based upon the assumptions laid out in the first demonstration of the existence of rotating waves in the continuous spatial setting \cite{Cohen,Greenberg}. Furthermore, conditions of this form were used in \cite{Bramburger2} to demonstrate the existence of rotating wave solutions to equation (\ref{LambdaOmegaLDS}), which serve as a major motivation for the analysis herein. We note that the assumptions to demonstrate the existence of rotating wave solutions to (\ref{LambdaOmegaLDS}) only required that $\lambda'(a) \neq 0$, but now to demonstrate asymptotic stability of solutions to (\ref{LambdaOmegaLDS}) we require $\lambda'(a) < 0$. To understand why this is the case, begin by setting $\alpha = 0$ in (\ref{LambdaOmegaLDS}) to arrive at the infinite system of uncoupled equations
\[
	\dot{z}_{i,j} = z_{i,j}[\lambda(|z_{i,j}|) + {\rm i}\omega_0(|z_{i,j}|,0)],	
\]
for all $(i,j)\in\mathbb{Z}^2$. Decomposing each $z_{i,j}$ into polar variables using the ansatz 
\begin{equation}
	z_{i,j}(t) = r_{i,j}(t)e^{{\rm i}\theta_{i,j}(t)}
\end{equation}
results in the set of ordinary differential equations
\begin{equation} \label{PolarUncoupled}
	\begin{split}
	\begin{aligned}
		&\dot{r}_{i,j} = r_{i,j}\lambda(r_{i,j}), \\
		&\dot{\theta}_{i,j} = \omega_0(0),
	\end{aligned}
	\end{split}
\end{equation}
for each $(i,j)\in\mathbb{Z}^2$. Taking $r_{i,j} = a$ leads to a periodic solution of the form
\begin{equation} \label{UncoupledPeriodicSoln}
	z_{i,j}(t) = ae^{{\rm i}(\omega_0(0)t + \theta_{i,j}^0)},
\end{equation}
where $\theta_{i,j}^0 \in S^1$ is an initial phase value for each $(i,j) \in \mathbb{Z}^2$. Hence, assuming $\lambda'(a) \neq 0$, each periodic solution of the form (\ref{UncoupledPeriodicSoln}) to (\ref{PolarUncoupled}) falls into one of two categories: locally attracting when $\lambda'(a) < 0$ and locally repelling when $\lambda'(a) > 0$. Since we are interested in the stability of solutions to (\ref{LambdaOmegaLDS}) with $\alpha$ in a connected neighbourhood to the right of $\alpha = 0$, we therefore must focus on the case $\lambda'(a) < 0$. 

Turning now to the function $\omega$, we first note that many applications simply work with $\omega$ as a constant function (or at least independent of its first argument). We will see in the coming sections that our nonlinear stability result relies on having $\omega$ be independent of its first argument, but we take this time to comment on what should be expected of nonlinear $\omega$ functions. When extending to non-constant $\omega$ functions we find that similar work exploring rotating waves in Lambda-Omega systems on finite lattices have considered functions $\omega$ to be slight perturbations off of a constant function \cite{ErmentroutLambdaOmega}. We see that indeed this is the case when $\alpha \geq 0$ is taken to be a small parameter in the system, which therefore measures the deviation of $\omega$ from a constant function. Furthermore, the condition $\omega_1(a,\alpha) = 0$ implies that 
\[
	\omega_1(R,\alpha) = \mathcal{O}(|R-a|),
\]   
simply using Taylor's Theorem. This condition is weaker than that which was assumed by Cohen et. al. in their proof of existence of spiral waves in the spatially continuous reaction-diffusion setting \cite{Cohen}. Particularly, they assumed the H\"older regularity condition 
\[
	\omega_1(R,\alpha) = \mathcal{O}(|R-a|^{1 + \mu}), 	
\]  
for some $\mu > 0$. Hence, we will see that our present investigation allows for a slightly larger class of functions $\omega$ to be considered for the first major result detailing the existence of an exponentially attracting invariant manifold, but we have to reduce ourselves to the constant $\omega$ case to present our nonlinear stability result.

To analyze the full system (\ref{LambdaOmegaLDS}) we follow in a similar way in which we inspected the uncoupled system (\ref{PolarUncoupled}) above and introduce the polar decomposition 
\begin{equation} \label{PolarAnsatz}
	z_{i,j} = r_{i,j}e^{{\rm i}(\omega_0(\alpha) t + \theta_{i,j})}, 
\end{equation}	
with $r_{i,j} =r_{i,j}(t)$ and $\theta_{i,j} = \theta_{i,j}(t)$ for each $(i,j)\in\mathbb{Z}^2$. Then the LDS ($\ref{LambdaOmegaLDS}$) can now be written in polar form as
\begin{equation} \label{FullPolarLattice}
	\begin{split}
		&\dot{r}_{i,j} = \alpha\sum_{i',j'} (r_{i',j'}\cos(\theta_{i',j'} - \theta_{i,j}) - r_{i,j}) + r_{i,j}\lambda(r_{i,j}), \\
		&\dot{\theta}_{i,j} =  \alpha\sum_{i',j'} \frac{r_{i',j'}}{r_{i,j}}\sin(\theta_{i',j'} - \theta_{i,j}) + \alpha\omega_1(r_{i,j},\alpha), \ \ \ (i,j) \in \mathbb{Z}^2.
	\end{split}
\end{equation}  
Note that for $0 \leq \alpha \ll 1$ the phase components become singularly perturbed, hence giving that in the small $\alpha > 0$ limit this polar decomposition is of the form of a fast-slow system of ordinary differential equations. We now make the following assumption.

\begin{hyp}\label{hyp:PolarSoln} 
	There exists $\alpha^* > 0$ such that for all $\alpha \in [0,\alpha^*]$ there exists a steady-state solution, denoted $\{\bar{r}_{i,j}(\alpha),\bar{\theta}_{i,j}(\alpha)\}_{(i,j)\in\mathbb{Z}^2}$, to (\ref{FullPolarLattice}). That is, 
	\[
		\begin{split}
			&0 = \alpha\sum_{i',j'} (\bar{r}_{i',j'}(\alpha)\cos(\theta_{i',j'} - \theta_{i,j}) - \bar{r}_{i,j}(\alpha)) +\bar{r}_{i,j}(\alpha)\lambda(\bar{r}_{i,j}(\alpha)), \\
			&0 =  \sum_{i',j'} \frac{\bar{r}_{i',j'}(\alpha)}{\bar{r}_{i,j}(\alpha)}\sin(\bar{\theta}_{i',j'}(\alpha) - \bar{\theta}_{i,j}(\alpha)) + \omega_1(\bar{r}_{i,j}(\alpha),\alpha),
		\end{split}	
	\]
	for all $\alpha \in [0,\alpha^*]$. Furthermore, $\bar{r}_{i,j}(0) = a$ and there exists a constant $C_r > 0$ such that  
	\begin{equation}\label{RotWaveLip}
		|\bar{r}_{i,j}(\alpha) - a| \leq C_r\alpha	
	\end{equation}
\end{hyp} 

\begin{rmk}
	Without loss of generality we will restrict $\alpha^* > 0$ such that $|\bar{r}_{i,j}(\alpha) - a| \leq \frac{a}{2}$ for all $(i,j)\in\mathbb{Z}^2$ using (\ref{RotWaveLip}). This will allow our analysis to be bounded away from the singularity at $r_{i,j}(\alpha) = 0$ in the differential equation for $\theta_{i,j}$ in (\ref{FullPolarLattice}). 
\end{rmk}

One sees that Hypothesis~\ref{hyp:PolarSoln} results in a periodic solution $\{z_{i,j}(t;\alpha)\}_{(i,j)\in\mathbb{Z}^2}$ of the form 
\[
	z_{i,j}(t;\alpha) = \bar{r}_{i,j}(\alpha)e^{{\rm i}(\omega_0(\alpha) t + \bar{\theta}_{i,j}(\alpha))},
\] 
for all $(i,j)\in\mathbb{Z}^2$. Here we see that each element of the lattice is oscillating with a frequency of $2\pi/\omega_0(\alpha)$, but potentially differs through the amplitude of its oscillation, $\bar{r}_{i,j}(\alpha)$, and/or a phase-lag, $\bar{\theta}_{i,j}(\alpha)$. Of course there is a trivial choice for a solution satisfying Hypothesis~\ref{hyp:PolarSoln} given by $\bar{r}_{i,j}(\alpha) = a$ and $\bar{\theta}_{i,j}(\alpha) = 0$ for all $(i,j) \in \mathbb{Z}^2$ and $\alpha \geq 0$. Aside from this trivial solution, it was shown in \cite{Bramburger1,Bramburger2} that there also exists a rotating wave solution satisfying Hypothesis~\ref{hyp:PolarSoln} for sufficiently small $\alpha \geq 0$. On top of these two solutions, the methods employed to obtain the rotating wave solution can be easily extended to obtain various other steady-states of (\ref{FullPolarLattice}) which satisfy Hypothesis~\ref{hyp:PolarSoln}, leading to oscillatory solutions of (\ref{LambdaOmegaLDS}). 

To properly state our results we centre system (\ref{FullPolarLattice}) at the steady-state $(\bar{r}(\alpha),\bar{\theta}(\alpha))$.  Let us introduce the changes of variable given by 
\begin{equation}
	r_{i,j} = \bar{r}_{i,j}(\alpha) + s_{i,j}, \ \ \ \ \ \theta_{i,j} = \bar{\theta}_{i,j}(\alpha) + \psi_{i,j}, 
\end{equation}
for all $(i,j)\in\mathbb{Z}^2$. Letting $s = \{s_{i,j}\}_{(i,j)\in\mathbb{Z}^2}$ and $\psi = \{\psi_{i,j}\}_{(i,j)\in\mathbb{Z}^2}$, we define the function 
\[
	F(s,\psi,\alpha) = \{F_{i,j}(s,\psi,\alpha)\}_{(i,j)\in\mathbb{Z}^2}
\] 
to describe the radial part as
\begin{equation}\label{Fdefn}
	\begin{split}
	F_{i,j}(s,\psi,\alpha) =\alpha&\sum_{i',j'} [(\bar{r}_{i',j'}(\alpha) + s_{i',j'})\cos(\bar{\theta}_{i',j'}(\alpha) + \psi_{i',j'} - \bar{\theta}_{i,j}(\alpha) - \psi_{i,j}) - (\bar{r}_{i,j}(\alpha)+ s_{i,j})\bigg] \\
	&+ (\bar{r}_{i,j}(\alpha) + s_{i,j})\lambda(\bar{r}_{i,j}(\alpha) + s_{i,j}),
	\end{split}
\end{equation}
so that we have $F(0,0,\alpha) = 0$ for all $\alpha \in [0,\alpha^*]$. Furthermore, we define the function 
\[
	G(s,\psi,\alpha) = \{G_{i,j}(s,\psi,\alpha)\}_{(i,j)\in\mathbb{Z}^2}
\] 
to describe the phase components as
\[
	G_{i,j}(s,\psi,\alpha) = \sum_{i',j'}\bigg[\bigg(\frac{\bar{r}_{i',j'}(\alpha) + s_{i',j'}}{\bar{r}_{i,j}(\alpha) + s_{i,j}}\bigg)\sin(\bar{\theta}_{i',j'} + \psi_{i',j'} - \bar{\theta}_{i,j} - \psi_{i,j})\bigg] + \omega_1(\bar{r}_{i,j}(\alpha) + s_{i,j},\alpha),
\] 
so that $G(0,0,\alpha) = 0$ for all $\alpha \in [0,\alpha^*]$. This therefore leads to the system which we will focus on throughout this work
\begin{equation} \label{FullODE}
	\begin{split}
		&\dot{s} = F(s,\psi,\alpha), \\
		&\dot{\psi} = \alpha G(s,\psi,\alpha).
	\end{split}
\end{equation}

Written in the form (\ref{FullODE}) it becomes easier to identify that when $0 < \alpha \ll 1$ our system resembles a fast-slow dynamical system, with the major caveat that we are working in infinite dimensions. Our approach to describing the behaviour of solutions to (\ref{FullODE}) will be motivated by the study of finite-dimensional fast-slow dynamical systems, but now the infinite-dimensionality of the problem requires one to be more careful since the notion of distance is dependent on the underlying infinite dimensional phase space in which the solutions exist in. The first major result of this work describes the persistence of an infinite-dimensional invariant manifold in (\ref{FullODE}), for which the proof is left to Section~\ref{sec:InvMan}.

\begin{thm}\label{thm:InvMan} 
	Assume Hypothesis~\ref{hyp:LambdaOmega} and \ref{hyp:PolarSoln}. Then, there exists $\alpha_0 \in (0,\alpha^*]$ such that the following is true: there exists a function 
	\[
		\sigma:\ell^1 \times [0,\alpha_0] \to \ell^1
	\] 
	such that the set 
	\[
		\{(s,\psi):\ s = \sigma(\psi,\alpha), \psi \in \ell^1\}
	\]	
	is an invariant manifold of system (\ref{FullODE}) for all $\alpha \in [0,\alpha_0]$. The function $\sigma$ satisfies the following properties:
	\begin{equation}\label{ManifoldProperties}
		\begin{split}
			\sigma(0,\alpha) &= 0, \\
			\sup_{\psi\in\ell^1} \|\sigma(\psi,\alpha)\|_\infty &\leq \sqrt{\alpha}, \\
			\|\sigma(\psi,\alpha) - \sigma(\tilde{\psi},\alpha)\|_p &\leq \sqrt{\alpha} Q_p(\psi - \tilde{\psi}),
		\end{split}
	\end{equation}
	for all $p \in [1,\infty]$, $\psi,\tilde{\psi} \in \ell^1$, and $\alpha \in [0,\alpha_0]$. This invariant manifold is asymptotically exponentially stable. That is, there exists $\delta^*,\beta > 0$ such that for all $\delta \in (0,\delta^*]$, if 
	\[
		\|s(0) - \sigma(\psi(0),\alpha)\|_1\leq \delta,
	\] 
	then 
	\[
		\|s(t)-\sigma(\psi(t),\alpha)\|_1 \leq 2\delta e^{-\beta t}
	\] 
	for all $t\geq 0$ and $\alpha \in [0,\alpha_0]$. 
\end{thm} 

Theorem~\ref{thm:InvMan} allows one to reduce the dynamics of (\ref{FullODE}) to the invariant manifold to understand the behaviour of the phase components, $\psi$. When put back into the single complex variable $z_{i,j}$, this invariant manifold represents an infinite dimensional invariant torus given by
\[
	\psi \mapsto \{[(\bar{r}_{i,j}(\alpha) + \sigma_{i,j}(\psi,\alpha)]e^{{\rm i}(\bar{\theta}_{i,j}(\alpha) + \psi)}]_{i,j}\}_{(i,j)\in\mathbb{Z}^2},	 
\]
where we write $\sigma(\psi) = \{\sigma_{i,j}(\psi)\}_{(i,j)\in\mathbb{Z}^2}$. Then, to extend Theorem~\ref{thm:InvMan} by examining the stability of solutions on the invariant manifold, we must first define the linear operator $L_\alpha$ acting on the sequences $x = \{x_{i,j}\}_{(i,j)\in\mathbb{Z}^2}$ by 
\begin{equation}\label{LinearPhase}
	[L_\alpha x]_{i,j} = \sum_{i',j'} \cos(\bar{\theta}_{i',j'}(\alpha) - \bar{\theta}_{i,j}(\alpha))(x_{i',j'} - x_{i,j}),	
\end{equation}
for all $(i,j)\in\mathbb{Z}^2$. From the fact that each index $(i,j)\in\mathbb{Z}^2$ has exactly four nearest-neighbours, it is a straightforward exercise to find that $L_\alpha:\ell^p \to \ell^p$ is a bounded linear operator for all $p \in [1,\infty]$ and $\alpha \in [0,\alpha^*]$, with uniformly bounded operator norm. Moreover, one may use the methods of \cite[Proposition~6.2]{Bramburger1} to show that $L_\alpha:\ell^\infty\to \ell^\infty$ is not a Fredholm operator, which implies that forward time exponential dichotomies of the $\ell^\infty$ norm of the solutions to the linear ordinary differential equation
\begin{equation}\label{GraphODE}
	\dot{x} = L_{\alpha}x
\end{equation}
cannot be obtained for arbitrary initial conditions $x^0 \in \ell^\infty$. When every $\cos(\bar{\theta}_{i',j'}(\alpha) - \bar{\theta}_{i,j}(\alpha))$ term is nonnegative the result \cite[Theorem~2.1]{Bauer} then extends this result to give that the spectrum of $L_\alpha:\ell^p \to \ell^p$ for every $p \in [1,\infty]$ must therefore have nontrivial intersection with the imaginary axis of the complex plane, implying that forward time exponential dichotomies of the solutions to (\ref{GraphODE}) cannot be obtained for initial conditions in any of the $\ell^p$ sequence spaces. The following hypothesis gives that although exponential dichotomies cannot be obtained, we still assume that solutions to (\ref{GraphODE}) exhibit an algebraic decay in the $\ell^p$ norms.

\begin{hyp}\label{hyp:LinearPhase} 
	Let $L_\alpha$ act on the sequences indexed by $\mathbb{Z}^2$ as in (\ref{LinearPhase}). There exists constants $C_L,\eta > 0$ such that for all $x^0 \in \ell^1$, $\alpha \in [0,\alpha^*]$ and $t \geq 0$ we have 
	\begin{equation}\label{LAlphaDecay}
		\begin{split}
			&\|e^{L_\alpha t} x^0\|_p \leq C_L(1 + t)^{-1 + \frac{1}{p}}\|x^0\|_1, \\
			&Q_p(e^{L_\alpha t} x^0) \leq C_L(1 + t)^{-2 + \frac{1}{p}}\|x^0\|_1,
		\end{split}
	\end{equation}
	where $e^{L_\alpha t}$ is the semi-group with infinitesimal generator given by $L_\alpha$.
\end{hyp} 

The decay rates stated in Hypothesis~\ref{hyp:LinearPhase} are not arbitrary and in Section~\ref{sec:Applications} we show that they hold for a wide range of linear operators of the form $L_\alpha$. Moreover, it was shown in \cite{Bramburger3} that when 
\[
	\cos(\bar{\theta}_{i',j'}(\alpha) - \bar{\theta}_{i,j}(\alpha)) \geq 0
\] 
for all $(i,j)\in\mathbb{Z}^2$, the linear operator (\ref{LinearPhase}) can be interpreted as a graph Laplacian operator associated to an infinite graph with vertex set lying in one-to-one correspondence with the indices of the lattice. Then, it is shown that the geometry of this underlying graph can be used to obtain similar decay rates to those in (\ref{LAlphaDecay}), but that the $\ell^p$ and $Q_p$ decay rates differ by a fixed constant $\eta > 0$. Determining the value $\eta$ comes from \cite[Theorem~5.4.12]{SaloffCoste}, which is quite technical since the investigation is conducted in complete generality. For now, the reader should note that the $Q_p$ semi-norms generalize the norms of the gradient of a function and hence one would expect that $\eta = 1$ as it is in the continuous spatial setting. We will see that this is exactly the case for a number of applications in Section~\ref{sec:Applications}. 

\begin{hyp}\label{hyp:pStar} 
	There exists a constant $C_\mathrm{sol} > 0$ and $p^* \geq 1$ such that for all $\alpha \in [0,\alpha^*]$ we have 
	\begin{equation}\label{pStar}
		\sum_{(i,j)\in\mathbb{Z}^2}\sum_{i',j'}\bigg|\frac{\bar{r}_{i',j'}(\alpha)}{\bar{r}_{i,j}(\alpha)} - 1\bigg|^{p^*} \leq C_\mathrm{sol}.
	\end{equation}
	We will denote $q^* \geq 1$ to be the H\"older conjugate of $p^*$ in that $\frac{1}{p^*} + \frac{1}{q^*} = 1$. 
\end{hyp} 

Hypothesis~\ref{hyp:pStar} is a localization statement for the radial variables. It states that the nearest-neighbours should asymptotically become constant at a fast enough rate. That rate is measured by the constant $p^* \geq 1$. Again in Section~\ref{sec:Applications} we will see that in many applications this hypothesis can be met easily since many steady-state solutions satisfying Hypothesis~\ref{hyp:PolarSoln} have $\bar{r}_{i,j}(\alpha)$ being independent of $(i,j)$ for each $\alpha$, which allows one to take $p^* = 1$. The following theorem shows exactly how the value of $p^*$ enters our analysis since we have that the nonlinear stability of the phase components does not fully mimic the linear decay rates stated in (\ref{LAlphaDecay}). These hypotheses now lead to the major result of this manuscript which details the nonlinear stability of these periodic solutions of the system (\ref{LambdaOmegaLDS}).

\begin{thm}\label{thm:PhaseDecay} 
	Assume Hypotheses~\ref{hyp:LambdaOmega}-\ref{hyp:pStar} and that $\omega_1(R,\alpha) = 0$ for all $R,\alpha \geq 0$.. Then there exists $\alpha_1 \in (0,\alpha^*]$ and a constant $\varepsilon^* > 0$ such that for all $\varepsilon \in [0,\varepsilon^*]$ and $s^0,\psi^0\in\ell^1$ with the property that
	\begin{equation}
		\|s^0 - \sigma(\psi^0,\alpha)\|_1\leq \varepsilon, \quad \|\psi^0\|_1 \leq \varepsilon,	
	\end{equation}
	there exists a unique solution of (\ref{FullODE}) for all $t\geq 0$ and $\alpha \in [0,\alpha_1]$, denoted $(s(t),\psi(t))$, satisfying the following properties:
	\begin{enumerate}
		\item $s(0) = s^0$ and $\psi(0) = \psi^0$.
		\item $(s(t),\psi(t)) \in \ell^1\times \ell^1$ for all $t \geq 0$.
		\item There exists a $\beta, C_\psi > 0$ such that 
			\[
				\|s(t)-\sigma(\psi(t),\alpha)\|_1 \leq 2\varepsilon e^{-\beta t},
			\]
			and
			\[
				\begin{split}
					\|\psi(t)\|_p &\leq \varepsilon C_\psi(1 + \alpha t)^{-1 + \frac{1}{p}}, \\
					Q_p(\psi(t)) &\leq \varepsilon C_\psi(1 + \alpha t)^{-\min\{2 - \frac{1}{p},2 - \frac{1}{q^*}\}},
				\end{split}
			\]
			for all $t\geq 0$, $\alpha \in [0,\alpha_1]$, and $p \in [1,\infty]$.
	\end{enumerate}	
\end{thm}

\begin{rmk}
	Theorem~\ref{thm:PhaseDecay} extends Theorem~\ref{thm:InvMan} by saying that if we choose an initial condition $(s^0,\psi^0) \in \ell^1\times\ell^1$ sufficiently close to the steady-state $(s,\psi) = (0,0)$ in the system (\ref{FullODE}), we obtain exponential decay onto the invariant manifold, along with algebraic decay of the phase component $\psi(t)$ with respect to the $\ell^p$ norms. Moreover, if we write
	\[
		\rho(t) := s(t) - \sigma(\psi(t),\alpha), 
	\]
	then $\rho(t)$ describes the deviation from the invariant manifold which is decaying exponentially. Theorem~\ref{thm:InvMan} gives that 
	\[
		\|\sigma(\psi(t),\alpha)\|_p \leq \sqrt{\alpha} Q_p(\psi(t)) \leq \sqrt{\alpha}\varepsilon C_\psi (1 + \alpha t)^{-\min\{2 - \frac{1}{p},2 - \frac{1}{q^*}\}},	
	\]
	which shows that the component of $s(t)$ that belongs to the invariant manifold decays at a significantly faster rate than $\psi(t)$ in each $\ell^p$ norm since we necessarily have $q^* > 1$. The reader should notice that the decay rates of $Q_p(\psi(t))$ in Theorem~\ref{thm:PhaseDecay} are all faster than their corresponding $\ell^p$ decay rate and are determined by the value $p^*$ coming from Hypothesis~\ref{hyp:pStar}. This is a consequence of our analysis, which becomes apparent in the proof of Lemma~\ref{lem:QPhaseBnd}. Finally, the reader should note the different timescales in which the decay is taking place, where all algebraic decay is with respect to the slow-time variable $\alpha t$. 
\end{rmk}

The proof of Theorem~\ref{thm:InvMan} is left to Section~\ref{sec:InvMan}, which is broken into two subsections which deal with existence and stability of the invariant manifold respectively. Then, the proof of Theorem~\ref{thm:PhaseDecay} that remains after proving Theorem~\ref{thm:InvMan} is left to Section~\ref{sec:Thm2Proof}. We first provide a pair of useful integral inequalities in Subsection~\ref{subsec:IntegralBounds} which are used throughout Section~\ref{sec:Thm2Proof}. Then, in Subsection~\ref{subsec:Linearization} we study the decay properties of a semigroup closely related to the semi-group $e^{L_\alpha t}$ given in Hypothesis~\ref{hyp:LinearPhase}. These linear estimates give way to the estimates on the nonlinear terms of the function $G$ in Subsection~\ref{subsec:Estimates}. Having then obtained estimates for both the linear and nonlinear terms associated to the function $G$, we then provide the proof of Theorem~\ref{thm:PhaseDecay} in Subsection~\ref{subsec:PhaseStability}. Prior to proving Theorems~\ref{thm:InvMan} and \ref{thm:PhaseDecay} though, we dedicate the following section to exploring some applications of these results.

\section{Applications of the Main Results}\label{sec:Applications} 

In this section we provide a number of important applications of our main results, as well as discuss the possibility of it applying to demonstrate nonlinear stability of a rotating wave. Throughout this section we will assume that $\omega$ is a function of $\alpha$ only so that we may apply both Theorem~\ref{thm:InvMan} and Theorem~\ref{thm:PhaseDecay}. Prior to providing these applications, we present the following result whose proof is presented in Subsection~\ref{subsec:Proof}.

\begin{lem}\label{lem:Hyp3} 
	Let $d_1,d_2 > 0$ be arbitrary and consider the linear operator $L$ acting on the sequences $x = \{x_{i,j}\}_{(i,j)\in\mathbb{Z}^2}$ by
	\[
		[Lx]_{i,j} = d_1(x_{i+1,j} + x_{i-1,j} - 2x_{i,j}) + d_2(x_{i,j+1} + x_{i,j-1} - 2x_{i,j})
	\]
	for all $(i,j)\in\mathbb{Z}^2$. Then $L$ satisfies Hypothesis~\ref{hyp:LinearPhase}.
\end{lem}

Lemma~\ref{lem:Hyp3} will be crucial to our applications in the following subsections, which are presented over the following subsections. We begin with the completely synchronous solution, then move to a doubly spatially periodic pattern, and then to periodic traveling waves. Upon providing these three applications of our results, we then discuss how they should be applicable to rotating waves solutions of (\ref{LambdaOmegaLDS}). The only hindrance to applying these results to rotating waves is confirming Hypotheses~\ref{hyp:pStar}. Despite this, we will use numerical simulations to demonstrate that there should be a $p^*$ for which Hypothesis~\ref{hyp:pStar} holds.

\subsection{The Trivial Solution}

As discussed in Section~\ref{sec:Main}, there exists a trivial solution to (\ref{FullPolarLattice}) given by $\bar{r}_{i,j}(\alpha) = a$ and $\bar{\theta}_{i,j}(\alpha) = 0$ for all $(i,j) \in \mathbb{Z}^2$ and $\alpha \geq 0$. This solution leads to a synchronous periodic solution to (\ref{LambdaOmegaLDS}) of the form
\[
	z_{i,j}(t) = a\mathrm{e}^{{\rm i}\omega_0(\alpha)t}
\] 
for all $(i,j)\in \mathbb{Z}^2$. It was already noted that this trivial solution satisfies Hypothesis~\ref{hyp:PolarSoln}, and the associated operator $L_\alpha$ is of the type described in Lemma~\ref{lem:Hyp3} with $d_1 = d_2 = 1$. Finally, since all $\bar{r}_{i,j}(\alpha)$ are identical, it follows that Hypothesis~\ref{hyp:pStar} holds for every $p^* \geq 1$. Hence, Hypotheses~\ref{hyp:LambdaOmega}-\ref{hyp:pStar} hold, and therefore Theorem~\ref{thm:PhaseDecay} can be applied to this trivial solution to demonstrate the stability of such a synchronous periodic solution to (\ref{LambdaOmegaLDS}). Moreover, the fact that we may take $p^* = 1$ implies that the linear decay rates of Hypothesis~\ref{hyp:LinearPhase} are inherited by the fully nonlinear system.

\subsection{Doubly Spatially Periodic Solutions}

Let us begin by fixing two integers $N,M \geq 5$ and consider the phases given by
\[
	\bar{\theta}_{i,j}(\alpha) = \frac{2\pi [i]_{N}}{N} +  \frac{2\pi [j]_{M}}{M},
\]
for all $\alpha \geq 0$, where we have used the notation $[n]_N = n \pmod{N}$. These choices of $\bar{\theta}_{i,j}$ give a sequence indexed by the lattice which is periodic in both the horizontal and the vertical directions. Taking $r_{i,j}(\alpha) = r(\alpha)$ for all $(i,j)\in\mathbb{Z}^2$ with $\bar{\theta}$ as prescribed above implies that $r(\alpha)$ must satisfy
\[
	0 = \alpha\bigg[2\cos\bigg(\frac{2\pi}{N}\bigg) + 2\cos\bigg(\frac{2\pi}{M}\bigg) - 4\bigg]r(\alpha) + r(\alpha)\lambda(r(\alpha)). 
\]
Hypothesis~\ref{hyp:LambdaOmega} implies that $\lambda'(a) < 0$ and hence we may apply the implicit function theorem to obtain such an $r(\alpha)$ which is smooth in $\alpha \geq 0$ (at least for $\alpha$ small), satisfies $r(0) = a$, and the assumptions of Hypothesis~\ref{hyp:PolarSoln}. 

Now, these choices of $(\bar{r}_{i,j}(\alpha),\bar{\theta}_{i,j}(\alpha))$ give
\[
	\sum_{i',j'} \frac{r_{i',j'}}{r_{i,j}}\sin(\theta_{i',j'} - \theta_{i,j}) = \sin\bigg(\frac{2\pi}{N}\bigg) - \sin\bigg(\frac{2\pi}{N}\bigg) + \sin\bigg(\frac{2\pi}{M}\bigg) - \sin\bigg(\frac{2\pi}{M}\bigg) = 0,	
\]
since we have assumed $\omega_1(R,\alpha) = 0$ identically. Hence, our choice of $\bar{\theta}$ and $\bar{r}(\alpha)$ give a periodic solution of (\ref{LambdaOmegaLDS}) which satisfies Hypothesis~\ref{hyp:PolarSoln}. Furthermore, the associated operator $L_\alpha$ is of the type described in Lemma~\ref{lem:Hyp3} with 
\[	
	d_1 = \cos\bigg(\frac{2\pi}{N}\bigg), \quad d_2 = \cos\bigg(\frac{2\pi}{M}\bigg)
\] 
which are both positive since $N,M \geq 5$. Again we have that Hypothesis~\ref{hyp:pStar} is trivially satisfied for all $p^* \geq 1$, and hence, we may apply the results of Theorem~\ref{thm:PhaseDecay} to find that these doubly spatially periodic solutions are nonlinearly asymptotically stable for every $N,M \geq 5$.

\subsection{Periodic Traveling Waves}

This example is very similar to the previous example. Take $N \geq 5$ and consider 
\[
	\bar{\theta}_{i,j}(\alpha) = \frac{2\pi [i]_{N}}{N},
\]
for all $\alpha \geq 0$, which comes as the formal limit of the doubly spatially periodic solutions with $M \to \infty$. It is easy to see that these phase values give an element indexed by the two-dimensional lattice which is periodic in the horizontal direction and identical in the vertical direction, representing a periodic traveling wave moving in the horizontal direction. We again can find the $\bar{r}_{i,j}(\alpha)$ to be identical, thus giving periodic solutions of (\ref{LambdaOmegaLDS}) which satisfy Hypothesis~\ref{hyp:PolarSoln}. Moreover, the results of Lemma~\ref{lem:Hyp3} apply with 
\[
	d_1 = \cos\bigg(\frac{2\pi}{N}\bigg), \quad d_2 = 1
\] 
which are both positive since $N \geq 5$. Hence, we conclude that such periodic traveling wave patterns are locally asymptotically stable for each $N \geq 5$.

\subsection{Rotating Waves}

As previously remarked, it was shown in \cite{Bramburger2} that system (\ref{LambdaOmegaLDS}) possesses a particular nontrivial solution which resembles a rotating wave from the continuous spatial context. This solution was shown to satisfy Hypothesis~\ref{hyp:PolarSoln} for a sufficiently small $\alpha^* > 0$. Throughout this section we simply denote this rotating wave solution by $\{(\bar{r}_{i,j}(\alpha),\bar{\theta}_{i,j}(\alpha))\}_{(i,j)\in\mathbb{Z}^2}$ for convenience. In the present discrete spatial context a rotating wave solution is identified by the discrete rotational identity
\begin{equation}\label{RotIdentity}
	z_{j,1-i}(t;\alpha) = \mathrm{e}^{{\rm i}\frac{\pi}{2}}\cdot z_{i,j}(t;\alpha),
\end{equation}
for every $(i,j)\in\mathbb{Z}^2$ and $\alpha \in [0,\alpha^*]$, upon returning back to the single complex variable $z_{i,j}$ via the ansatz (\ref{PolarAnsatz}). The meaning of the identity (\ref{RotIdentity}) is that a rotation of the entire lattice clockwise through an angle of $\pi/2$ about a theoretical centre cell at $i = j =\frac{1}{2}$ simply leads to a phase advance of a quarter period. Indeed, the ansatz (\ref{PolarAnsatz}) gives
\[
	\begin{split}
	\mathrm{e}^{{\rm i}\frac{\pi}{2}}\cdot z_{i,j}(t;\alpha) &= \bar{r}_{i,j}(\alpha)\mathrm{e}^{{\rm i}(\omega_0(\alpha) t + \bar{\theta}_{i,j}(\alpha) + \frac{\pi}{2})} \\
	&= \bar{r}_{i,j}(\alpha)\mathrm{e}^{{\rm i}(\omega_0(\alpha) (t + \frac{\pi}{2\omega_0(\alpha)}) + \bar{\theta}_{i,j}(\alpha))} \\
	&= z_{i,j}(t + T(\alpha)/4;\alpha),		
	\end{split}
\]
where $T(\alpha):= 2\pi/\omega_0(\alpha)$ is the period of the periodic solution. Here the theoretical centre cell at $i = j =\frac{1}{2}$ acts as the centre of rotating for the rotating wave. 

For the purpose of visualization, we provide an example of a rotating wave solution to the system
\begin{equation} \label{LambdaOmegaLDS2} 
	\dot{z}_{i,j} = \alpha\sum_{i',j'} (z_{i'j'} -z_{i,j}) + z_{i,j}(1 + {\rm i}\omega_0(\alpha) - |z_{i,j}|^2), \quad (i,j) \in \mathbb{Z}^2,	
\end{equation} 
obtained on a $100\times 100$ lattice with Neumann boundary conditions. Here we can see that we have $\lambda(R) = 1 - R^2$, and the rotating wave solution is presented in Figure~\ref{fig:RotWave}. One finds that the values of the phases, $\{\bar{\theta}_{i,j}(\alpha)\}_{(i,j)\in\mathbb{Z}^2}$, around each concentric ring about the centre four cell ring at the indices $(i,j) = (0,0),(0,1),(1,0),(1,1)$ increase from $-\pi$ up to $\pi$ monotonically. This was proven for $\alpha = 0$ in \cite{Bramburger1}, and numerical investigations on the finite lattice lead one to conjecture that this holds for all $\alpha > 0$ for which the solution exists. Most importantly, we have confirmed that Hypotheses~\ref{hyp:LambdaOmega} and \ref{hyp:PolarSoln} do indeed hold for system (\ref{LambdaOmegaLDS2}) and this rotating wave solution, giving that the results of Theorem~\ref{thm:InvMan} hold for such a rotating wave solution.

\begin{figure} 
	\centering
		\includegraphics[width = 0.45\textwidth]{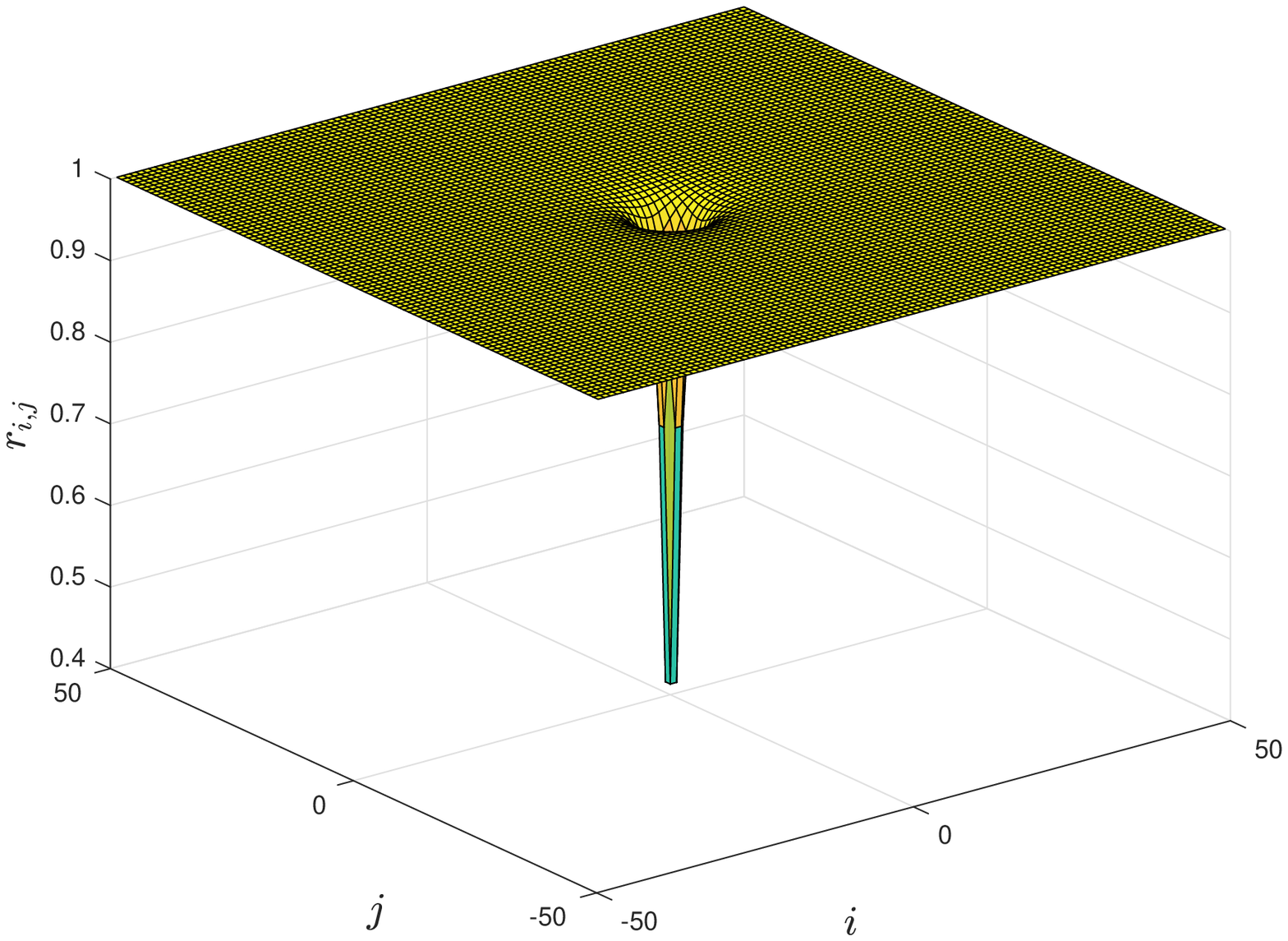}
		\includegraphics[width = 0.45\textwidth]{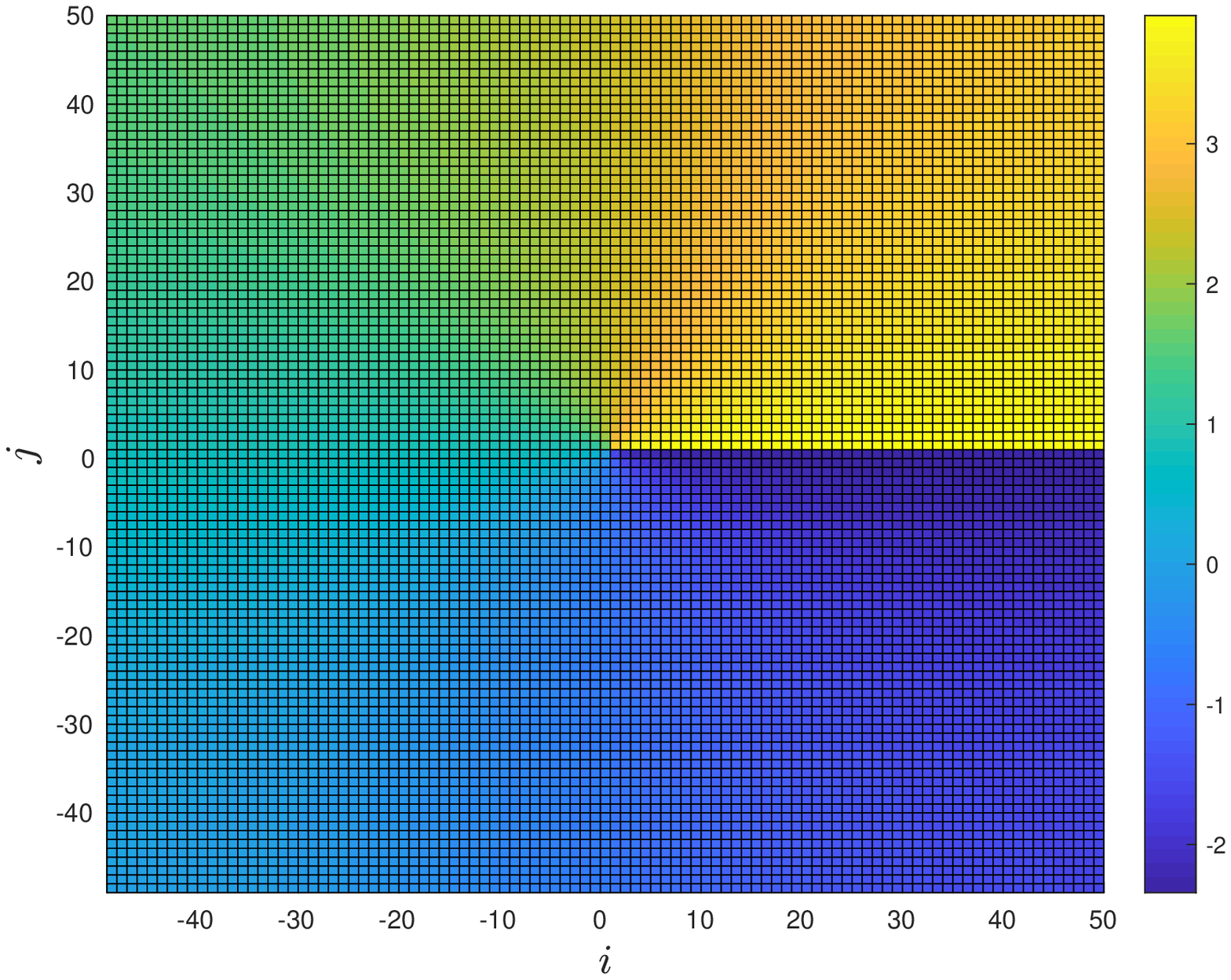}
		\caption{The radial and phase components of a rotating wave solution to (\ref{LambdaOmegaLDS2}) with $\lambda(R) = 1 - R^2$ and $\alpha = 1$ on a finite $100\times 100$ lattice with Neumann boundary conditions. The four cell ring given by the indices $(i,j) = (0,0),(0,1),(1,0),(1,1)$ acts as the centre of the rotating wave, and on the left one sees that the radial components show very little deviation from the unique positive root at $R = 1$ of $\lambda(R) = (1 - R^2)$ as one moves out from this centre of the rotating wave. Around each concentric ring about the centre four cell ring, the phase components increase from $-\pi$ up to $\pi$ monotonically, which can be observed in the contour plot of the phase components on the right.}
		\label{fig:RotWave}	
\end{figure} 

We now turn to the problem of verifying Hypothesis~\ref{hyp:LinearPhase}. The decay rates on the $\ell^p$ norms required by (\ref{LAlphaDecay}) were shown to be true for $\alpha = 0$ in \cite[Section~6.2]{Bramburger3}, and these arguments can be replicated to show that the decay rates (\ref{LAlphaDecay}) hold for sufficiently small $\alpha \geq 0$. To verify this claim, we begin by noting that the work in \cite{Bramburger2} gives that the coupling between any two of the four centre cells is exactly $\pi/2$ since we have 
\[
	\begin{split}
		\bar{\theta}_{1,1}(\alpha) &= 0, \\
		\bar{\theta}_{0,1}(\alpha) &= \frac{\pi}{2}, \\
		\bar{\theta}_{0,0}(\alpha) &= \pi, \\
		\bar{\theta}_{1,0}(\alpha) &= \frac{3\pi}{2},
	\end{split}
\] 
for all $\alpha > 0$ for which the solution exists. At $\alpha = 0$ we have that all other nearest-neighbour interactions are such that $|\bar{\theta}_{i',j'}(0) - \bar{\theta}_{i,j}(0)| < \frac{\pi}{2}$, thus giving that continuity with respect to $\alpha$ will ensure that for sufficiently small $\alpha > 0$ we have 
\[
	\cos(\bar{\theta}_{i',j'}(\alpha) - \bar{\theta}_{i,j}(\alpha)) > 0,
\]  
for all $(i,j),(i',j') \notin \{(0,0), (0,1), (1,0), (1,1)\}$. We then consider a graph with vertices given by the indices of the lattice $\mathbb{Z}^2$ together with an edge set which connects nearest-neighbours if, and only if, the quantity $\cos(\bar{\theta}_{i',j'}(\alpha) - \bar{\theta}_{i,j}(\alpha))$ is positive, and no other edges present in the graph. Visually, this graph is simply the standard integer lattice with edges connecting all nearest-neighbours less those between the four centre cells, as is shown in Figure~\ref{fig:Graph}. Our arguments above imply that for sufficiently small $\alpha \geq 0$ the geometry of this graph remains unchanged in that no new edges are created or destroyed. This understanding of the graph geometry for sufficiently small $\alpha \geq 0$ then allows to provide the following lemma which confirms Hypothesis~\ref{hyp:LinearPhase} for this application. 

\begin{figure} 
	\centering
	\includegraphics[height=0.3\textwidth]{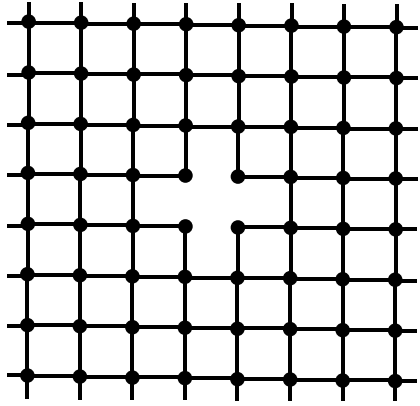}
	\caption{The graph associated to the linear operator $L_\alpha$, coming from the phase components $\{\bar{\theta}_{i,j}(\alpha)\}_{(i,j)\in\mathbb{Z}^2}$. The vertex set lies in one-to-one correspondence with the index set $\mathbb{Z}^2$ and an edge connects two vertices if, and only if, the vertices are nearest-neighbours and $\cos(\bar{\theta}_{i',j'}(\alpha) - \bar{\theta}_{i,j}(\alpha)) > 0$.}
	\label{fig:Graph}
\end{figure} 

\begin{lem}\label{lem:RotHyp} 
	The linear operator $L_\alpha$ associated to the rotating wave satisfies Hypothesis~\ref{hyp:LinearPhase}.
\end{lem}

\begin{proof}
	This proof is exactly the same as the proof of Lemma~\ref{lem:Hyp3} since it was shown in \cite{Bramburger3} that the conditions to obtain (\ref{RandomWalk}) are also satisfied for this rotating wave. 
\end{proof} 

We now turn to confirming Hypothesis~\ref{hyp:pStar}. In the continuous spatial setting of rotating wave solutions to Lambda-Omega systems explored in \cite{Cohen} it was shown that as one moves away from the centre of the rotating wave the radial components converge to $a$, the root of the function $\lambda$. Hence, this would lead one to conjecture that the same happens in the lattice case as well and this can easily be seen in the numerical solution presented in Figure~\ref{fig:RotWave}. Unfortunately a complete analytic proof of this remains elusive, and therefore we restrict ourselves to numerical investigations in an effort to at least conjecture that this hypothesis does indeed hold for this rotating wave solution. We define the following quantities 
\begin{equation} \label{MpFns}
	M_p^r(\alpha) := \sum_{(i,j)\in\mathbb{Z}^2}|r_{i,j}(\alpha) - 1|^{p}, \\
\end{equation}
for $p \geq 1$. The function $M_p^r$ allows one to measure the deviation from the background state $a$ of the rotating wave solution. Most importantly, we have
\[
	\bigg|\frac{r_{i',j'}(\alpha)}{r_{i,j}(\alpha)} - 1\bigg| \leq \frac{1}{r_{i,j}(\alpha)}|r_{i',j'}(\alpha)-r_{i,j}(\alpha)| \leq \frac{2}{a}(|r_{i',j'}(\alpha) - a| + |r_{i,j}(\alpha)-a|)
\] 
since we have assumed that $r_{i,j}(\alpha) \leq \frac{a}{2}$ for all $(i,j)\in\mathbb{Z}^2$. Hence, there exists a constant $C > 0$ such that 
\[
	\sum_{(i,j)\in\mathbb{Z}^2}\sum_{i',j'}\bigg|\frac{\bar{r}_{i',j'}(\alpha)}{\bar{r}_{i,j}(\alpha)} - 1\bigg|^{p} \leq CM_p^r(\alpha),	
\]
for all $p \geq 1$ and determining a bound on $M_p^r(\alpha)$ for some $p \geq 1$ necessarily implies that Hypothesis~\ref{hyp:pStar} holds for this solution.  

Using MATLAB we are able to generate the rotating wave solution on a finite $N\times N$ lattice with Neumann boundary conditions, for increasing $N$, with $\lambda(R) = 1 - R^2$. This allows one to conjecture that there exists $p \geq 1$ such that $M_p^r(\alpha) < \infty$ for sufficiently small $\alpha \geq 0$, and hence provide heuristic evidence that Hypothesis~\ref{hyp:pStar} holds, at least for this choice of $\lambda$. In Figure~\ref{fig:Hyp4} we provide plots of the quantities $M_1^r(\alpha)$ and $M_5^r(\alpha)$ for $\alpha = 0.1,0.5,1.0$ for the rotating wave solution simulated on a lattice of size $N \times N$, with $N$ increasing by ten from $N = 10$ to $N = 200$. From Figure~\ref{fig:Hyp4} one infers that the quantity $M_1^r(\alpha)$ should not be finite in the infinite lattice limit ($N \to \infty$), and hence we cannot guarantee that $p^*$ can be taken to be $1$ in Hypothesis~\ref{hyp:pStar}. If $p^* > 1$ then we would find that the decay of the $Q_p$ semi-norms of the solutions given in Theorem~\ref{thm:PhaseDecay} are not asymptotically equivalent to the estimates in Hypothesis~\ref{hyp:LinearPhase} for large $p > 1$.  

\begin{figure} 
	\centering
	\includegraphics[height=0.37\textwidth]{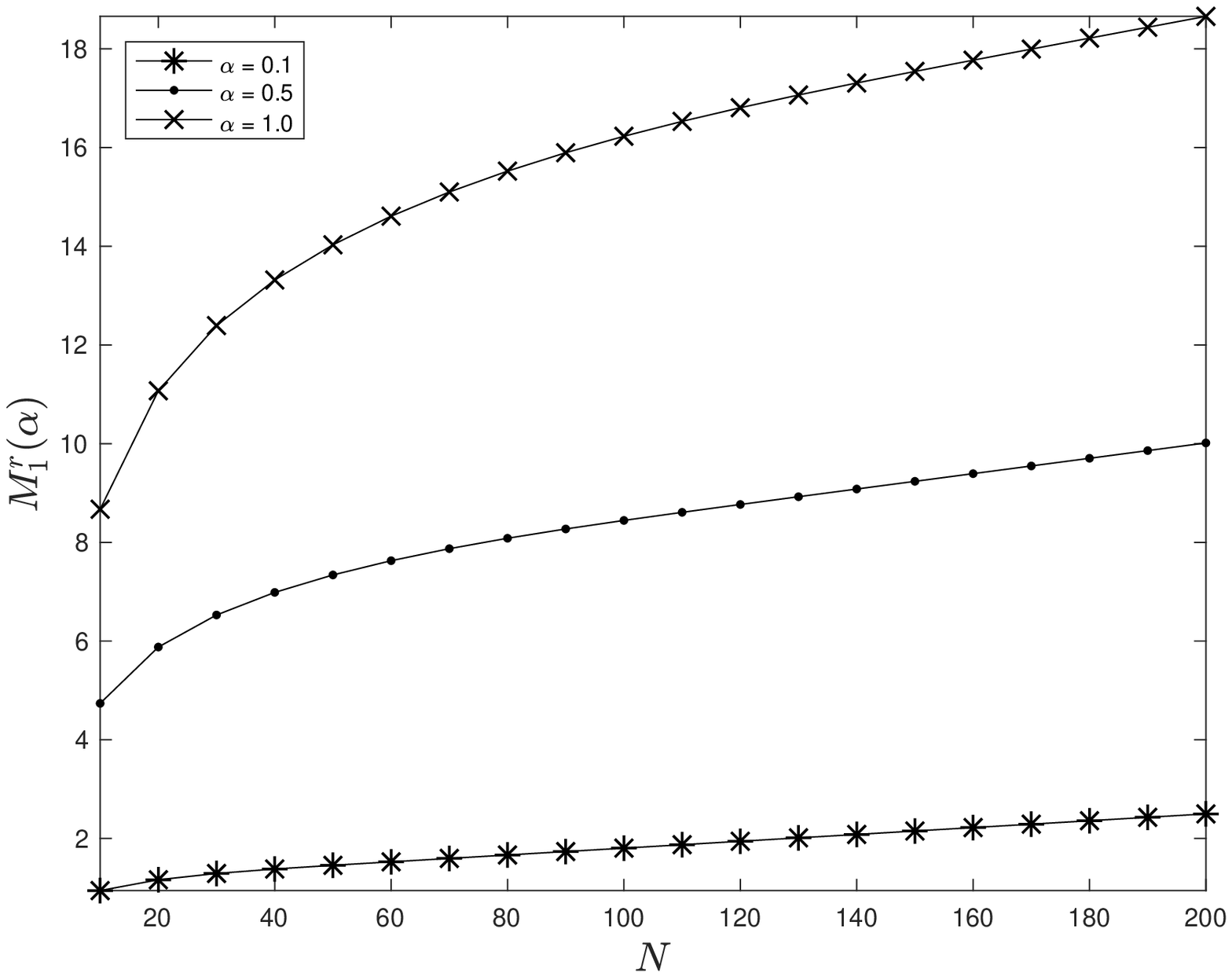}
	\includegraphics[height=0.37\textwidth]{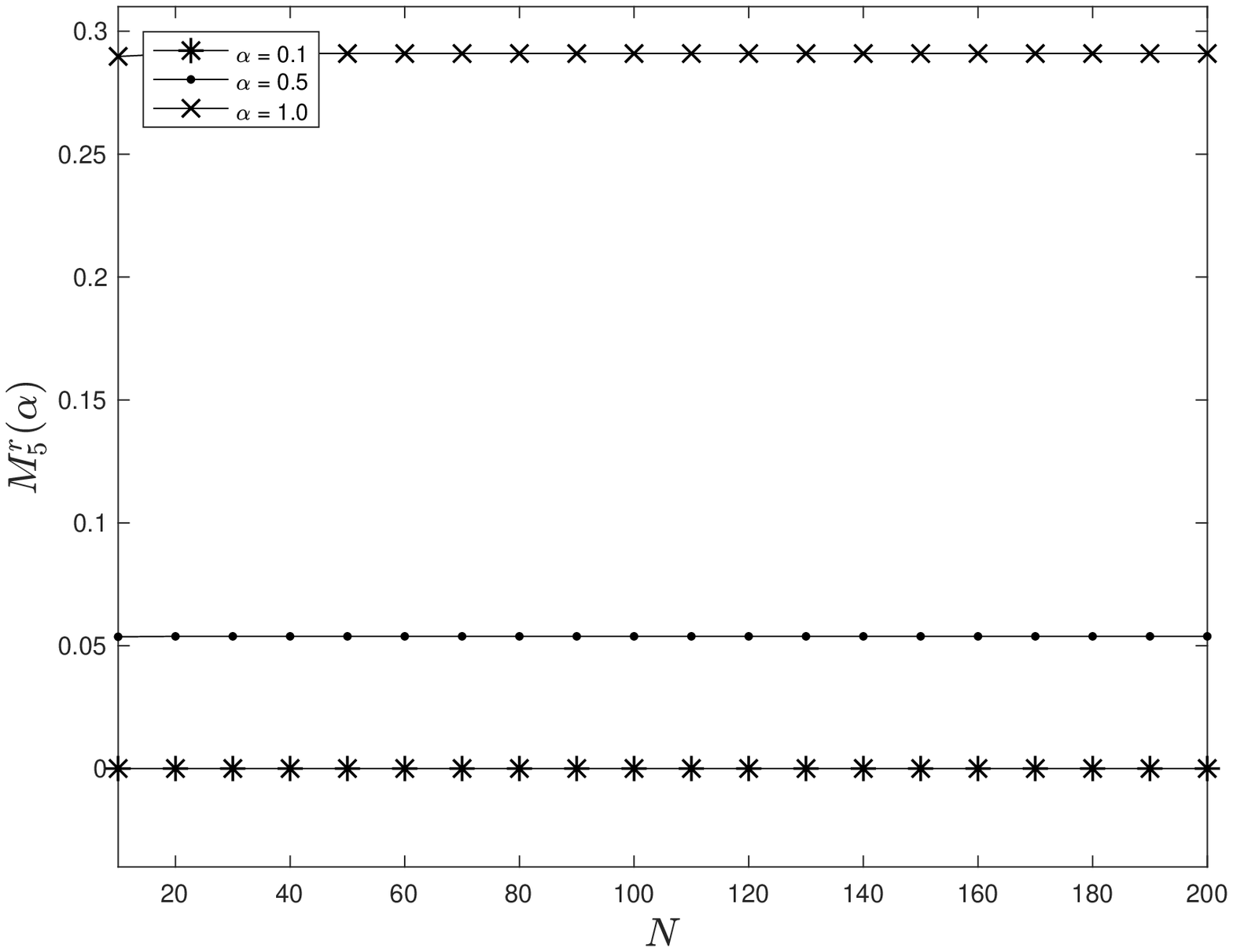} 
	\caption{Provided are plots of the quantities (left) $M_1^r(\alpha)$ and (right) $M_5^r(\alpha)$, defined in (\ref{MpFns}), for $\alpha = 0.1,0.5,1.0$ for the rotating wave solution simulated on a lattice of size $N \times N$. From these plots one conjectures that as $N \to \infty$ the quantity $M_1^r(\alpha)$ is unbounded and that $M_5^r(\alpha)$ remains bounded.}
	\label{fig:Hyp4}
\end{figure} 

Since we have conjectured that $M_1^r(\alpha)$ is unbounded, the right image in Figure~\ref{fig:Hyp4} provides an analogous plot for $M_5^r(\alpha)$. The choice for providing approximations of $M_p^r(\alpha)$ with $p =5$ is simply due to the fact that we have chosen $p > 1$ sufficiently large to see that the plots appear to level off at relatively small lattice sizes. Figure~\ref{fig:Hyp4} leads one to conjecture that upon moving to the infinite lattice limit both $M_5^r(\alpha)$ is finite, at least for the simulated values $\alpha = 0.1,0.5,1.0$. Hence, we see that these numerical investigations lead one to conjecture that Hypothesis~\ref{hyp:pStar} is true for some $p^* \in (1,5]$, but this would still require a full analytical treatment.

We conclude this section with a heuristic argument that $M_p^r(\alpha)$ should in fact be finite for sufficiently small $\alpha \geq 0$ and all $p > 1$. In \cite{ErmentroutSpiral2} the author heuristically regards a rotating wave solution in this discrete spatial context as a collection of nested rings, so that a $2N\times 2N$ array is comprised of $N$ concentric rings. The inner core is comprised of four elements (representing the centre of the rotating wave described above), the next wraps around this inner core and has length twelve, and so on. This analogy easily extends to the case when $N \to \infty$, where we now have a countable infinity of nested rings concentrically wrapped around each other, starting with the inner core of the centre four cells. Each concentric ring has $l_n = 8n-4$ elements for $n \in\{ 1,2,3,\dots\}$. When these concentric rings are uncoupled from one another we can find a rotating wave solution such that the difference between two consecutive phases in the ring is exactly $\pm2\pi/l_n$. 

In \cite{ErmentroutSpiral2} the author argues that when the concentric rings are coupled together, the solution should be similar. Indeed, in numerical simulations one can observe that as one moves away from the centre of rotation, the solution about concentric rings appear to converge to a solution for which the difference between two consecutive phases in the ring is exactly $\pm 2\pi/l_n$. If this were true, then we would have 
\begin{equation}\label{ThetaAsymptotics}
	|\bar{\theta}_{i',j'}(\alpha) - \bar{\theta}_{i,j}(\alpha)| = \mathcal{O}\bigg(\frac{1}{n}\bigg),
\end{equation}  
for all four nearest-neighbours $(i',j')$, assuming that $\bar{\theta}_{i,j}(\alpha)$ belongs to the ring of length $l_n$. Then using (\ref{FullPolarLattice}) we can expand $\bar{r}_{i,j}(\alpha)$ as a asymptotic series in $\alpha$ to see that 
\[
	\bar{r}_{i,j}(\alpha) = a + \frac{\alpha}{\lambda'(a)}\sum_{i',j'}[1 - \cos(\bar{\theta}_{i',j'}(0) - \bar{\theta}_{i,j}(0))] + \mathcal{O}(\alpha^2),
\]
for $\alpha \geq 0$ sufficiently small. The asymptotic relation (\ref{ThetaAsymptotics}) implies that 
\[
	\bar{r}_{i,j}(\alpha) = a + \mathcal{O}\bigg(\frac{\alpha}{n^2}\bigg),	
\] 
assuming that $\bar{r}_{i,j}(\alpha)$ belongs to the ring of length $l_n$. Therefore, since $l_n = \mathcal{O}(n)$, we would have that 
\[
	M_p^r(\alpha) \sim \alpha^p \sum_{n = 1}^\infty \frac{1}{n^{2p - 1}} 
\]
since we first sum over all $\bar{r}_{i,j}(\alpha)$ belonging to each ring $l_n$, giving the $1/n^{2p}$ term, and then over all rings indexed by $n \geq 1$. Hence, this heuristic argument implies that this rotating wave solution should satisfy (\ref{pStar}) for every $p^* > 1$. Although we do not have a proof of this fact, it appears to be consistent with the current state of the literature on rotating wave solutions in such a discrete spatial framework.

\subsection{Proof of Lemma~\ref{lem:Hyp3}}\label{subsec:Proof}

In this section we provide the proof of Lemma~\ref{lem:Hyp3}. Throughout this proof we will consider the infinite-dimensional linear ordinary differential equation 
\[
	\dot{x} = Lx,
\]
where $L$ is the linear operator acting on the sequences $x = \{x_{i,j}\}_{(i,j)\in\mathbb{Z}^2}$ by
\[
	[Lx]_{i,j} = d_1(x_{i+1,j} + x_{i-1,j} - 2x_{i,j}) + d_2(x_{i,j+1} + x_{i,j-1} - 2x_{i,j})
\]
for all $(i,j)\in\mathbb{Z}^2$. Throughout this proof we take $d_1,d_2 > 0$ to be fixed constants. We note that the work of \cite{Bramburger3} takes care of the $\ell^p$ norm decay, and hence we work exclusively on proving the $Q_p$ semi-norm decay. 

Now, the solution $x(t)$ to $\dot{x} = Lx$ is given explicitly by
\[
	x_{i,j}(t) = \sum_{(n,m)\in\mathbb{Z}^2} g_{i,j,n,m}(t)x_{n,m}^0
\]  
for all $(i,j)\in\mathbb{Z}^2$ and $x^0 =\{x_{i,j}^0\}_{(i,j)\in\mathbb{Z}^2} \in \ell^1$. Furthermore, there exists a uniform constants $C_g,\eta > 0$ such that 
\begin{equation}\label{RandomWalk}
	\begin{split}
		|g_{i,j,n,m}(t)| &\leq C_g(1+t)^{-1}\mathrm{e}^{-C_g\frac{(i - n)^2 + (j - m)^2}{t}}, \\
		|g_{i',j',n,m}(t) - g_{i,j,n,m}(t)| &\leq C_g(1+t)^{-1} g_{i,j,n,m}(2t),
	\end{split}
\end{equation}
for all $(i,j),(n,m)\in\mathbb{Z}^2$, which follow from the work of \cite{Delmotte} and \cite[Theorem~5.4.12]{SaloffCoste}, respectively. Then, using (\ref{RandomWalk}) we have 
\[
	\begin{split}
		Q_1(x(t)) &= \sum_{(i,j)\in\mathbb{Z}^2}\sum_{i',j'}|x_{i',j'}(t) - x_{i,j}(t)| \\
		&\leq \sum_{(i,j)\in\mathbb{Z}^2}\sum_{i',j'}\sum_{(n,m)\in\mathbb{Z}^2} |g_{i',j',n,m}(t) - g_{i,j,n,m}(t)||x^0_{n,m}| \\
		&\leq C_g(1+t)^{-1}\sum_{(i,j)\in\mathbb{Z}^2}\sum_{i',j'}\sum_{(n,m)\in\mathbb{Z}^2} |g_{i,j,n,m}(2t)||x^0_{n,m}| \\
		&\leq 4C_g(1+t)^{-1}\|\tilde{x}_{i,j}(t)\|_1,
	\end{split}
\]
where the multiple $4$ comes from the fact that each $(i,j)\in\mathbb{Z}^2$ has exactly four nearest-neighbours $(i',j')$ and $\tilde{x}(t) = \{\tilde{x}_{i,j}\}_{(i,j)\in\mathbb{Z}^2}$ represents the solution of $\dot{x} = Lx$ with initial condition $|x^0| = \{|x_{i,j}|\}_{(i,j)\in\mathbb{Z}^2} \in \ell^1$. Then, since there exists $C_L > 0$ such that 
\[
	\|\tilde{x}_{i,j}(t)\|_1 \leq C_L\|x^0\|_1
\]
from the fact that $x^0$ and $|x^0|$ have the same $\ell^1$ norm, it follows that 
\[
	Q_1(x(t)) \leq 4C_gC_L(1+t)^{-1}\|x^0\|_1. 
\]
This gives the bound on $Q_1$.

The $Q_\infty$ decay is significantly easier in that we have 
\[
	\begin{split}
	|x_{i',j'}(t) - x_{i,j}(t)| &\leq \sup_{(i,j) \in \mathbb{Z}^2}\sum_{(n,m)\in\mathbb{Z}^2} |g_{i',j',n,m}(t) - g_{i,j,n,m}(t)||x^0_{n,m}| \\
	&\leq C_g(1+t)^{-1} \sup_{(i,j) \in \mathbb{Z}^2}\sum_{i',j'}\sum_{(n,m)\in\mathbb{Z}^2}|g_{i,j,n,m}(2t)||x^0_{n,m}| \\ 
	&\leq 4C_g^2(1+2t)^{-2} \|x^0\|_1  
	\end{split}
\]
for all $(i,j) \in \mathbb{Z}^2$. This then gives that 
\[
	Q_\infty(x(t)) \leq 4C_g^2(1+t)^{-2}\|x^0\|_1, 	
\]
for all $x^0 \in \ell^1$, thus establishing the bound on $Q_\infty$. From here we can interpolate over $p \in [1,\infty]$ to arrive at the desired result. This completes the proof.

\section{Existence and Stability of an Invariant Manifold}\label{sec:InvMan} 

Much of this section extends the work of \cite[\S VII]{Hale} to our particular situation in infinite dimensions. Throughout this section we assume that Hypothesis~\ref{hyp:LambdaOmega} and \ref{hyp:PolarSoln} are true, but for convenience we will not explicitly state that they have been assumed in the statements of our results. Subsection~\ref{subsec:ManExistence} first deals with the existence of the invariant manifold and then Subsection~\ref{subsec:ManStability} proves the stability.

\subsection{Existence of the Invariant Manifold}\label{subsec:ManExistence} 

Let us define the vector space 
\begin{equation}
	X = \{\sigma:\ell^1 \to \ell^1:\ \sigma(0) = 0,\ \sup_{\psi \in \ell^1}\|\sigma(\psi)\|_\infty < \infty\},
\end{equation}
along with the associated norm on $X$ given by
\begin{equation}\label{XNorm} 
	\|\sigma\|_X := \sup_{\psi\in\ell^1} \|\sigma(\psi)\|_\infty.
\end{equation}
Notice that this norm is indeed well-defined because $\ell^1 \subsetneq \ell^\infty$ implies that $\|\sigma(\psi)\|_\infty < \infty$ for all $\psi \in \ell^1$. Our interest will lie in the following subsets of $X$:
\begin{equation}\label{XSpace}
	\begin{split}
	\tilde{X}(D,\Delta) = \{\sigma:\ell^1 \to \ell^1:\ \|\sigma\|_X \leq D,\  &\|\sigma(\psi) - \sigma(\tilde{\psi})\|_p \leq \Delta Q_p(\psi - \tilde{\psi})\ \forall \psi,\tilde{\psi} \in \ell^1\}.
	\end{split}
\end{equation}
where $D,\Delta > 0$ are constants to be specified shortly. Here the norm on $X$ will simply be used to obtain a contraction of an appropriate mapping whose fixed points are exactly invariant manifolds of (\ref{FullODE}). Hence, what matters is the boundedness and Lipschtz properties of elements of $\tilde{X}(D,\Delta)$ since they will give the desired properties (\ref{ManifoldProperties}) of the invariant manifold. We provide the following result. 

\begin{lem} \label{lem:X_Complete} 
	For every $D,\Delta > 0$, $\tilde{X}(D,\Delta)$ is complete with respect to the norm $\|\cdot\|_X$.	
\end{lem}  

\begin{proof}
	Let us fix $D,\Delta > 0$. Then, consider a Cauchy sequence $\{\sigma_n\}_{n=1}^\infty \subset \tilde{X}(D,\Delta)$. The existence of a pointwise limit $\sigma$ converging in the $\ell^\infty$ norm follows in a straightforward way from the fact that $\|\sigma_n(\psi)\|_\infty \leq D$ for all $\psi \in \ell^1$. We now wish to show that $\sigma \in \tilde{X}(D,\Delta)$.
	
	For a contradiction, let us assume that there exists $\psi \in \ell^1$ such that $\|\sigma(\psi)\|_\infty > D$. Set 
	\[
		\varepsilon := \frac{1}{2}(\|\sigma(\psi)\|_\infty - D) > 0,
	\] 
	and take $N \geq 1$ sufficiently large to guarantee that $\|\sigma_n(\psi) - \sigma(\psi)\|_\infty < \varepsilon$. This then gives
	\[
		\|\sigma(\psi)\|_\infty \leq \|\sigma_n(\psi) - \sigma(\psi)\|_\infty + \|\sigma_n(\psi)\|_\infty < \varepsilon + D = \frac{1}{2}(\|\sigma(\psi)\|_\infty - D) + D = \frac{1}{2}\|\sigma(\psi)\|_\infty + \frac{1}{2}D, 
	\]
	where we have used the fact that $\sigma_n \in \tilde{X}(D,\Delta)$ implies that $\|\sigma_n(\psi)\|_\infty \leq D$. But then rearranging this expression gives
	\[
		\frac{1}{2}\|\sigma(\psi)\|_\infty \leq \frac{1}{2}D \implies \|\sigma(\psi)\|_\infty \leq D,	
	\] 
	which is a contradiction. This therefore shows that $\|\sigma\|_X \leq D$.
	
	To show that $\|\sigma(\psi) - \sigma(\tilde{\psi})\|_p \leq \Delta Q_p(\psi - \tilde{\psi})$ for every $\psi,\tilde{\psi} \in \ell^1$, we proceed in a nearly identical way to our previous proof showing that $\|\sigma\|_X \leq D$, and therefore this proof is omitted. We do remark that since $\sigma_n(0) = 0$ for every $n \geq 1$, we necessarily have $\sigma(0) = 0$, and since $\sigma$ is such that 
	\[
		\|\sigma(\psi) - \sigma(\tilde{\psi})\|_1 \leq \Delta Q_1(\psi - \tilde{\psi})
	\] 
	for every  $\psi,\tilde{\psi} \in \ell^1$, we may take $\tilde{\psi} = 0$ to find that
	\[
		\|\sigma(\psi)\|_1 \leq \Delta Q_1(\psi),
	\]
	for all $\psi \in \ell^1$. Since Lemma~\ref{lem:NormBnds} details that $Q_1(\psi) \leq 8\|\psi\|_1 < \infty$, we then have that $\sigma(\psi) \in \ell^1$ for all $\psi \in \ell^1$. Therefore, $\sigma \in \tilde{X}(D,\Delta)$, completing the proof. 
\end{proof} 

Now, let us consider arbitrary $0 < D \leq \frac{a}{4}$ and a function $\sigma \in \tilde{X}(D,\Delta)$. We will denote $\psi^*(t;\psi^0,\sigma)$ to be the solution of the initial value problem
\begin{equation} \label{InvManIVP} 
	\begin{cases}
		\dot{\psi} = \alpha G(\sigma(\psi),\psi,\alpha), \\
		\psi(0) = \psi^0,
	\end{cases}
\end{equation}
where we use the superscript notation $\psi^0$ so as not to confuse with the subscripts relating to the indices of the lattice $\mathbb{Z}^2$. Note that we do indeed require the condition $D \leq \frac{a}{4}$ to guarantee that $\bar{r}_{i,j}(\alpha) + \sigma_{i,j}(\psi) \geq \frac{a}{4}$, to avoid the singularity in the phase equations when $\bar{r}_{i,j}(\alpha) + \sigma_{i,j}(\psi) = 0$. We also note that taking $\psi^0 = 0$ results in the solution $\psi^*(t;\psi^0,\sigma) = 0$ for all $t\geq 0$ and $\sigma \in \tilde{X}(D,\Delta)$ since $\sigma(0) = 0$ and $G(0,0,\alpha) = 0$.  

\begin{lem} \label{lem:InvManPhase} 
	There exists a constant $C_1 > 0$ such that for every $D \in (0,\frac{a}{4}]$, $\Delta \in (0,1]$, $\sigma,\tilde{\sigma} \in \tilde{X}(D,\Delta)$, $\psi,\tilde{\psi} \in \ell^1$, and $p \in [1,\infty]$ we have the following:
	\begin{subequations}
		\begin{align}
			Q_p(\psi^*(t;\psi,\sigma) - \psi^*(t;\tilde{\psi},\sigma)) &\leq e^{\alpha C_1|t|}Q_p(\psi - \tilde{\psi}), \label{PhaseBnd1} \\
			\|\psi^*(t;\psi,\sigma) - \psi^*(t;\psi,\tilde{\sigma})\|_\infty &\leq e^{\alpha C_1|t|}\|\sigma - \tilde{\sigma}\|_X, \label{PhaseBnd2}
		\end{align}
	\end{subequations}
	for all $t \in \mathbb{R}$.
\end{lem}

\begin{proof}
	We will only prove the inequalities for $t \geq 0$, since the proof for $t < 0$ is handled in a nearly identical way by introducing the temporal transformation $t \to -t$ and taking advantage of the fact that the differential equation is autonomous. Then to begin, recall that $G(0,0,\alpha) = 0$, by definition of $s$ and $\psi$. Furthermore, for each $(i,j) \in \mathbb{Z}^2$ we have that $G_{i,j}(s,\psi,\alpha)$ vanishes when $s_{i,j} = 0$, $s_{i',j'} = 0$ and $(\psi_{i',j'} - \psi_{i,j}) = 0$, for all $(i',j')$. Hence, from the fact that $G_{i,j}(s,\psi,\alpha)$ is globally Lipschitz in $\psi$, uniformly in $(i,j)$, and since $\|\sigma(\psi)\|_\infty$ is uniformly bounded by $D \leq \frac{a}{4}$, we find that there exists a $C > 0$, uniform in $D \in (0,\frac{a}{4}]$, such that
	\begin{equation} \label{PhaseSolnBnd}
		\begin{split}
		|G_{i,j}(\sigma(\psi),\psi,\alpha) - G_{i,j}(\tilde{\sigma}(\tilde{\psi}),\tilde{\psi},\alpha)| \leq &C\bigg(|\sigma_{i,j}(\psi) - \tilde{\sigma}_{i,j}(\tilde{\psi})| + \sum_{i',j'} |\sigma_{i',j'}(\psi) - \tilde{\sigma}_{i',j'}(\tilde{\psi})|\bigg) \\
		&+ \sum_{i',j'} |(\psi_{i',j'} - \psi_{i,j}) - (\tilde{\psi}_{i',j'} - \tilde{\psi}_{i,j})|,
		\end{split}
	\end{equation}
	for all $\alpha \in [0,\alpha^*]$.
	
	Now starting with the first bound we wish to prove, begin by taking $\sigma = \tilde{\sigma}$. Then, from Lemma~\ref{lem:AltGradient} we have that 
	\[
		\bigg(\sum_{(i,j)\in\mathbb{Z}^2}\bigg(\sum_{i',j'} |(\psi_{i',j'} - \psi_{i,j}) - (\tilde{\psi}_{i',j'} - \tilde{\psi}_{i,j})|\bigg)^p\bigg)^\frac{1}{p} \leq 4Q_p(\psi - \tilde{\psi}),	
	\]
	for all $p \in [1,\infty)$. Hence, we have
	\[
		\begin{split}
		\|G(\sigma(\psi),\psi,\alpha) - G(\sigma(\tilde{\psi}),\tilde{\psi},\alpha)\|_p &\leq 5C\|\sigma(\psi) - \sigma(\tilde{\psi})\|_p + 4CQ_p(\psi - \tilde{\psi}), \\
			&\leq 5C\Delta Q_p(\psi - \tilde{\psi}) + 4CQ_p(\psi - \tilde{\psi}) \\
			&\leq (5\Delta + 4)CQ_p(\psi - \tilde{\psi}) \\
			&\leq 9CQ_p(\psi - \tilde{\psi}),
		\end{split}
	\]
	for all $p \in [1,\infty)$, since we have assumed $\Delta \leq 1$. Similarly, the case of $p = \infty$ gives
	\[
		\begin{split}
		\|G(\sigma(\psi),\psi,\alpha) - G(\sigma(\tilde{\psi}),\tilde{\psi},\alpha)\|_\infty &\leq 5C\|\sigma(\psi) - \sigma(\tilde{\psi})\|_\infty + CQ_\infty(\psi - \tilde{\psi}), \\
			&\leq 5C\Delta Q_\infty(\psi - \tilde{\psi}) + CQ_p(\psi - \tilde{\psi}) \\
			&\leq (5\Delta + 1)CQ_\infty(\psi - \tilde{\psi}) \\
			&\leq 6CQ_\infty(\psi - \tilde{\psi}), \\
			&< 9CQ_p(\psi - \tilde{\psi}),
		\end{split}
	\]
	so that
	\begin{equation}\label{PhaseSolnBnd2}
		\|G(\sigma(\psi),\psi,\alpha) - G(\sigma(\tilde{\psi}),\tilde{\psi},\alpha)\|_p \leq 9CQ_p(\psi - \tilde{\psi}),
	\end{equation}  
	for all $p \in [1,\infty]$ and $\psi,\tilde{\psi}\in\ell^1$.
	
	Then, if $\psi^*(t;\psi,\sigma)$ is a solution of (\ref{InvManIVP}), it therefore satisfies the integral form equation
	\begin{equation}\label{InvManIVPInt}
		\psi^*(t;\psi,\sigma) = \psi + \int_0^t G(\sigma(\psi^*(u;\psi,\sigma)),\psi^*(u;\psi,\sigma),\alpha) \mathrm{d}u. 
	\end{equation}
	Hence, using this integral formulation we obtain 
	\[
		\begin{split}
		Q_p(\psi^*(t;\psi,\sigma) - \psi^*(t;\tilde{\psi},\sigma)) &\leq Q_p(\psi - \tilde{\psi}) \\ &+ \int_0^t \alpha Q_p(G(\sigma(\psi^*(u;\psi,\sigma)),\psi^*(u;\psi,\sigma),\alpha) - G(\psi^*(u;\tilde{\psi},\sigma)),\psi^*(u;\tilde{\psi},\sigma),\alpha)) \mathrm{d}u \\
		&\leq Q_p(\psi - \tilde{\psi}) + \int_0^t 8\alpha \|\psi^*(u;\psi,\sigma) - \psi^*(u;\tilde{\psi},\sigma)\|_p\mathrm{d}u, \\ 
		&\leq Q_p(\psi - \tilde{\psi}) + \int_0^t 72C\alpha Q_p(\psi^*(u;\psi,\sigma) - \psi^*(u;\tilde{\psi},\sigma))\mathrm{d}u, \\ 
		\end{split}	
	\]
	where we have applied the righthand bound of Lemma~\ref{lem:NormBnds} followed by the bound (\ref{PhaseSolnBnd2}) in the final two steps, respectively. Then, using Gronwall's inequality we obtain
	\[
		Q_p(\psi^*(t;\psi,\sigma) - \psi^*(t;\tilde{\psi},\sigma)) \leq e^{72C\alpha t}Q_p(\psi - \tilde{\psi}).	
	\]  
	This proves the first bound stated in the lemma.
	
	Now, using the bound (\ref{PhaseSolnBnd}) again we have
	\[
		\begin{split}
		|G_{i,j}(\sigma(\psi),\psi,\alpha) - G_{i,j}(\tilde{\sigma}(\tilde{\psi}),\tilde{\psi},\alpha)| &\leq 5C\|\sigma(\psi) - \sigma(\tilde{\psi})\|_\infty  + C\sum_{i',j'} |\psi_{i',j'} - \tilde{\psi}_{i',j'}|, \\  
		&\leq 5C\|\sigma - \tilde{\sigma}\|_X + 4C\|\psi - \tilde{\psi}\|_\infty	
		\end{split}
	\]
	for all $(i,j) \in \mathbb{Z}^2$. Hence, using the integral formulation (\ref{InvManIVPInt}), we arrive at  
	\[
		\begin{split}
		|\psi_{i,j}^*(t;\psi,\sigma) - \psi_{i,j}^*(t;\psi,\tilde{\sigma})| &\leq \alpha\int_0^t|G(\sigma(\psi^*(u;\psi,\tilde{\sigma})),\psi^*(u;\psi,\sigma),\alpha) - G(\psi^*(u;\psi,\tilde{\sigma}),\psi^*(u;\psi,\tilde{\sigma}),\alpha)| \mathrm{d}u \\
		&\leq 5C\alpha\int_0^t \|\sigma - \tilde{\sigma}\|_X + \|\psi^*(u;\psi,\sigma) - \psi^*(u;\psi,\tilde{\sigma})\|_\infty \mathrm{d}u \\
		&\leq 5C\alpha \|\sigma - \tilde{\sigma}\|_X t + 4C\alpha\int_0^t\|\psi^*(u;\psi,\sigma) - \psi^*(u;\psi,\tilde{\sigma})\|_\infty \mathrm{d}u.
		\end{split}
	\]
	Taking the supremum over all $(i,j)\in\mathbb{Z}^2$ we arrive at
	\[
		\|\psi^*(t;\psi,\sigma) - \psi^*(t;\psi,\tilde{\sigma})\|_\infty \leq 5C\alpha \|\sigma - \tilde{\sigma}\|_X t + 4C\alpha\int_0^t\|\psi^*(u;\psi,\sigma) - \psi^*(u;\psi,\tilde{\sigma})\|_\infty du. 	
	\]
	We now apply Gronwall's inequality to arrive at the bound 
	\[
		\|\psi^*(t;\psi,\sigma) - \psi^*(t;\psi,\tilde{\sigma})\|_\infty \leq 5C\alpha t e^{4C\alpha t}\|\sigma - \tilde{\sigma}\|_X.	
	\] 
	Finally, there exists $C' > 0$ such that
	\[
		5C\alpha t e^{4C\alpha t} \leq e^{\alpha C' t},
	\] 
	for all $t \geq 0$ and $\alpha \in [0,\alpha^*]$. This completes the proof.
\end{proof} 

Our goal is now to obtain a mapping acting on the space $\tilde{X}(D,\Delta)$, for $D,\Delta > 0$ appropriately chosen, so that a fixed point is exactly an invariant manifold for the differential equations (\ref{FullODE}). To begin, let us rewrite 
\[
	\dot{s} = F(s,\psi,\alpha),
\] 
as
\[
	\dot{s} = a\lambda'(a)s + [F(s,\psi,\alpha)-a\lambda'(a)s].
\]
One should recall that the derivative of $F$ with respect to $s$ is $D_s\sigma(0,0,\alpha) = a\lambda'(a)I$, where $I$ is the identity mapping. We recall that $a\lambda'(a) < 0$ from Hypothesis~\ref{hyp:LambdaOmega}. Then, using the variation of constants formula we arrive at 
\[
	s(t) = e^{a\lambda'(a)(t-t_0)}s(t_0) + \int_{t_0}^te^{a\lambda'(a)(t-u)}[F(s(u),\psi(u),\alpha)-a\lambda'(a)s(u)]\mathrm{d}u. 	
\]
Assuming that $s$ belongs to $\tilde{X}(D,\Delta)$ for all $t$, it is bounded and we can therefore take $t_0 \to -\infty$ to arrive at 
\[
	s(t) = \int_{-\infty}^te^{a\lambda'(a)(t-u)}[F(s(u),\psi(u),\alpha)-a\lambda'(a)s(u)]\mathrm{d}u. 	
\]
Finally, taking $t = 0$ allows one to define a mapping $T$ with domain $\tilde{X}(D,\Delta)$, for appropriately chosen $D,\Delta >0$, given by
\begin{equation}\label{T1Mapping} 
	T_1\sigma(\psi) = \int_{-\infty}^0e^{-a\lambda'(a)u}[F(\sigma(\psi^*(u;\psi,\sigma)),\psi^*(u;\psi,\sigma),\alpha)-a\lambda'(a)\sigma(\psi^*(u;\psi,\sigma))]du. 
\end{equation} 
Fixed points of $T_1$ are exactly invariant manifolds of the full system (\ref{FullODE}). To understand this mapping, $\psi \in \ell^1$ is an initial condition of the flow governed by (\ref{InvManIVP}) using the function $\sigma$. This solution, denoted $\psi^*(t;\psi,\sigma)$, is then put into the radial component equation and we describe the flow of the radial component. If this flow matches the flow governed by the flow of the original input function $\sigma$, we have indeed obtained a flow-invariant invariant manifold for (\ref{InvManIVP}). For the ease of notation we will define 
\begin{equation} \label{FTilde}
	\tilde{F}(s,\psi,\alpha) = F(s,\psi,\alpha)-a\lambda'(a)s,
\end{equation}  
so that we can write
\[
	T_1\sigma(\psi) = \int_{-\infty}^0e^{-a\lambda'(a)u}\tilde{F}(\sigma(\psi^*(u;\psi,\sigma)),\psi^*(u;\psi,\sigma),\alpha)\mathrm{d}u. 	
\]
Notice that $\tilde{F}(0,\psi,\alpha) = F(0,\psi,\alpha)$ and therefore $\tilde{F}(0,0,\alpha) = 0$ for all $\alpha \in [0,\alpha^*]$.

\begin{lem} \label{lem:T1WellDefined} 
	There exists $\alpha_{X,1} > 0$ such that for each $\alpha \in [0,\alpha_{X,1}]$, $T_1:\tilde{X}(\sqrt{\alpha},\sqrt{\alpha}) \to \tilde{X}(\sqrt{\alpha},\sqrt{\alpha})$ is well-defined.
\end{lem}

\begin{proof}
	We break this proof into three major components to show that for appropriately chosen $D,\Delta > 0$ and sufficiently small $\alpha > 0$ we can guarantee that $T_1\sigma(0) = 0$, $\|T_1\sigma\|_X \leq D$ and 
	\[
		\|T_1\sigma(\psi) - T_1\sigma(\tilde{\psi})\|_p \leq \Delta Q_p(\psi - \tilde{\psi})
	\] 
	for all $\psi,\tilde{\psi} \in \ell^1$, $\sigma \in \tilde{X}(\sqrt{\alpha},\sqrt{\alpha})$, and $p \in [1,\infty]$. Throughout this proof we will always assume $\alpha > 0$ is taken to satisfy: $\alpha \leq \min\{\alpha^*,(\frac{a}{4})^2,1\}$, so that our choices $D = \sqrt{\alpha}$ and $\Delta = \sqrt{\alpha}$ satisfy the assumptions of Lemma~\ref{lem:InvManPhase}.
	
	\underline{$T_1\sigma(0) = 0$}: Recall that taking $\psi = 0$ results in $\psi^*(t;0,\sigma) = 0$ for all $\sigma$ since $\sigma(0) = 0$. Then, evaluating $T_1\sigma(0)$ gives 
	\[
		T_1\sigma(0) = \int_{-\infty}^0e^{-a\lambda'(a)u}\tilde{F}(0,0,\alpha)du = 0,	
	\]
	since $\tilde{F}(0,0,\alpha) = 0$ for all $\alpha \in [0,\alpha^*]$. 
	
	\underline{$\|T_1\sigma\|_X \leq D$}: To begin, we consider arbitrary $\psi \in \ell^1$ and recall that 
	\begin{equation}\label{WellDefined1}
		\begin{split}
		&\tilde{\sigma}_{i,j}(\sigma(\psi^*(t;\psi,\sigma)),\psi^*(t;\psi,\sigma),\alpha)\\
		&= \alpha\sum_{i',j'} [(\bar{r}_{i',j'}(\alpha) + \sigma_{i',j'}(\psi^*(t;\psi,\sigma)))\cos(\bar{\theta}_{i',j'}(\alpha) + \psi_{i',j'}^*(t;\psi,\sigma) - \bar{\theta}_{i,j}(\alpha)- \psi_{i,j}^*(t;\psi,\sigma))\\ 
		&- (\bar{r}_{i,j}(\alpha)+ \sigma_{i,j}(\psi^*(t;\psi,\sigma))\bigg]+ (\bar{r}_{i,j}(\alpha) + \sigma_{i,j}(\psi^*(t;\psi,\sigma))\lambda(\bar{r}_{i,j}(\alpha) + \sigma_{i,j}(\psi^*(t;\psi,\sigma))-a\lambda'(a)\sigma_{i,j}(\psi^*(t;\psi,\sigma)).
		\end{split}
	\end{equation}
	For convenience, we will break down the bounds of this term into separate parts. First, we have
	\[
		|(\bar{r}_{i',j'}(\alpha) + \sigma_{i',j'}(\psi^*(t;\psi,\sigma)))\cos(\bar{\theta}_{i',j'}(\alpha) + \psi_{i',j'}^*(t;\psi,\sigma) - \bar{\theta}_{i,j}(\alpha)- \psi_{i,j}^*(t;\psi,\sigma))| \leq \frac{3a}{2} + \sqrt{\alpha},	
	\]
	since our definition of $\alpha^* > 0$ implies that $|\bar{r}_{i',j'}(\alpha)| \leq \frac{3a}{2}$ and $| \sigma_{i',j'}(\psi^*(u;\psi,\sigma)| \leq \sqrt{\alpha}$ since $\sigma \in \tilde{X}(\sqrt{\alpha},\sqrt{\alpha})$. Similarly, 
	\[
		|\bar{r}_{i,j}(\alpha)+ \sigma_{i,j}(\psi^*(t;\psi,\sigma))| \leq \frac{3a}{2} + \sqrt{\alpha}. 	
	\] 
	Then, we use the fact that $\lambda(a) = 0$ to apply Taylor's Theorem to see that there exists a constant $C_\lambda > 0$, independent of $D =  \sqrt{\alpha} \leq \frac{a}{4}$, so that 
	\begin{equation}\label{C_lambda}
		\begin{split}
		|(\bar{r}_{i,j}(\alpha) + \sigma_{i,j}(\psi^*(t;\psi,\sigma)))\lambda(\bar{r}_{i,j}(\alpha) &+ \sigma_{i,j}(\psi^*(t;\psi,\sigma)))-a\lambda'(a)\sigma_{i,j}(\psi^*(t;\psi,\sigma))| \\
		&\leq C_\lambda (|\sigma_{i,j}(\psi^*(t;\psi,\sigma)|^2 + |\bar{r}_{i,j}(\alpha) - a|) \\
		&\leq C_\lambda \alpha + C_\lambda C_r\alpha \\
		&= (1 + C_r)C_\lambda\alpha  
		\end{split}
	\end{equation}
	where $C_r > 0$ is the constant guaranteed by (\ref{RotWaveLip}). Then, using these inequalities we therefore return to (\ref{WellDefined1}) to see that
	\[
		\begin{split}
		|\tilde{F}_{i,j}(\sigma(\psi^*(t;\psi,\sigma)),\psi^*(t;\psi,\sigma),\alpha) | &\leq 4\alpha\bigg(\frac{3a}{2} + \sqrt{\alpha}\bigg) + 4\alpha\bigg(\frac{3a}{2} + \sqrt{\alpha}\bigg) +  (1 + C_r)C_\lambda\alpha \\ 
		&= 8\alpha\bigg(\frac{3a}{2} + \sqrt{\alpha}\bigg) + (1 + C_r)C_\lambda\alpha \\
		&\leq 8\alpha\bigg(\frac{3a}{2} + \frac{a}{4}\bigg) + (1 + C_r)C_\lambda\alpha \\ 
		&\leq (14a + C_\lambda + C_rC_\lambda)\alpha,	
		\end{split}
	\] 
	for all $(i,j) \in \mathbb{Z}^2$, since we have assumed $\sqrt{\alpha} \leq \frac{a}{4}$. Therefore, recalling that $a\lambda'(a) < 0$, this then implies that for any $\psi\in\ell^1$ we have
	\[
	\begin{split}
		\|T_1\sigma(\psi)\|_\infty &\leq (14a + C_\lambda + C_rC_\lambda)\alpha \int_{-\infty}^0e^{-a\lambda'(a)u}du \\
		&\leq \frac{-1}{a\lambda'(a)}\bigg[14a + C_\lambda + C_rC_\lambda\bigg]\alpha.  
	\end{split}
	\]
	Taking
	\[
		\alpha \leq \min\bigg\{\alpha^*,\bigg(\frac{a}{4}\bigg)^2,1,\bigg(\frac{a\lambda'(a)}{14a + C_\lambda + C_rC_\lambda}\bigg)^2\bigg\},
	\]
	gives that $\|T_1\sigma(\psi)\|_\infty \leq \sqrt{\alpha}$, for all $\psi \in \ell^1$. Then taking the supremum over all $\psi\in\ell^1$ we have $\|T_1\sigma\|_X \leq \sqrt{\alpha}$, as required. 
	
	\underline{$\|T_1\sigma(\psi) - T_1\sigma(\tilde{\psi})\|_p \leq \Delta Q_p(\psi - \tilde{\psi})$}: This proof proceeds in a similar manner to the previous bound. Begin by fixing $\sigma \in \tilde{X}(\sqrt{\alpha},\sqrt{\alpha})$, $\psi,\tilde{\psi}\in\ell^1$. We again use the form (\ref{WellDefined1}) and break the bounds into smaller pieces as in the proof of the previous bound. 
	
	To begin, we use the uniform boundedness of cosine and its derivatives to obtain  
	\[
	\begin{split}
		|&(\bar{r}_{i',j'}(\alpha) + \sigma_{i',j'}(\psi^*(t;\psi,\sigma)))\cos(\bar{\theta}_{i',j'}(\alpha) + \psi_{i',j'}^*(t;\psi,\sigma) - \bar{\theta}_{i,j}(\alpha)- \psi_{i,j}^*(t;\psi,\sigma)) \\
		&-(\bar{r}_{i',j'}(\alpha) + \sigma_{i',j'}(\psi^*(t;\tilde{\psi},\sigma)))\cos(\bar{\theta}_{i',j'}(\alpha) + \psi_{i',j'}^*(t;\tilde{\psi},\sigma) - \bar{\theta}_{i,j}(\alpha)- \psi_{i,j}^*(t;\tilde{\psi},\sigma))| \\
		&\leq |\sigma_{i',j'}(\psi^*(t;\psi,\sigma)) - \sigma_{i',j'}(\psi^*(t;\tilde{\psi},\sigma))| + |(\psi_{i',j'}^*(t;\psi,\sigma) - \psi_{i,j}^*(t;\psi,\sigma)) - (\psi_{i',j'}^*(t;\tilde{\psi},\sigma) - \psi_{i,j}^*(t;\tilde{\psi},\sigma))|.	
	\end{split}
	\]
	Then, trivially we have
	\[
		|(\bar{r}_{i,j}(\alpha)+ \sigma_{i,j}(\psi^*(t;\psi,\sigma))) - (\bar{r}_{i,j}(\alpha)+ \sigma_{i,j}(\psi^*(t;\tilde{\psi},\sigma)))| = |\sigma_{i,j}(\psi^*(t;\psi,\sigma)) - \sigma_{i,j}(\psi^*(t;\tilde{\psi},\sigma))|,	
	\]
	which we point out for the sake of completeness. And finally, for all $\sigma \in \tilde{X}(\sqrt{\alpha},\sqrt{\alpha})$, there exists a $C'_\lambda>0$, independent of $0 \leq \alpha \leq \frac{a}{4}$, so that 
	\[
		\begin{split}
		|(\bar{r}_{i,j}(\alpha) &+ \sigma_{i,j}(\psi^*(t;\psi,\sigma)))\lambda(\bar{r}_{i,j}(\alpha)+\sigma_{i,j}(\psi^*(t;\psi,\sigma))) - (\bar{r}_{i,j}(\alpha) + \sigma_{i,j}(\psi^*(t;\tilde{\psi},\sigma)))\lambda(\bar{r}_{i,j}(\alpha) + \sigma_{i,j}(\psi^*(t;\tilde{\psi},\sigma)))\\ 
		&-a\lambda'(a)(\sigma_{i,j}(\psi^*(t;\psi,\sigma))-\sigma_{i,j}(\psi^*(t;\tilde{\psi},\sigma)))| \\
		&\leq \sup_{|x| \leq \sqrt{\alpha}} |\lambda(\bar{r}_{i,j}(\alpha)+x)+ (\bar{r}_{i,j}(\alpha) + x)\lambda'(\bar{r}_{i,j}(\alpha)+x) - a\lambda'(a)| |\sigma_{i,j}(\psi^*(t;\psi,\sigma))-\sigma_{i,j}(\psi^*(t;\tilde{\psi},\sigma))| \\
		&\leq C'_\lambda \alpha|\sigma_{i,j}(\psi^*(t;\psi,\sigma))-\sigma_{i,j}(\psi^*(t;\tilde{\psi},\sigma))|, 
		\end{split}
	\]
	since the function $\lambda(\bar{r}_{i,j}(\alpha)+x)+ (\bar{r}_{i,j}(\alpha) + x)\lambda'(\bar{r}_{i,j}(\alpha)+x) - a\lambda'(a)$ vanishes when $(x,\alpha) = (0,0)$ because $\lambda(a) = 0$ by assumption and $r_{i,j}(0) = a$ for all $(i,j) \in \mathbb{Z}^2$.
	
	Now, we use these three previous bounds and the form (\ref{WellDefined1}) to see that 
	\begin{equation}\label{WellDefined2}
		\begin{split}
			|\tilde{F}_{i,j}(\sigma(\psi^*(t;\psi,\sigma),&\psi^*(t;\psi,\sigma),\alpha)-\tilde{F}_{i,j}(\sigma(\psi^*(t;\tilde{\psi},\sigma),\psi^*(t;\tilde{\psi},\sigma),\alpha)| \\
			&\leq \alpha\sum_{i',j'} |\sigma_{i',j'}(\psi^*(t;\psi,\sigma)- \sigma_{i',j'}(\psi^*(t;\tilde{\psi},\sigma)| + 4\alpha|\sigma_{i,j}(\psi^*(t;\psi,\sigma)) - \sigma_{i,j}(\psi^*(t;\tilde{\psi},\sigma))| \\ 
			&+ \alpha\sum_{i',j'} |(\psi^*_{i',j'}(t;\psi,\sigma) - \psi^*_{i,j}(t;\psi,\sigma)) - (\psi^*_{i',j'}(t;\tilde{\psi},\sigma) - \psi^*_{i,j}(t;\tilde{\psi},\sigma))| \\ 
			&+ C'_\lambda \alpha|\sigma_{i,j}(\psi^*(t;\psi,\sigma))-\sigma_{i,j}(\psi^*(t;\tilde{\psi},\sigma))|, 
		\end{split}
	\end{equation}
	for all $(i,j) \in \mathbb{Z}^2$. Then, for all $p \in [1,\infty)$, using Lemma~\ref{lem:AltGradient}, we obtain
	\[
		\begin{split}
			\|\tilde{F}(\sigma(\psi^*(t;\psi,\sigma),\psi^*(t;\psi,\sigma),\alpha) &-\tilde{F}(\sigma(\psi^*(t;\tilde{\psi},\sigma),\psi^*(t;\tilde{\psi},\sigma),\alpha)\|_p \\
			&\leq 8\alpha\|\sigma(\psi^*(t;\psi,\sigma)) - \sigma(\psi^*(t;\tilde{\psi},\sigma))\|_p + 4\alpha Q_p(\psi^*(t;\psi,\sigma) - \psi^*(t;\tilde{\psi},\sigma)) \\ 
			&+ C'_\lambda \alpha\|\sigma(\psi^*(t;\psi,\sigma)) - \sigma(\psi^*(t;\tilde{\psi},\sigma))\|_p \\
			&\leq  (8\alpha^\frac{3}{2} + 4\alpha + C'_\lambda \alpha^\frac{3}{2})Q_p(\psi^*(t;\psi,\sigma) - \psi^*(t;\tilde{\psi},\sigma)) \\
			&\leq  (13 + C'_\lambda)\alpha e^{\alpha C_1 |t|}Q_p(\psi - \tilde{\psi}),		
		\end{split}
	\]
	where we have used the facts that $\sigma \in \tilde{X}(\sqrt{\alpha},\sqrt{\alpha})$ and $\sqrt{\alpha} \leq 1$, as well as applied (\ref{PhaseBnd1}) from Lemma~\ref{lem:InvManPhase} with the constant $C_1 > 0$. Taking $\alpha \leq \frac{-a\lambda'(a)}{2C_1}$ guarantees that 
	\[
		\alpha C_1 + a\lambda'(a) \leq \frac{a\lambda'(a)}{2} < 0, 	
	\]
	and hence,
	\[
		\begin{split}
		\|T_1\sigma(\psi) - T_1\sigma(\tilde{\psi})\|_p &\leq (9 + C'_\lambda)\alpha Q_p(\psi - \tilde{\psi}) \int_{-\infty}^0 e^{-(\alpha C_1 + a\lambda'(a)) u} du \\ 
		&= \bigg(\frac{9 + C'_\lambda}{-(\alpha C_1+a\lambda'(a))}\bigg)\alpha Q_p(\psi - \tilde{\psi}) \\
		&\leq -\bigg(\frac{18 + 2C'_\lambda}{a\lambda'(a)}\bigg)\alpha Q_p(\psi - \tilde{\psi}).  
		\end{split}
	\]
	Therefore, taking 
	\[
		\alpha \leq \min\bigg\{\alpha^*,\bigg(\frac{a}{4}\bigg)^2,1,\frac{-a\lambda'(a)}{2C_1},\bigg(\frac{a\lambda'(a)}{18 + 2C'_\lambda}\bigg)^2\bigg\},
	\]
	provides that $\|T_1\sigma(\psi) - T_1\sigma(\tilde{\psi})\|_p \leq \sqrt{\alpha}Q_p(\psi - \tilde{\psi})$ for all $\psi,\tilde{\psi} \in \ell^1$ and $p \in [1,\infty)$. The case when $p = \infty$ follows in exactly the same way, and is committed. 
	
	In closing, we can define $\alpha_{X,1}$ as
	\[
		\alpha_{X,1}:= \min\bigg\{\alpha^*,\bigg(\frac{a}{4}\bigg)^2,1,\bigg(\frac{a\lambda'(a)}{14a + C_\lambda + C_rC_\lambda}\bigg)^2,\frac{-a\lambda'(a)}{2C_1},\bigg(\frac{a\lambda'(a)}{18 + 2C'_\lambda}\bigg)^2\bigg\} 
	\]
	so that for all $\alpha \in [0,\alpha_{X,1}]$ we have that $T_1:\tilde{X}(\sqrt{\alpha},\sqrt{\alpha}) \to \tilde{X}(\sqrt{\alpha},\sqrt{\alpha})$. This concludes the proof.
\end{proof} 

\begin{lem} \label{lem:T1Contraction} 
	There exists $\alpha_{X,2} > 0$ such that for each $\alpha \in [0,\alpha_{X,2}]$, $T_1:\tilde{X}(\sqrt{\alpha},\sqrt{\alpha}) \to \tilde{X}(\sqrt{\alpha},\sqrt{\alpha})$ is a contraction with contraction constant at most $\frac{1}{2}$.
\end{lem}

\begin{proof}
	This proof proceeds by applying very similar manipulations to that of the previous lemma to show that $T_1$ is well-defined, and therefore we will omit some details which are redundant. Furthermore, we will always consider $\alpha \leq \alpha_{X,1}$, so that the conclusion of Lemma~\ref{lem:T1WellDefined} holds and that our choices $D = \sqrt{\alpha}$ and $\Delta = \sqrt{\alpha}$ satisfy the assumptions of Lemma~\ref{lem:InvManPhase}.
	
	To begin, let us consider $\sigma,\tilde{\sigma} \in \tilde{X}(\sqrt{\alpha},\sqrt{\alpha})$. Then, following the manipulations in (\ref{WellDefined2}) we obtain the similar bound 
 	\begin{equation} \label{TContractBnd1}
		\begin{split}
			|\tilde{F}_{i,j}(\sigma(\psi^*(t;\psi,\sigma),&\psi^*(t;\psi,\sigma),\alpha)-\tilde{F}_{i,j}(\tilde{\sigma}(\psi^*(t;\psi,\tilde{\sigma}),\psi^*(t;\psi,\tilde{\sigma}),\alpha)| \\
			&\leq \alpha\sum_{i',j'} |\sigma_{i',j'}(\psi^*(t;\psi,\sigma))- \tilde{\sigma}_{i',j'}(\psi^*(t;\psi,\tilde{\sigma}))| + 4\alpha|\sigma_{i,j}(\psi^*(t;\psi,\sigma)) - \tilde{\sigma}_{i,j}(\psi^*(t;\psi,\tilde{\sigma}))| \\ 
			&+ \alpha\sum_{i',j'} |(\psi^*_{i',j'}(t;\psi,\sigma) - \psi^*_{i,j}(t;\psi,\sigma)) - (\psi^*_{i',j'}(t;\psi,\tilde{\sigma}) - \psi^*_{i,j}(t;\psi,\tilde{\sigma}))| \\ 
			&+ C'_\lambda \alpha|\sigma_{i,j}(\psi^*(t;\psi,))-\tilde{\sigma}_{i,j}(\psi^*(t;\psi,\tilde{\sigma}))|, 
		\end{split}
	\end{equation}
	for all $(i,j)\in\mathbb{Z}^2$, and we recall that $C'_\lambda > 0$ is the constant used in the proof of Lemma~\ref{lem:T1WellDefined} for which
	\[
		|\lambda(\bar{r}_{i,j}(\alpha)+x)+ (\bar{r}_{i,j}(\alpha) + x)\lambda'(\bar{r}_{i,j}(\alpha)+x) - a\lambda'(a)| \leq C'_\lambda \alpha
	\] 
	for all $|x| \leq \sqrt{\alpha}$. Then, using (\ref{TContractBnd1}) we can take the supremum over all $(i,j) \in\mathbb{Z}^2$ to get
	\[
		\begin{split}
			\|&\tilde{F}(\sigma(\psi^*(t;\psi,\sigma)),\psi^*(t;\psi,\sigma),\alpha)-\tilde{F}(\tilde{\sigma}(\psi^*(t;\psi,\tilde{\sigma})),\psi^*(t;\psi,\tilde{\sigma}),\alpha)\|_\infty \\
			&\leq (8 + C'_\lambda)\alpha \|\sigma(\psi^*(t;\psi,\sigma)) -\tilde{\sigma}(\psi^*(t;\psi,\tilde{\sigma}))\|_\infty +4\alpha \|\psi^*(t;\psi,\sigma) - \psi^*(t;\psi,\tilde{\sigma})\|_\infty \\
			&\leq (8 + C'_\lambda)\alpha\|\sigma - \tilde{\sigma}\|_X +4\alpha e^{\alpha C_1|t|}\|\sigma - \tilde{\sigma}\|_X \\
			&\leq (12 + C'_\lambda)\alpha e^{\alpha C_1|t|}\|\sigma - \tilde{\sigma}\|_X. 
		\end{split}	
	\]
	Therefore, taking $\alpha \leq \frac{-a\lambda'(a)}{2C_1}$ guarantees that 
	\[
		\alpha C_1 + a\lambda'(a) \leq \frac{a\lambda'(a)}{2} < 0, 	
	\]
	and hence,
	\[
		\begin{split}
		\|T_1\sigma(\psi) - T\tilde{\sigma}(\psi)\|_\infty &\leq (12 + C'_\lambda)\alpha\|\sigma - \tilde{\sigma}\|_X\int_{-\infty}^0 e^{-(\alpha C_1 + a\lambda'(a)) u} du \\
		&\leq -\bigg(\frac{12 + C'_\lambda}{\alpha C_1 + a\lambda'(a)}\bigg)\alpha\|\sigma - \tilde{\sigma}\|_X \\
		&\leq -\bigg(\frac{24 + 2C'_\lambda}{a\lambda'(a)}\bigg)\alpha\|\sigma - \tilde{\sigma}\|_X.
		\end{split}
	\]
	Then, taking 
	\[
		\alpha_{X,2} := \min\bigg\{\alpha^*,\alpha_{X,1},\frac{-a\lambda'(a)}{2C_1},\frac{-a\lambda'(a)}{48 + 4C'_\lambda}\bigg\}
	\]
	gives $\|T_1\sigma(\psi) - T_1\tilde{\sigma}(\psi)\|_\infty \leq \frac{1}{2}\|\sigma-\tilde{\sigma}\|_X$ for all $\psi\in\ell^1$. Taking the supremum over $\psi\in\ell^1$ shows that $T_1$ is a contraction with contraction constant at most $\frac{1}{2}$, concluding the proof. 
\end{proof} 

Lemma~\ref{lem:T1Contraction} gives that for each $\alpha \in [0,\alpha_{X,2}]$ there exists a unique fixed point of the mapping $T_1$, simply denoted $\sigma(\psi,\alpha) \in \tilde{X}(\sqrt{\alpha},\sqrt{\alpha})$. Recall that the properties of $\tilde{X}(\sqrt{\alpha},\sqrt{\alpha})$ therefore imply that 
\[
	\begin{split}
		\sigma(0,\alpha) &= 0, \\
		\|\sigma(\psi,\alpha)\|_\infty &\leq \sqrt{\alpha}, \\
		\|\sigma(\psi,\alpha) - \sigma(\tilde{\psi},\alpha)\|_p &\leq \sqrt{\alpha}Q_p(\psi - \tilde{\psi}), 	
	\end{split}
\]
for all $\psi,\tilde{\psi}\in\ell^1$ and $p \in [1,\infty]$. As previously stated, this fixed point corresponds to an invariant manifold of the differential equation (\ref{FullODE}). We close this section by discussing the final point in Theorem~\ref{thm:InvMan} relating to the situation when $\omega_1$ is identically zero.

\subsection{Stability of the Invariant Manifold}\label{subsec:ManStability} 

This subsection proceeds in a similar way to the previous subsection in that we apply a bootstrapping argument to an appropriate mapping to determine the stability of the invariant manifold. In fact, much of this section follows the proof of the Stable Manifold Theorem, but since we are working in infinite dimensions some extra attention must be paid to certain aspects of the problem. Recall that in the previous section we determined the existence of an invariant manifold for the system (\ref{FullODE}), which we write as a function of the phase variable: $\sigma(\psi,\alpha)$, for $\alpha \geq 0$ sufficiently small. 

Now, to understand the decay of perturbations from the invariant manifold, we write $s = \sigma(\psi,\alpha) + \rho$, where $\rho$ captures the deviation of $s$ from the invariant manifold. For a fixed $\delta > 0$, let us consider the spaces 
\[
	Y(\delta) = \{\rho:[0,\infty)\to \ell^1:\ \sup_{t \geq 0} \|\rho(t)\|_1 \leq 2\delta e^{\frac{a\lambda'(a)}{2}t}\}
\]  
along with associated norm
\begin{equation}\label{YNorm}
	\|\rho(t)\|_Y := \sup_{t \in[0,\infty)} \|\rho(t)\|_1. 
\end{equation}
Here we recall that $a\lambda'(a) < 0$, and hence the $\ell^1$-norm of elements in $Y$ decay exponentially in $t$. We present the following lemma.

\begin{lem}\label{lem:Y_Complete} 
	For every $\delta > 0$, $Y(\delta)$ is a complete with respect to the norm $\|\cdot\|_Y$.
\end{lem} 

\begin{proof}
	Let us begin by fixing $\delta > 0$. Let us take $\{\rho_n\}_{n=1}^\infty \subset Y(\delta)$ to be a Cauchy sequence. Then, by definition we have 
	\[
		\|\rho_n(t)\|_1 \leq \|\rho_n\|_Y \leq 2\delta e^{\frac{a\lambda'(a)}{2}t} \leq 2\delta,
	\]
	for all $t\in [0,\infty)$ and $n \geq 1$. Then, as in the proof of Lemma~\ref{lem:X_Complete}, uniformity of the norm in $t \geq 0$ implies the existence of a pointwise limit, denoted $\rho(t)$, so that $\rho_n(t) \to \rho(t)$ in $\ell^1$ for all $t \in [0,\infty)$. Furthermore, the proof of the decay of the $\ell^1$-norm of $\rho(t)$ in $t$ to ensure $\rho \in Y(\delta)$ follows through nearly identical arguments to those laid out in Lemma~\ref{lem:X_Complete}.
\end{proof} 

Throughout this subsection, as in the former, we will use the constant $\alpha_{X,2} > 0$ to represent the maximal value of $\alpha$ for which the invariant manifold $\sigma(\cdot,\alpha)$ exists. Following as in the previous subsection, we will let $\psi^*(t;\psi^0,\sigma+\rho)$ denote that solution to the initial value problem
\begin{equation}\label{PhaseIVP2}
	\begin{cases}
		\dot{\psi} = \alpha G(\sigma(\psi,\alpha)+\rho(t),\psi,\alpha), \\
		\psi(0) = \psi^0.
	\end{cases}
\end{equation}
Note again that we require the condition $\rho \in \tilde{Y}(\delta)$ for $\delta \leq \frac{a}{16}$ to guarantee that $\bar{r}_{i,j}(\alpha) + \sigma_{i,j}(\psi,\alpha) + \rho(t) \geq \frac{a}{8}$, to avoid the singularity in the phase equations when $\bar{r}_{i,j}(\alpha) + \sigma_{i,j}(\psi,\alpha) + \rho(t) = 0$. This leads to our first result which is similar to Lemma~\ref{lem:InvManPhase}.

\begin{lem}\label{lem:InvManPhase2} 
	There exists a constant $C_2 > 0$ such that for all $\delta\in[0,\frac{a}{16}]$, $\rho,\tilde{\rho} \in \tilde{Y}(\delta)$, $\psi \in \ell^1$ and $\alpha \in [0,\alpha_{X,2}]$ sufficiently small we have
	\begin{equation} \label{PhaseBnd3}
		\|\psi^*(t;\psi,\sigma+\rho) - \psi^*(t;\psi,\sigma+\tilde{\rho})\|_1 \leq e^{\alpha C_2t}\|\rho - \tilde{\rho}\|_Y,
	\end{equation}
	for all $t \geq 0$.
\end{lem}

\begin{proof}
	This proof proceeds in a similar way to that of (\ref{PhaseBnd2}) in Lemma~\ref{lem:InvManPhase}, and therefore many details will be omitted. Nearly identical manipulations to those undertaken in the proof of Lemma~\ref{lem:InvManPhase} yield the existence of a constant $C > 0$ such that
	\[
		\begin{split}
		|G_{i,j}(\sigma(\psi,\alpha) + \rho(t),\psi,\alpha) &- G_{i,j}(\sigma(\tilde{\psi},\alpha)+\tilde{\rho}(t),\tilde{\psi},\alpha)| \leq C\bigg(|\rho_{i,j}(t) + \tilde{\rho}_{i,j}(t)| +\sum_{i',j'}|\rho_{i',j'}(t) + \tilde{\rho}_{i',j'}(t)| \\ 
		&+ |\sigma_{i,j}(\psi,\alpha) + \sigma_{i,j}(\tilde{\psi},\alpha)| +\sum_{i',j'}|\sigma_{i',j'}(\psi,\alpha) + \sigma_{i',j'}(\tilde{\psi},\alpha)| \\
		&+ |\psi_{i,j} - \tilde{\psi}_{i',j'}| + \sum_{i',j'} |\psi_{i',j'} - \tilde{\psi}_{i',j'}|\bigg), 
		\end{split}
	\]  
	for all $\rho,\tilde{\rho} \in \tilde{Y}(\delta)$ with $\delta\in[0,\frac{a}{16}]$, $\psi,\tilde{\psi} \in \ell^1$ and $\alpha$ taken sufficiently small. Then, taking the sum over all $(i,j)\in\mathbb{Z}^2$ we obtain
	\[
		\begin{split}
			\|G(\sigma(\psi,\alpha) + \rho(t),\psi,\alpha) - G(\sigma(\tilde{\psi},\alpha)+\tilde{\rho}(t),\tilde{\psi},\alpha)\|_1 &\leq 5C\|\rho(t) - \tilde{\rho}(t)\|_1 + 5C\|\sigma(\psi,\alpha) - \sigma(\tilde{\psi},\alpha)\|_1 \\ 
			&+ CQ_1(\psi-\tilde{\psi}) \\
			&\leq 5C\|\rho(t) - \tilde{\rho}(t)\|_1 + (5\sqrt{\alpha}+1)C{\alpha}Q_1(\psi-\tilde{\psi}) \\ 
			&\leq 5C\|\rho(t) - \tilde{\rho}(t)\|_1 + 8(5\sqrt{\alpha}+1)C\|\psi - \tilde{\psi}\|_1,   
		\end{split}	
	\]
	where we have used the fact that $\|\sigma(\psi,\alpha) - \sigma(\tilde{\psi},\alpha)\|_1 \leq \sqrt{\alpha}Q_1(\psi-\tilde{\psi}) \leq 8\sqrt{\alpha}\|\psi-\tilde{\psi}\|_1$, which comes from both the properties of the invariant manifold and the inequalities of Lemma~\ref{lem:NormBnds}.
	
	Then, the integral form of the initial value problem (\ref{PhaseIVP2}) is given by
	\[
		\psi^*(t;\psi,\sigma+\rho) = \psi + \alpha\int_0^t G(\sigma(\psi^*(u;\psi,\sigma+\rho),\alpha) + \rho(u),\psi^*(u;\psi,\sigma+\rho),\alpha)\mathrm{d}u.
	\]
	Using this integral form and the above previously proven inequality we obtain
	\[
		\begin{split}
		\|\psi^*(t;\psi,\sigma+\rho) &- \psi^*(t;\psi,\sigma+\tilde{\rho})\|_1 \\ 
		&\leq \alpha\int_0^t \bigg[5C\|\rho(u) - \tilde{\rho}(u)\|_1 + 8(5\sqrt{\alpha}+1)C\|\psi^*(u;\psi,\sigma+\rho) - \psi^*(u;\psi,\sigma+\tilde{\rho})\|_1\bigg]\mathrm{d}u \\
		&\leq \alpha\int_0^t \bigg[5C\|\rho - \tilde{\rho}\|_Y + 8(5\sqrt{\alpha}+1)C\|\psi^*(u;\psi,\sigma+\rho) - \psi^*(u;\psi,\sigma+\tilde{\rho})\|_1\bigg]\mathrm{d}u \\
		&= 5C\alpha\|\rho - \tilde{\rho}\|_Yt + 8(5\sqrt{\alpha}+1)C\int_0^t \|\psi^*(u;\psi,\sigma+\rho) - \psi^*(u;\psi,\sigma+\tilde{\rho})\|_1\mathrm{d}u.
		\end{split}
	\]
	From here the bound (\ref{PhaseBnd3}) is obtained by an application of Gronwall's inequality, and follows as in the proof of (\ref{PhaseBnd2}). Therefore, we omit these final steps due to their redundancy and complete the proof of the lemma. 
\end{proof} 

Now, taking $s = \sigma + \rho$, we use the differential equation (\ref{FullODE}) and the fact that $\sigma$ is an invariant manifold of this differential equation to obtain the autonomous dynamical system governing the evolution of $\rho$:
\begin{equation}\label{rhoEqn}
	\dot{\rho} = F(\sigma(\psi^*(t;\psi,\sigma+\rho),\alpha) +\rho,\psi^*(t;\psi,\sigma+\rho),\alpha) - F(\sigma(\psi^*(t;\psi,\sigma+\rho),\alpha),\psi^*(t;\psi,\sigma+\rho),\alpha),
\end{equation}
which we use the definition of $F$ in (\ref{Fdefn}) to write explicitly as
\begin{equation}\label{rhoEqnij}
	\begin{split}
	\dot{\rho}_{i,j} = \alpha&\sum_{i',j'} \bigg[\rho_{i',j'}\cos(\bar{\theta}_{i',j'}(\alpha) + \psi_{i',j'}^*(t;\psi,\sigma+\rho) - \bar{\theta}_{i,j}(\alpha) - \psi_{i,j}^*(t;\psi,\sigma+\rho)) - \rho_{i,j}\bigg] \\ 
	&+ (\bar{r}_{i,j}(\alpha) + \sigma_{i,j}(\psi^*(t;\psi,\sigma+\rho),\alpha) + \rho_{i,j})\lambda(\bar{r}_{i,j}(\alpha) + \sigma_{i,j}(\psi^*(t;\psi,\sigma+\rho),\alpha) + \rho_{i,j}) \\ 
	&- (\bar{r}_{i,j}(\alpha) + \sigma_{i,j}(\psi^*(t;\psi,\sigma+\rho),\alpha))\lambda(\bar{r}_{i,j}(\alpha) + \sigma_{i,j}(\psi^*(t;\psi,\sigma+\rho),\alpha))
	\end{split}
\end{equation}
for each $(i,j) \in \mathbb{Z}^2$. Note that this differential equation has a steady-state solution $\rho = 0$, representing the flow on the invariant manifold since $\rho$ is being used to capture deviations from this manifold.

We may further recast (\ref{rhoEqn}) using the function $\tilde{F}$ introduced in (\ref{FTilde}) so that 
\[
	\dot{\rho} = a\lambda'(a)\rho + \tilde{F}(\sigma(\psi^*(t;\psi,\sigma+\rho),\alpha) +\rho,\psi^*(t;\psi,\sigma+\rho),\alpha) - \tilde{F}(\sigma(\psi^*(t;\psi,\sigma+\rho),\alpha),\psi^*(t;\psi,\sigma+\rho),\alpha),
\]
which in turn allows one to apply the variation of constants formula for the differential equation (\ref{rhoEqn}) with initial value $\rho(0) = \rho^0 \in \ell^1$ to obtain
\[
	\begin{split}
	\rho(t) = e^{a\lambda'(a)t}\rho^0 &+ \int_0^t e^{a\lambda'(a)(t-u)}\bigg[\tilde{F}(\sigma(\psi^*(u;\psi,\sigma+\rho)) +\rho(u),\psi^*(u;\psi,\sigma+\rho),\alpha)\\ 
	&- \tilde{F}(\sigma(\psi^*(u;\psi,\sigma+\rho)),\psi^*(u;\psi,\sigma+\rho),\alpha)\bigg]du.  
	\end{split}
\] 
Let us then define the mapping, denoted $T_2$, as 
\begin{equation}\label{T2Mapping} 
	\begin{split}
	T_2\rho(t) := e^{a\lambda'(a)t}\rho^0 &+ \int_0^t e^{a\lambda'(a)(t-u)}\bigg[\tilde{F}(\sigma(\psi^*(u;\psi,\sigma+\rho)) +\rho(u),\psi^*(u;\psi,\sigma+\rho),\alpha)\\ 
	&- \tilde{F}(\sigma(\psi^*(u;\psi,\sigma+\rho)),\psi^*(u;\psi,\sigma+\rho),\alpha)\bigg]du.
	\end{split}
\end{equation} 
so that fixed points of $T_2$ correspond to solutions of (\ref{rhoEqn}) with initial value $\rho(0) = \rho^0 \in \ell^1$.

\begin{lem}\label{lem:T2WellDefined} 
	There exists $\delta_1,\alpha_{Y,1} > 0$ such that for all $\alpha \in [0,\alpha_{Y,1}]$, $\delta \in (0,\delta_1]$ and $\rho^0 \in \ell^1$ with $\|\rho^0\|_1 \leq \delta$ we have that $T_2:Y(\delta) \to Y(\delta)$ is well-defined.
\end{lem}

\begin{proof}
	As in the proof of Lemma~\ref{lem:T1WellDefined}, we break the integrand up into smaller components to makes things more manageable, then put them back together at the end. To begin, we use the definition of $\tilde{F}$ along with (\ref{rhoEqnij}) to note that 
	\begin{equation}\label{rhoEqn2}
		\begin{split}
			&\tilde{F}(\sigma(\psi^*(t;\psi,\sigma+\rho)) +\rho(t),\psi^*(t;\psi,\sigma+\rho),\alpha) - \tilde{F}(\sigma(\psi^*(t;\psi,\sigma+\rho)),\psi^*(t;\psi,\sigma+\rho),\alpha) \\
			&= \alpha\sum_{i',j'} \bigg[\rho_{i',j'}\cos(\bar{\theta}_{i',j'}(\alpha) + \psi_{i',j'}^*(t;\psi,\sigma+\rho) - \bar{\theta}_{i,j}(\alpha) - \psi_{i,j}^*(t;\psi,\sigma+\rho)) - \rho_{i,j}\bigg] \\ 
			&+ (\bar{r}_{i,j}(\alpha) + \sigma_{i,j}(\psi^*(t;\psi,f+\rho),\alpha) + \rho_{i,j})\lambda(\bar{r}_{i,j}(\alpha) + \sigma_{i,j}(\psi^*(t;\psi,\sigma+\rho),\alpha) + \rho_{i,j}) \\ 
			&- (\bar{r}_{i,j}(\alpha) + \sigma_{i,j}(\psi^*(t;\psi,f+\rho),\alpha))\lambda(\bar{r}_{i,j}(\alpha) + \sigma_{i,j}(\psi^*(t;\psi,\sigma+\rho),\alpha)) - a\lambda'(a)\rho,	
		\end{split}
	\end{equation}
	where the only change from (\ref{rhoEqnij}) is that addition of the $-a\lambda'(a)\rho$ term at the end.
	
	Now, 
	\[
		|\rho_{i',j'}\cos(\bar{\theta}_{i',j'}(\alpha) + \psi_{i',j'}^*(t;\psi,\sigma+\rho) - \bar{\theta}_{i,j}(\alpha) - \psi_{i,j}^*(t;\psi,\sigma+\rho))| \leq |\rho_{i',j'}|,
	\]
	for all $(i,j)\in\mathbb{Z}^2$. This therefore gives that 
	\[
		\begin{split}
		\alpha\bigg|\sum_{i',j'} &\bigg[\rho_{i',j'}\cos(\bar{\theta}_{i',j'}(\alpha) + \psi_{i',j'}^*(t;\psi,\sigma+\rho) - \bar{\theta}_{i,j}(\alpha) - \psi_{i,j}^*(t;\psi,\sigma+\rho)) - \rho_{i,j}\bigg]\bigg| \leq 4\alpha|\rho_{i,j}| + \alpha\sum_{i',j'}|\rho_{i',j'}|.
		\end{split}
	\] 
	Furthermore, the remaining parts of (\ref{rhoEqn2}) can be compactly written as the function 
	\[
		(x_{i,j}+\rho_{i,j})\lambda(x_{i,j}+\rho_{i,j}) - x_{i,j}\lambda(x_{i,j}) - a\lambda'(a)\rho_{i,j}, 
	\] 
	where $x_{i,j} = \bar{r}_{i,j}(\alpha) + \sigma_{i,j}(\psi^*(t;\psi,\sigma+\rho),\alpha)$ for notational convenience. Then, since the terms $\bar{r}_{i,j}(\alpha) + \sigma_{i,j}(\psi^*(t;\psi,\sigma+\rho),\alpha)$ are uniformly bounded, a nearly identical argument to that employed in (\ref{C_lambda}) results in the existence of a constant $C'_\lambda > 0$ such that 
	\[
		|(x_{i,j}+\rho_{i,j})\lambda(x_{i,j}+\rho_{i,j}) - x_{i,j}\lambda(x_{i,j}) - a\lambda'(a)\rho_{i,j}| \leq C'_\lambda (\alpha + \delta)|\rho_{i,j}|,
	\]  
	for every $\delta \in [0,1]$ and $\alpha \leq \min\{\alpha^*,\frac{a}{4}\}$, uniformly in $(i,j)\in\mathbb{Z}^2$.
	
	Putting this all together then gives that 
	\[
		\begin{split}
		\|\tilde{F}&(\sigma(\psi^*(t;\psi,\sigma+\rho)) +\rho,\psi^*(t;\psi,\sigma+\rho),\alpha) - \tilde{F}(\sigma(\psi^*(t;\psi,\sigma+\rho)),\psi^*(t;\psi,\sigma+\rho),\alpha)\|_1 \\ 
		&\leq 4\alpha\sum_{i,j}|\rho_{i,j}| + \alpha\sum_{i,j}\sum_{i',j'}|\rho_{i',j'}| + C'_\lambda (\alpha + \delta)\sum_{i,j}|\rho_{i,j}| \\
		&= 4\alpha\|\rho(t)\|_1 + 4\alpha\|\rho\|_1 + C'_\lambda (\alpha + \delta)\|\rho\|_1 \\
		&= (8\alpha+ C'_\lambda\alpha + C'_\lambda\delta)\|\rho\|_1.  	
		\end{split}
	\]
	Then, this therefore implies that for any $\rho \in Y(\delta)$ we have 
	\[
		\begin{split}
			\|T_2\rho(t)\|_1 &\leq e^{a\lambda'(a)t}\|\rho^0\|_1 + (8\alpha + C'_\lambda\alpha + C'_\lambda\delta)\int_0^te^{a\lambda'(a)(t-u)}\|\rho(u)\|_1du \\
			&\leq e^{a\lambda'(a)t}\|\rho^0\|_1 + 2\delta(8\alpha + C'_\lambda\alpha + C'_\lambda\delta)\int_0^te^{a\lambda'(a)(t-u)}e^{\frac{a\lambda'(a)}{2}u}du \\
			&= e^{a\lambda'(a)t}\|\rho^0\|_1 + \frac{4\delta(8\alpha + C'_\lambda\alpha + C'_\lambda\delta)}{|a\lambda'(a)|}\bigg(e^{\frac{a\lambda'(a)}{2}t} - e^{a\lambda'(a)t}\bigg) \\
			&\leq e^{a\lambda'(a)t}\|\rho^0\|_1 + \frac{4\delta(8\alpha + C'_\lambda\alpha + C'_\lambda\delta)}{|a\lambda'(a)|}e^{\frac{a\lambda'(a)}{2}t}.
		\end{split}
	\]
	Then, taking $\|\rho^0\|_1 \leq \delta$, we can use the fact that $e^{a\lambda'(a)t} \leq e^{\frac{a\lambda'(a)}{2}t}$ for all $t \geq 0$ to see that for all $\delta \in (0,1]$ we have 
	\[
		\|T_2\rho(t)\|_1 \leq \delta\bigg(1 + \frac{\alpha(8 + C'_\lambda) + 4\delta C'_\lambda}{|a\lambda'(a)|}\bigg)e^{\frac{a\lambda'(a)}{2}t}	,
	\]
	and therefore one sees that upon taking 
	\[
		\alpha_{Y,1} := \min\bigg\{\alpha^*,\frac{a}{4},\frac{|a\lambda'(a)|}{16+2C'_\lambda}\bigg\}
	\]
	and
	\[
		\delta \leq \min\bigg\{\frac{|a\lambda'(a)|}{8C'_\lambda},1\bigg\},
	\]
	gives that 
	\[
		\|T_2\rho(t)\|_1 \leq 2\delta e^{\frac{a\lambda'(a)}{2}t},
	\]
	for all $\alpha \in [0,\alpha_{Y,1}]$ and $\delta > 0$ chosen appropriately small. Our choices of $\alpha$ and $\delta$ therefore show that $T_2$ maps elements of $Y(\delta)$ back into $Y(\delta)$, completing the proof.
\end{proof} 

\begin{lem}\label{lem:T2Contraction} 
	There exists $\delta_2,\alpha_{Y,2} > 0$ such that for all $\alpha \in [0,\alpha_{Y,2}]$, $\delta\in(0,\delta_2]$ and $\rho^0 \in \ell^1$ with $\|\rho^0\|_1 \leq \delta$ we have that $T_2$ is a contraction on $Y(\delta)$ with contraction constant at most $\frac{1}{2}$.
\end{lem}

\begin{proof}
	The proof uses nearly identical manipulations to those employed in the proof of Lemma~\ref{lem:T2WellDefined}, and therefore we merely focus on those aspects that differentiate it. First, to obtain Lipschitz properties of the integrand
	\begin{equation}\label{Integrand1}
		\tilde{F}(\sigma(\psi^*(u;\psi,\sigma+\rho)) +\rho(u),\psi^*(u;\psi,\sigma+\rho),\alpha) - \tilde{F}(\sigma(\psi^*(u;\psi,\sigma+\rho)),\psi^*(u;\psi,\sigma+\rho),\alpha),
	\end{equation}
	with respect to $\rho$ much of the manipulations follow in a similar way to that of Lemma~\ref{lem:T1Contraction}, except the term
	\[
		\begin{split}
		|&\rho_{i',j'}\cos(\bar{\theta}_{i',j'}(\alpha) + \psi_{i',j'}^*(t;\psi,\sigma+\rho) - \bar{\theta}_{i,j}(\alpha) - \psi_{i,j}^*(t;\psi,\sigma+\rho)) \\ 
		&- \tilde{\rho}_{i',j'}\cos(\bar{\theta}_{i',j'}(\alpha) + \psi_{i',j'}^*(t;\psi,\sigma+\tilde{\rho}) - \bar{\theta}_{i,j}(\alpha) - \psi_{i,j}^*(t;\psi,\sigma+\tilde{\rho}))|,
		\end{split}
	\]
	which requires one to utilize the results of Lemma~\ref{lem:InvManPhase2}. That is, for $C_3 > 0$, the constant guaranteed by Lemma~\ref{lem:InvManPhase2}, we obtain the bound
	\[
		\begin{split}
			|&\rho_{i',j'}\cos(\bar{\theta}_{i',j'}(\alpha) + \psi_{i',j'}^*(t;\psi,\sigma+\rho) - \bar{\theta}_{i,j}(\alpha) - \psi_{i,j}^*(t;\psi,\sigma+\rho)) \\ 
		&- \tilde{\rho}_{i',j'}\cos(\bar{\theta}_{i',j'}(\alpha) + \psi_{i',j'}^*(t;\psi,\sigma+\tilde{\rho}) - \bar{\theta}_{i,j}(\alpha) - \psi_{i,j}^*(t;\psi,\sigma+\tilde{\rho}))| \\
		&\leq |\rho_{i',j'}|\cdot|\cos(\bar{\theta}_{i',j'}(\alpha) + \psi_{i',j'}^*(t;\psi,\sigma+\rho)- \bar{\theta}_{i,j}(\alpha) - \psi_{i,j}^*(t;\psi,\sigma+\rho))  \\ 
		&- \cos(\bar{\theta}_{i',j'}(\alpha) + \psi_{i',j'}^*(t;\psi,\sigma+\tilde{\rho}) - \bar{\theta}_{i,j}(\alpha) - \psi_{i,j}^*(t;\psi,\sigma+\tilde{\rho}))| \\ 
		&+|\rho_{i',j'} - \tilde{\rho}_{i',j'}|\cdot|\cos(\bar{\theta}_{i',j'}(\alpha) + \psi_{i',j'}^*(t;\psi,\sigma+\tilde{\rho}) - \bar{\theta}_{i,j}(\alpha) - \psi_{i,j}^*(t;\psi,\sigma+\tilde{\rho}))| \\
		&\leq |\rho_{i',j'}| Q_1(\psi^*(t;\psi,\sigma+\rho) - \psi^*(t;\psi,\sigma+\tilde{\rho})) +|\rho_{i',j'} - \tilde{\rho}_{i',j'}| \\
		&\leq 8|\rho_{i',j'}|\|\psi^*(t;\psi,\sigma+\rho) - \psi^*(t;\psi,f+\tilde{\rho})\|_1 +|\rho_{i',j'} - \tilde{\rho}_{i',j'}| \\
		&\leq 8e^{\alpha C_3 t}|\rho_{i',j'}| \|\rho-\tilde{\rho}\|_Y +|\rho_{i',j'} - \tilde{\rho}_{i',j'}|.    	
		\end{split}
	\]
	for all $t \geq 0$. Then, taking the sum over $(i,j)\in\mathbb{Z}^2$ gives
	\[
		\begin{split}
			\sum_{i,j}\sum_{i',j'}|&\rho_{i',j'}\cos(\bar{\theta}_{i',j'}(\alpha) + \psi_{i',j'}^*(t;\psi,\sigma+\rho) - \bar{\theta}_{i,j}(\alpha) - \psi_{i,j}^*(t;\psi,\sigma+\rho)) \\ 
		&- \tilde{\rho}_{i',j'}\cos(\bar{\theta}_{i',j'}(\alpha) + \psi_{i',j'}^*(t;\psi,\sigma+\tilde{\rho}) - \bar{\theta}_{i,j}(\alpha) - \psi_{i,j}^*(t;\psi,\sigma+\tilde{\rho}))| \\
		&\leq 8e^{\alpha C_3 t} \|\rho-\tilde{\rho}\|_Y\sum_{i,j}\sum_{i',j'}|\rho_{i',j'}| + \sum_{i,j}\sum_{i',j'}|\rho_{i',j'} - \tilde{\rho}_{i',j'}| \\
		&\leq 4(8e^{\alpha C_3 t}\|\rho(t)\|_1 + 1)\|\rho-\tilde{\rho}\|_Y \\
		&\leq 4(16\delta e^{(\alpha C_3 +\frac{a\lambda'(a)}{2})t} + 1)\|\rho-\tilde{\rho}\|_Y,     	
		\end{split}
	\]
	where we have applied the fact that $\rho \in Y(\delta)$ implies $\|\rho(t)\|_1 \leq 2\delta e^{\frac{a\lambda'(a)}{2}t}$. Then, taking $\alpha$ appropriately small will guarantee that 
	\[
		\alpha C_3 +\frac{a\lambda'(a)}{2} < 0,	
	\]
	and hence the bound $4(16\delta e^{(\alpha C_3 +\frac{a\lambda'(a)}{2})t} + 1)$ is bounded uniformly in $t\geq 0$.  
	
	From here much of the work is the same as in Lemmas~\ref{lem:T1Contraction} and \ref{lem:T2WellDefined}. Upon obtaining Lipschitz bounds on the integrand (\ref{Integrand1}), we simply use the fact that the integral
	\[
		\int_0^t e^{a\lambda'(a)(t-u)}\mathrm{d}u
	\]
	is uniformly bounded in $t$ (since $a\lambda'(a)<0$), showing that $T_2$ can be made a contraction via appropriately small choices of $\alpha,\delta > 0$. Therefore we state this result without a full proof since it follows in a similar way to much of the work that has been done in this subsection and the one previous to it. 
\end{proof} 

Lemmas~\ref{lem:T2WellDefined} and \ref{lem:T2Contraction} combine to show that if we choose $\|\rho^0\|_1$ sufficiently small, we have that the solution $\rho(t)$ to (\ref{rhoEqn}) converges exponentially to $0$. This means that for the original system (\ref{FullODE}), if we start sufficiently close to the invariant manifold, we will converge exponentially fast to the invariant manifold. This therefore completes the proof of Theorem~\ref{thm:InvMan}.

\section{Stability on the Invariant Manifold}\label{sec:Thm2Proof} 

As previously noted, system (\ref{FullODE}) is a fast-slow system which is singular at $\alpha = 0$. Now that we have analyzed the fast component, $s$, we wish to inspect the slow component, $\psi$. We begin by using Theorem~\ref{thm:InvMan} to write the fast variable as $s = \sigma(\psi,\alpha) + \rho$. Furthermore, the work of the previous section implies that the specific evolution of $\rho(t)$ depends on the evolution of the phase variable, $\psi(t)$. Coupling the global Lipschitz properties of the function $G$ with respect to $\psi$ with the results of Lemmas~\ref{lem:InvManPhase} and \ref{lem:InvManPhase2} implies that for all $\psi^0 \in \ell^1$ a global solution $\psi(t)$ exists and belongs to $\ell^1$ for all $t \geq 0$. Hence, for any fixed $\psi^0 \in \ell^1$ we consider $\rho(t) = \rho(t;\psi^0)$ to be the solution of (\ref{rhoEqn}) with $\rho(0) = \rho^0$ such that $\|\rho^0\|_1 \leq \delta$. Then, the results of Theorem~\ref{thm:InvMan} imply that $\|\rho(t)\|_1 \leq 2\delta e^{-\beta t}$, for some $\beta > 0$, regardless of our choice of $\psi^0 \in \ell^1$.

Having now solved the fast variable equation, we therefore introduce the slow-time variable $\tau := \alpha t$, for $\alpha > 0$. This results in the system 
\begin{equation}\label{GSlow}
	\frac{\mathrm{d}}{\mathrm{d}\tau}\psi = G(\sigma(\psi,\alpha)+\rho(\alpha^{-1}\tau),\psi(\tau),\alpha),
\end{equation}
along with an initial condition $\psi(0) = \psi^0 \in \ell^1$, so that $\rho(t)$ is well-defined. We note that $\|\rho(\alpha^{-1}\tau)\|_1\leq2\delta e^{-\frac{\beta}{\alpha}\tau}$, for $\alpha > 0$ so long as $\|\rho^0\|_1 \leq \delta$, for $\delta > 0$ taken sufficiently small. It is therefore through system (\ref{GSlow}) which we plan to investigate the stability of $\psi = 0$ in this section. As in the previous section, we will assume that Hypotheses~\ref{hyp:LambdaOmega}-\ref{hyp:LinearPhase} hold throughout, but do not explicitly say this in the statement of all results in this section.

\subsection{Integral Bounds}\label{subsec:IntegralBounds} 

In this section we provide two important integral bounds which will be used throughout the proof of the stability on the invariant manifold. The first of which is a restatement of a result which can be found in \cite{IntegralLemma} and therefore is stated without proof, whereas the second will be stated with proof since an appropriate reference could not be found. 

\begin{lem}[{\em \cite{IntegralLemma}, \S 3, Lemma 3.2}] \label{lem:IntegralLemma} 
	Let $\gamma_1, \gamma_2$ be positive real numbers. If $\gamma_1,\gamma_2 \neq 1$ or if $\gamma_1 = 1 < \gamma_2$ then there exists a $C_{\gamma_1,\gamma_2} > 0$ (continuously depending on $\gamma_1$ and $\gamma_2$) such that
	\begin{equation}
		\int_0^\tau (1 + \tau - u)^{- \gamma_1}(1 + u)^{-\gamma_2}\mathrm{d}u \leq C_{\gamma_1,\gamma_2} (1 + \tau)^{-\min\{\gamma_1 + \gamma_2 - 1, \gamma_1, \gamma_2\}},
	\end{equation}  	 
	for all $\tau \geq 0$.
\end{lem}

\begin{lem} \label{lem:IntegralLemmaExp} 
	If $\beta,\gamma$ are positive constants, there exists a constant $C_{\gamma,\beta} > 0$ (continuously depending on $\beta$ and $\gamma$) such that 
	\[
		\int_0^\tau(1 + \tau - u)^{-\gamma}e^{-\frac{\beta}{\alpha}u}\mathrm{d}u \leq \alpha C_{\gamma,\beta} (1 + \tau)^{-\gamma},
	\]
	for all $\tau \geq 0$ and $\alpha \in (0,1]$.
\end{lem}

\begin{proof}
	We begin by integrating by parts to obtain
	\begin{equation}\label{IntBndExp}
		\int_0^\tau(1 + \tau - u)^{-\gamma}e^{-\frac{\beta}{\alpha}u}\mathrm{d}u = \frac{\alpha}{\beta}(1+\tau)^{-\gamma} -\frac{\alpha}{\beta}e^{-\frac{\beta}{\alpha}\tau} + \frac{\alpha \gamma}{\beta}\int_0^\tau(1 + \tau - u)^{-\gamma-1}e^{-\frac{\beta}{\alpha}u}\mathrm{d}u.  	
	\end{equation}
	Then, the rightmost integral can be bounded as
	\[
		\begin{split}
			\int_0^\tau(1 + \tau - u)^{-\gamma-1}e^{-\frac{\beta}{\alpha}u}\mathrm{d}u &\leq (1 + \frac{\tau}{2})^{-\gamma-1}\int_0^\frac{\tau}{2}e^{-\frac{\beta}{\alpha}u}\mathrm{d}u + e^{-\frac{\beta}{2\alpha}\tau}\int_\frac{\tau}{2}^\tau(1 + \tau - u)^{-\gamma-1}\mathrm{d}u \\
			&\leq \frac{\alpha}{\beta}(1 + \frac{\tau}{2})^{-\gamma-1}\underbrace{[1 - e^{-\frac{\beta}{2\alpha}\tau}]}_{\leq 1} + \frac{1}{\gamma}e^{-\frac{\beta}{2\alpha}\tau}\underbrace{[1 - (1+\frac{\tau}{2})^{-\gamma}]}_{\leq 1} \\
			&\leq \frac{\alpha}{\beta}(1 + \frac{\tau}{2})^{-\gamma-1} + \frac{1}{\gamma}e^{-\frac{\beta}{2\alpha}\tau}.	
		\end{split}
	\]
	Putting this bound back into (\ref{IntBndExp}) we arrive at
	\[
		\begin{split}
		\int_0^\tau(1 + \tau - u)^{-\gamma}e^{-\frac{\beta}{\alpha}u}\mathrm{d}u &\leq \frac{\alpha}{\beta}(1+\tau)^{-\gamma} -\frac{\alpha}{\beta}e^{-\frac{\beta}{\alpha}\tau} + \frac{\alpha^2 \gamma}{\beta^2}(1 + \frac{\tau}{2})^{-\gamma-1} + \frac{\alpha}{\beta}e^{-\frac{\beta}{2\alpha}\tau}	 \\
		&= \frac{\alpha}{\beta}(1+\tau)^{-\gamma} + \frac{\alpha^2 \gamma}{\beta^2}(1 + \frac{\tau}{2})^{-\gamma-1} + \frac{\alpha}{\beta}\underbrace{[e^{-\frac{\beta}{2\alpha}\tau} - e^{-\frac{\beta}{\alpha}\tau}]}_{\leq 0} \\
		&\leq \frac{\alpha}{\beta}(1+\tau)^{-\gamma} + \frac{\alpha^2 \gamma}{\beta^2}(1 + \frac{\tau}{2})^{-\gamma-1}. 
		\end{split}
	\]
	Therefore, when $0 < \alpha \leq 1$ we can find a $C > 0$, independent of $\alpha$ and $\tau$, so that 
	\[
		\frac{\alpha^2 \gamma}{\beta^2}(1 + \frac{\tau}{2})^{-\gamma-1} \leq \frac{\alpha \gamma}{\beta^2} C(1 + \tau)^{-\gamma}
	\]
	for all $\tau \geq 0$, thus completing the proof.
\end{proof} 

\begin{rmk}
	We note that the constant $C_{\gamma_1,\gamma_2}$ used in the statement of Lemma~\ref{lem:IntegralLemma} depends continuously on $\gamma_1,\gamma_2$. Hence, if Lemma~\ref{lem:IntegralLemma} is applied for a range of $\gamma_1,\gamma_2$ we are able to uniformly bound $C_{\gamma_1,\gamma_2}$ provided that the full range of $\gamma_1,\gamma_2$ belong to a compact set. This will be the case throughout the following sections. The same statement applies to the constant $C_{\gamma,\beta}$ from Lemma~\ref{lem:IntegralLemmaExp}.
\end{rmk}

\subsection{Semigroup Decay}\label{subsec:Linearization} 

Let us consider the linear operator $\tilde{L}_\alpha$, parametrized by $\alpha$, acting upon the sequences $x = \{x_{i,j}\}_{(i,j)\in\mathbb{Z}^2}$ by 
\begin{equation} \label{LinearPhase2}
	[\tilde{L}_\alpha x ]_{i,j} = \sum_{i',j'} \frac{\bar{r}_{i',j'}(\alpha)}{\bar{r}_{i,j}(\alpha)}\cos(\bar{\theta}_{i',j'}(\alpha) - \bar{\theta}_{i,j}(\alpha))(x_{i',j'} - x_{i,j}),
\end{equation}
for all $(i,j) \in \mathbb{Z}^2$. The following proposition details how the decay properties of the semigroup $e^{L_\alpha t}$, where $L_\alpha$ is defined in (\ref{LinearPhase}), can be extended to give decay properties of the semigroup $e^{\tilde{L}_\alpha t}$. 

\begin{prop} \label{prop:LinearPhaseDecay} 
	There exists a constant $\tilde{C}_L > 0$ such that for all $x_0 \in \ell^1$, $\alpha \in [0,\alpha^*]$, and $t \geq 0$ we have 
	\begin{equation}\label{LAlphaDecay2}
		\begin{split}
			&\|e^{\tilde{L}_\alpha t} x_0\|_p \leq \tilde{C}_L(1 + t)^{-1 + \frac{1}{p}}\|x_0\|_1, \\
			&Q_p(e^{\tilde{L}_\alpha t}x_0) \leq \tilde{C}_L(1 + t)^{-\min\{2 - \frac{1}{p},2-\frac{1}{q^*}\}}\|x_0\|_1,
		\end{split}
	\end{equation}
	where $e^{\tilde{L}_\alpha t}$ is the semi-group with infinitesimal generator given by $\tilde{L}_\alpha$.	
\end{prop} 

We see that the $\ell^p$ norm decay of the semigroup $e^{\tilde{L}_\alpha}$ is asymptotically equivalent to those of the semigroup $e^{\tilde{L}_\alpha}$ guaranteed by Hypothesis~\ref{hyp:LinearPhase}, but the $Q_p$ semi-norm decay has a slight adjustment for $p \geq q^*$. In Lemma~\ref{lem:q^*Decay} we will see that this is a consequence of Hypothesis~\ref{hyp:pStar}, and moreover, if $p^* = 1$ the decay with respect to the $Q_p$ semi-norms are asymptotically equivalent to those of the semigroup $e^{\tilde{L}_\alpha}$. To emphasize where these differences in decay rates come from, the proof of Proposition~\ref{prop:LinearPhaseDecay} is broken down into the following series of results.

\begin{lem} \label{lem:AsymBnd} 
	There exists a constant $C_{q^*} > 0$ such that for all $x\in\ell^1$ and $\alpha \in [0,\alpha^*]$ we have
	\[
		\|(\tilde{L}_{\alpha} - L_\alpha) x\|_1 \leq C_{q^*}Q_{q^*}(x),
	\]
	where $q^*$ is the H\"older conjugate of $p^*$ from Hypothesis~\ref{hyp:pStar}.
\end{lem}

\begin{proof}
	Using the definitions of $\tilde{L}_{\alpha}$ and $L_{\alpha}$, we find that for any $x = \{x_{i,j}\}_{(i,j)\in\mathbb{Z}^2} \in \ell^1$ we have 
	\[
		[(\tilde{L}_{\alpha} - L_\alpha) x]_{i,j} = \sum_{i',j'} \bigg(\frac{\bar{r}_{i',j'}(\alpha)}{\bar{r}_{i,j}(\alpha)} - 1\bigg)\cos(\bar{\theta}_{i',j'}(\alpha) - \bar{\theta}_{i,j}(\alpha))(x_{i',j'} - x_{i,j}),	
	\]
	for all $(i,j)\in\mathbb{Z}^2$ and $\alpha \in [0,\alpha^*]$. Then, using the uniform boundedness of cosine we get 
	\[
		|[(\tilde{L}_{\alpha} - L_\alpha) x]_{i,j}| \leq \sum_{i',j'} \bigg|\frac{\bar{r}_{i',j'}(\alpha)}{\bar{r}_{i,j}(\alpha)} - 1\bigg||x_{i',j'} - x_{i,j}|. 	
	\] 
	Then, taking the sum over all $(i,j)$ we arrive at
	\[
		\begin{split}
			\|(\tilde{L}_{\alpha} - L_\alpha) x\|_1 &= \sum_{(i,j)\in\mathbb{Z}^2}|[(\tilde{L}_{\alpha} - L_\alpha) x]_{i,j}| \\
			&\leq \sum_{(i,j)\in\mathbb{Z}^2}\sum_{i',j'} \bigg|\frac{\bar{r}_{i',j'}(\alpha)}{\bar{r}_{i,j}(\alpha)} - 1\bigg||x_{i',j'} - x_{i,j}|, \\
			&\leq \sum_{(i,j)\in\mathbb{Z}^2} \bigg(\sum_{i',j'}\bigg|\frac{\bar{r}_{i',j'}(\alpha)}{\bar{r}_{i,j}(\alpha)} - 1\bigg|^{p^*}\bigg)^{\frac{1}{p^*}}\bigg(\sum_{i',j'}|x_{i',j'} - x_{i,j}|^{q^*}\bigg)^\frac{1}{q^*} \\
			&\leq \bigg(\sum_{(i,j)\in\mathbb{Z}^2}\sum_{i',j'}\bigg|\frac{\bar{r}_{i',j'}(\alpha)}{\bar{r}_{i,j}(\alpha)} - 1\bigg|^{p^*}\bigg)^\frac{1}{p^*}Q_{q^*}(x) \\
		\end{split}
	\]
	where we have applied H\"older's inequality twice, once for the nearest-neighbour summation and once for the full summation over the full set of indices. Finally, from Hypothesis~\ref{hyp:pStar} we have that   
	\[
		\bigg(\sum_{(i,j)\in\mathbb{Z}^2}\sum_{i',j'}\bigg|\frac{\bar{r}_{i',j'}(\alpha)}{\bar{r}_{i,j}(\alpha)} - 1\bigg|^{p^*}\bigg)^\frac{1}{p^*} \leq C_\mathrm{sol}^\frac{1}{p^*} < \infty,
	\]
	completing the proof.
\end{proof} 

Now, if $x(t)$ is a solution to the differential equation
\begin{equation}\label{LinearPhaseODE}
	\dot{x} = \tilde{L}_\alpha x
\end{equation}
with initial condition $x(0) = x_0 \in \ell^1$, then this solution $x(t)$ is such that $x(t) = e^{\tilde{L}_\alpha t}x_0$, and hence understanding the decay of solutions to (\ref{LinearPhaseODE}) leads to the proof of Proposition~\ref{prop:LinearPhaseDecay}. Trivially we have that
\[
	\dot{x}(t) = L_\alpha x(t) + (\tilde{L}_\alpha - L_\alpha)x(t),
\]
and using the variation of constants formula we obtain the equivalent integral form of the ordinary differential equation (\ref{LinearPhaseODE}), given as
\begin{equation}\label{PhaseIntForm}
	x(t) = e^{L_\alpha t}x_0 + \int_0^t e^{L_\alpha (t-s)}(\tilde{L}_\alpha - L_\alpha)x(s)ds,
\end{equation}
where we again recall that $e^{L_{\alpha}t}$ is the semigroup with infinitesimal generator $L_{\alpha}$ with the decay properties given in Hypothesis~\ref{hyp:LinearPhase}. We now use the integral form (\ref{PhaseIntForm}) to prove Proposition~\ref{prop:LinearPhaseDecay}.

\begin{lem} \label{lem:q^*Decay} 
	There exists a constant $C_Q > 0$ such that for every $x_0 \in \ell^1$ and $\alpha\in[0,\alpha^*]$, the solution $x(t) = e^{\tilde{L}_\alpha t}x_0$ to (\ref{LinearPhaseODE}) with initial condition $x(0) = x_0$ satisfies
	\[
		Q_{q^*}(x(t)) \leq C_Q(1 + t)^{-2 +\frac{1}{q^*}}\|x_0\|_1,
	\]
	for all $t \geq 0$.
\end{lem}

\begin{proof}
	Through straightforward manipulations of the integral form (\ref{PhaseIntForm}) one finds that 
	\[
		Q_{q^*}(x(t)) \leq Q_{q^*}(e^{L_\alpha t}x_0) + \int_0^t Q_{q^*}(e^{L_\alpha (t-s)}(\tilde{L}_\alpha - L_\alpha)x(s))ds.
	\]
	Then, using Hypothesis~\ref{hyp:LinearPhase} we obtain
	\[
		\begin{split}
			Q_{q^*}(x(t)) &\leq C_L(1 + t)^{-2 +\frac{1}{q^*}}\|x_0\|_1 + C_L\int_0^t (1 + t - s)^{-2 +\frac{1}{q^*}}\|(\tilde{L}_\alpha - L_\alpha)x(s)\|_1ds \\
			&\leq C_L(1 + t)^{-2 +\frac{1}{q^*}}\|x_0\|_1 + C_LC_{q^*}\int_0^t (1 + t -s)^{-2 +\frac{1}{q^*}}Q_{q^*}(x(s))ds,  
		\end{split}		
	\]
	where we have applied the results of Lemma~\ref{lem:AsymBnd}. We now apply Gronwall's inequality to see that
	\[
		\begin{split}
		Q_{q^*}(x(t)) \leq &C_L(1 + t)^{-2 +\frac{1}{q^*}}\|x_0\|_1 \\ 
		&+ C_L^2C_{q^*}\|x_0\|_1\int_0^t (1 + t -s)^{-2 +\frac{1}{q^*}}(1 + s)^{-2 + \frac{1}{q^*}}e^{C_LC_{q^*}\int_s^t(1 + t -r)^{-2 +\frac{1}{q^*}}dr}ds. 	
		\end{split}
	\]
	
	Now, from Hypothesis~\ref{hyp:pStar} we have $1 \geq 1 + \frac{1}{p^*} = 2 - \frac{1}{q^*}$, which in turn implies that $-2 +\frac{1}{q^*} < -1$. Hence, there exists a uniform upper bound $C_\mathrm{exp} > 0$ so that for all $t \geq s \geq 0$ we have 
	\[
		e^{C_LC_{q^*}\int_s^t(1 + t -r)^{-2 +\frac{1}{q^*}}dr} \leq C_\mathrm{exp}. 
	\] 
	Then, combining this bound with the result of Lemma~\ref{lem:IntegralLemma} we find that
	\[
		\int_0^t (1 + t -s)^{-2 + \frac{1}{q^*}}(1 + s)^{-2 + \frac{1}{q^*}}e^{C_LC_{q^*}\int_s^t(1 + t -r)^{-2 + \frac{1}{q^*}}dr}ds \leq C_\mathrm{exp}C_{2 - \frac{1}{q^*},2 - \frac{1}{q^*}}(1 + t)^{-2 + \frac{1}{q^*}}. 	
	\]
	Putting this all together therefore gives
	\[
		Q_{q^*}(x(t)) \leq C_L(1 + C_\mathrm{exp}C_L^2C_{q^*}C_{2 - \frac{1}{q^*},2 - \frac{1}{q^*}})(1 + t)^{-2 + \frac{1}{q^*}}\|x_0\|_1,	
	\]
	This completes the proof.
\end{proof} 

The results of Proposition~\ref{prop:LinearPhaseDecay} now follow from the results of Lemma~\ref{lem:q^*Decay}. The reason for this is that from (\ref{PhaseIntForm}) we can obtain 
\[
	\|x(t)\|_p \leq \|(e^{L_\alpha t}x_0\|_p + \int_0^t \|e^{L_\alpha (t-s)}(\tilde{L}_\alpha - L_\alpha)x(s)\|_p ds.
\]
Proceeding as in the proof of Lemma~\ref{lem:q^*Decay} will give the bound
\[
	\|x(t)\|_p \leq C_L(1+t)^{-1 + \frac{1}{p}}\|x_0\|_1 + C_L^2C_{q^*}C_\mathrm{exp}\|x_0\|_1\int_0^t(1 + t -s)^{-1 + \frac{1}{p}}(1 + s)^{-2 + \frac{1}{q^*}} ds, 	
\]
which we can obtain the appropriate bound by applying Lemma~\ref{lem:IntegralLemma}. Similar manipulations can be done for the $Q_p$ semi-norms, leading to the proof of Proposition~\ref{prop:LinearPhaseDecay}.

\subsection{Nonlinear Estimates}\label{subsec:Estimates} 

Here we provide Lipschitz estimates on the function $G(s,\psi,\alpha)$ in terms of $(s,\psi)\in\ell^1\times\ell^1$, uniformly in $\alpha \in [0,\alpha^*]$. Recall that $G(0,0,\alpha) = 0$ for all $\alpha \in [0,\alpha^*]$ and $\omega_1$ is assumed to be identically zero. We introduce the functions $G^{(1)}$ and $G^{(2)}$ given by 
\begin{equation}\label{GSubfunctions}
	\begin{split}
		G^{(1)}_{i,j}(s,\psi,\alpha) &= \sum_{i',j'}\bigg(\frac{\bar{r}_{i',j'}(\alpha) + s_{i',j'}}{\bar{r}_{i,j}(\alpha) + s_{i,j}} - \frac{\bar{r}_{i',j'}(\alpha)}{\bar{r}_{i,j}(\alpha)}\bigg)\sin(\bar{\theta}_{i',j'}(\alpha) + \psi_{i',j'} - \bar{\theta}_{i,j}(\alpha) - \psi_{i,j}), \\
		G^{(2)}_{i,j}(\psi,\alpha) &= \sum_{i',j'} \bigg[\frac{\bar{r}_{i',j'}(\alpha)}{\bar{r}_{i,j}(\alpha)}\bigg(\sin(\bar{\theta}_{i',j'}(\alpha) + \psi_{i',j'} - \bar{\theta}_{i,j}(\alpha)-\psi_{i,j}) - \sin(\bar{\theta}_{i',j'}(\alpha)- \bar{\theta}_{i,j}(\alpha)) \\
		&\quad -\cos(\bar{\theta}_{i',j'}(\alpha)- \bar{\theta}_{i,j}(\alpha))(\psi_{i',j'} - \psi_{i,j}) \bigg)\bigg], 
	\end{split}
\end{equation}
so that 
\[
	G(s,\psi,\alpha) = G(s,\psi,\alpha) - G(0,0,\alpha) = \tilde{L}_{\alpha}\psi + G^{(1)}_{i,j}(s,\psi,\alpha) + G^{(2)}_{i,j}(\psi,\alpha) 
\]
We now obtain Lipschitz estimates on the functions $G^{(1)}$ and $G^{(2)}$.

\begin{lem}\label{lem:G1} 
	There exists a constant $C_{G,1} > 0$ such that for all $\psi \in \ell^1$, $\alpha \in [0,\alpha^*]$, and $s\in\ell^1$ with $\|s\|_\infty \leq \frac{3a}{8}$ we have 
	\[
		\|G^{(1)}(s,\psi,\alpha)\|_1 \leq C_{G,1}(\|s\|_{q^*} + \|s\|_2 Q_2(\psi)).
	\]
\end{lem} 

\begin{proof}
	We remark that the restriction $\|s\|_\infty \leq \frac{3a}{8}$ guarantees that the terms $\bar{r}_{i,j}(\alpha) + s_{i,j}$ are uniformly bounded in absolute value away from zero. Hence, there exists a Lipschitz constant $K_1 > 0$ such that 
	\[
		\bigg|\frac{\bar{r}_{i',j'}(\alpha) + s_{i',j'}}{\bar{r}_{i,j}(\alpha) + s_{i,j}} - \frac{\bar{r}_{i',j'}(\alpha)}{\bar{r}_{i,j}(\alpha)}\bigg| \leq K_1(|s_{i,j}| + |s_{i',j'}|),		
	\]   	
	for all $\alpha \in [0,\alpha^*]$ and $(i,j)\in\mathbb{Z}^2$. Similarly, we get the following estimate 
	\[
		\begin{split}
			|\sin(\bar{\theta}_{i',j'}(\alpha)& + \psi_{i',j'} - \bar{\theta}_{i,j}(\alpha) - \psi_{i,j})| \\ 
			&\leq |\sin(\bar{\theta}_{i',j'}(\alpha) + \psi_{i',j'} - \bar{\theta}_{i,j}(\alpha) - \psi_{i,j}) - \sin(\bar{\theta}_{i',j'}(\alpha)- \bar{\theta}_{i,j}(\alpha)| + |\sin(\bar{\theta}_{i',j'}(\alpha) - \bar{\theta}_{i,j}(\alpha))| \\
			&\leq |\psi_{i',j'} - \psi_{i,j}| + |\sin(\bar{\theta}_{i',j'}(\alpha) - \bar{\theta}_{i,j}(\alpha))| 
		\end{split}
	\]
	for all $(i,j)$ since sine is globally Lipschitz with Lipschitz constant one.	
	
	Putting this all together gives
	\begin{equation}\label{G1Bnd1}
		\begin{split}
			\|G^{(1)}(s,\psi,\alpha)\|_1 &\leq \sum_{(i,j)\in\mathbb{Z}^2}\sum_{i',j'} \bigg|\frac{\bar{r}_{i',j'}(\alpha) + s_{i',j'}}{\bar{r}_{i,j}(\alpha) + s_{i,j}}- \frac{\bar{r}_{i',j'}(\alpha)}{\bar{r}_{i,j}(\alpha)} \bigg||\sin(\bar{\theta}_{i',j'}(\alpha) + \psi_{i',j'} - \bar{\theta}_{i,j}(\alpha) - \psi_{i,j})| \\
			&\leq K_1\sum_{(i,j)\in\mathbb{Z}^2}\sum_{i',j'}(|s_{i,j}| + |s_{i',j'}|)|\psi_{i',j'} - \psi_{i,j}| \\ 
			&+ \bigg|\frac{\bar{r}_{i',j'}(\alpha) + s_{i',j'}}{\bar{r}_{i,j}(\alpha) + s_{i,j}}- \frac{\bar{r}_{i',j'}(\alpha)}{\bar{r}_{i,j}(\alpha)} \bigg||\sin(\bar{\theta}_{i',j'}(\alpha) - \bar{\theta}_{i,j}(\alpha))| \\
			&=  K_1\sum_{(i,j)\in\mathbb{Z}^2}\sum_{i',j'}(|s_{i,j}| + |s_{i',j'}|)(|\psi_{i',j'} - \psi_{i,j}|) \\ 
			&\quad+  \sum_{(i,j)\in\mathbb{Z}^2}\sum_{i',j'}\bigg|\frac{\bar{r}_{i',j'}(\alpha)}{\bar{r}_{i,j}(\alpha)} - 1\bigg|(|s_{i,j}| + |s_{i',j'}|)|\sin(\bar{\theta}_{i',j'}(\alpha) - \bar{\theta}_{i,j}(\alpha))|, 
		\end{split}
	\end{equation}
	uniformly in $\alpha \in [0,\alpha^*]$. Now, through a simple application of H\"older's inequality one obtains 
	\begin{equation}\label{G1Bnd2}
		\sum_{(i,j)\in\mathbb{Z}^2}\sum_{i',j'}(|s_{i,j}| + |s_{i',j'}|)(|\psi_{i',j'} - \psi_{i,j}|) \leq 8\|s\|_2 Q_2(\psi), 
	\end{equation}
	since each element has exactly four nearest-neighbours. Then, since sine is uniformly bounded we may apply H\"older's inequality and Hypothesis~\ref{hyp:pStar} to obtain the bound 
	\begin{equation}\label{G1Bnd3}
		\begin{split}
			\sum_{(i,j)\in\mathbb{Z}^2}\sum_{i',j'}\bigg|\frac{\bar{r}_{i',j'}(\alpha)}{\bar{r}_{i,j}(\alpha)} - 1\bigg|(|s_{i,j}|& + |s_{i',j'}|)|\sin(\bar{\theta}_{i',j'}(\alpha) - \bar{\theta}_{i,j}(\alpha))| \leq 8C_\mathrm{sol}^{p^*}\|s\|_{q^*},
		\end{split}
	\end{equation}
	where we have again used the fact that each element has exactly four nearest-neighbours. Therefore, the bounds (\ref{G1Bnd2}) and (\ref{G1Bnd3}) together with (\ref{G1Bnd1}) give the bound stated in the lemma.  
\end{proof} 

\begin{lem}\label{lem:G2} 
	There exists a constant $C_{G,2} > 0$ such that for all $\psi \in \ell^1$ and $\alpha \in [0,\alpha^*]$ we have 
	\[
		\|G^{(2)}(\psi,\alpha)\|_1 \leq C_{G,2}Q_2^2(\psi).
	\]
\end{lem} 

\begin{proof}
	Begin by recalling that $\frac{a}{2} \leq \bar{r}_{i,j}(\alpha) \leq \frac{3a}{2}$ for all $(i,j)\in\mathbb{Z}^2$. This implies that 
	\[
		\frac{\bar{r}_{i',j'}(\alpha)}{\bar{r}_{i,j}(\alpha)} \leq 3,	
	\]
	and hence 
	\[
		\begin{split}
		\|G^{(2)}(\psi,\alpha)\|_1 \leq 3\sum_{(i,j)\in\mathbb{Z}^2}\sum_{i',j'} &|\sin(\bar{\theta}_{i',j'}(\alpha) + \psi_{i',j'} - \bar{\theta}_{i,j}(\alpha)-\psi_{i,j}) - \sin(\bar{\theta}_{i',j'}(\alpha)- \bar{\theta}_{i,j}(\alpha))\\  
		& -\cos(\bar{\theta}_{i',j'}(\alpha)- \bar{\theta}_{i,j}(\alpha))(\psi_{i',j'} - \psi_{i,j}) |.
		\end{split}
	\]
	The bound stated in the lemma now follows through a simple application of Taylor's theorem as in \cite[Lemma~5.1]{Bramburger3}.
\end{proof} 

We state the following result as a corollary of the previous three lemmas for the convenience of citing these results in the following section. 

\begin{cor}\label{cor:GBnd} 
	There exists a constant $C_{G} > 0$ such that for all $\psi \in \ell^1$, $\alpha \in [0,\alpha^*]$, and $s\in\ell^1$ with $\|s\|_\infty \leq \frac{3a}{8}$ we have 
	\[
		\|G(s,\psi,\alpha) - \tilde{L}_\alpha\psi\|_1 \leq C_{G}(\|s\|_{q^*} + \|s\|_2 Q_2(\psi) + Q_2^2(\psi)).
	\]	
\end{cor}

\subsection{Nonlinear Stability of the Phase Components}\label{subsec:PhaseStability} 

In this subsection we use the linear and nonlinear estimates of the previous two subsections to prove Theorem~\ref{thm:PhaseDecay}. We write (\ref{GSlow}) as 
\begin{equation}\label{GSlow2}
	\frac{\mathrm{d}}{\mathrm{d}\tau}\psi = \tilde{L}_\alpha\psi + [G(\sigma(\psi(\tau),\alpha)+\rho(\alpha^{-1}\tau),\psi(\tau),\alpha) - \tilde{L}_\alpha\psi(\tau)],
\end{equation}
along with the initial condition $\psi(0) = \psi^0 \in \ell^1$. Then, using the variation of constants formula, (\ref{GSlow2}) can we written equivalently as
\begin{equation}\label{GSlowInt}
	\psi(t) = e^{\tilde{L}_\alpha t}\psi^0 + \int_0^t e^{\tilde{L}_\alpha (t-u)}[G(\sigma(\psi(u),\alpha)+\rho(\alpha^{-1}u),\psi(u),\alpha) - \tilde{L}_\alpha\psi(u)]\mathrm{d}u. 
\end{equation}
It will be the equation (\ref{GSlowInt}) which will remain the focus throughout this subsection. We note that we can no longer take $\alpha = 0$ since the temporal variable $\alpha^{-1}\tau$ will be undefined, and therefore we only consider $\alpha > 0$ and small.  

Let us now define the mapping, denoted $T_3$, by 
\begin{equation}\label{UMapping}
	T_3\psi(t) = e^{\tilde{L}_\alpha t}\psi^0 + \int_0^t e^{\tilde{L}_\alpha (t-u)}[G(\sigma(\psi(u),\alpha)+\rho(\alpha^{-1}u),\psi(u),\alpha) - \tilde{L}_\alpha\psi(u)]\mathrm{d}u,
\end{equation}	
so that fixed points of $T_3$ are exactly the solutions of (\ref{GSlow2}) via the equivalent formulation (\ref{GSlowInt}). Then, prior to stating our results, we note that for all $\psi \in \ell^1$ and $p \in [1,\infty]$ we have 
\begin{equation}\label{PhaseStabilityBnd1}
	\|\sigma(\psi,\alpha)+\rho(t)\|_p \leq \|\sigma(\psi,\alpha)\|_p + \|\rho(t)\|_p \leq \|\sigma(\psi,\alpha)\|_p + \|\rho(t)\|_1,
\end{equation}
following from the Lipschtz properties of the invariant manifold and the monotonicity of $\ell^p$ norms. This leads to the following lemma.

\begin{lem} \label{lem:QPhaseBnd} 
	Let $\tau_0 > 0$ be arbitrary and assume that 
	\begin{equation}\label{Q2PhaseBnd}
		Q_2(\psi(\tau)) \leq 2\varepsilon\tilde{C}_L(1 + \tau)^{-\min\{\frac{3}{2},2 - \frac{1}{q^*}\}}
	\end{equation}
	and
	\begin{equation}\label{Qq*PhaseBnd}
		Q_{q^*}(\psi(\tau)) \leq 2\varepsilon\tilde{C}_L(1 + \tau)^{-2+\frac{1}{q^*}\}}
	\end{equation}
	for all $0 \leq \tau \leq \tau_0$, for some $\varepsilon > 0$. Then, there exists $\alpha_1,\varepsilon^* > 0$, independent of $\tau_0$, such that for all $\alpha \in (0,\alpha_1]$ and $\varepsilon \in [0,\varepsilon^*]$, if $\|\psi^0\|_1,\|\rho^0\|_1 \leq \varepsilon$ we have 
	\[
		\begin{split}
			Q_2(T_3\psi(\tau)) &\leq 2\varepsilon\tilde{C}_L(1 + \tau)^{-\min\{\frac{3}{2},2 - \frac{1}{q^*}\}}, \\
			Q_{q^*}(T_3\psi(\tau)) &\leq 2\varepsilon\tilde{C}_L(1 + \tau)^{-2+\frac{1}{q^*}}, 	
		\end{split}
	\]
	for all $0 \leq \tau \leq \tau_0$. 
\end{lem}

\begin{proof}
	We begin by noting that restricting $\varepsilon \leq \delta^*$ we have guaranteed that $\|\rho^0\|_1\leq \varepsilon$ implies that $\|\rho(\alpha^{-1}\tau)\|_1 \leq 2\varepsilon e^{\frac{-\beta}{\alpha}\tau}$, from Theorem~\ref{thm:InvMan}. Hence, we ensure that $\varepsilon^* \leq \delta^*$. 
	
	Then, we begin with the bound on $Q_2$. Using (\ref{UMapping}) we obtain
	\[
		Q_2(T_3\psi(\tau)) \leq Q_2(e^{\tilde{L}_\alpha \tau}\psi^0) + \int_0^{\tau} Q_2(e^{\tilde{L}_\alpha (\tau-u)}[G(\sigma(\psi(u),\alpha)+\rho(\alpha^{-1}u),\psi(u),\alpha) - \tilde{L}_\alpha\psi(u)])\mathrm{d}u, 
	\]
	for all $0 \leq \tau \leq \tau_0$, and from Proposition~\ref{prop:LinearPhaseDecay} and Corollary~\ref{cor:GBnd} we obtain
	\begin{equation}\label{PhaseQ2Bnd}
		\begin{split}
		Q_2(T_3\psi(\tau)) &\leq \tilde{C}_L(1+\tau)^{-\min\{\frac{3}{2},2 - \frac{1}{q^*}\}}\|\psi^0\|_1 \\ 
		&+ \tilde{C}_LC_G\int_0^{\tau}(1+\tau - u)^{-\min\{\frac{3}{2},2 - \frac{1}{q^*}\}}\|\sigma(\psi(u),\alpha)+\rho(\alpha^{-1}u)\|_{q^*}\mathrm{d}u \\ 
		&+ \tilde{C}_LC_G\int_0^{\tau}(1+\tau - u)^{-\min\{\frac{3}{2},2 - \frac{1}{q^*}\}}\|\sigma(\psi(u),\alpha)+\rho(\alpha^{-1}u)\|_2^2\mathrm{d}u \\ 
		&+ \tilde{C}_LC_G\int_0^{\tau}(1+\tau - u)^{-\min\{\frac{3}{2},2 - \frac{1}{q^*}\}}\|\sigma(\psi(u),\alpha)+\rho(\alpha^{-1}u)\|_2 Q_2(\psi(u))\mathrm{d}u\\ 
		&+ \tilde{C}_LC_G\int_0^{\tau}(1+\tau - u)^{-\min\{\frac{3}{2},2 - \frac{1}{q^*}\}}Q_2^2(\psi(u))\mathrm{d}u	
		\end{split}
	\end{equation}
	where $\tilde{C}_L > 0$ is the constant coming from Proposition~\ref{prop:LinearPhaseDecay} and $C_G > 0$ is the constant coming from Corollary~\ref{cor:GBnd}. We now bound each integral in (\ref{PhaseQ2Bnd}) separately, and bring them together at the end.
	
	First, using (\ref{PhaseStabilityBnd1}) we get   
	\[
		\begin{split}
		\int_0^{\tau}(1+\tau - u)^{-\min\{\frac{3}{2},2 - \frac{1}{q^*}\}}&\|\sigma(\psi(u),\alpha)+\rho(\alpha^{-1}u)\|_{q^*}\mathrm{d}u \\ 
		&\leq 2\sqrt{\alpha}\varepsilon C_\mathrm{thm}\int_0^{\tau}(1+\tau - u)^{-\min\{\frac{3}{2},2 - \frac{1}{q^*}\}}Q_{q^*}(\psi(u))\mathrm{d}u \\
		&+\int_0^{\tau}(1+\tau - u)^{-\min\{\frac{3}{2},2 - \frac{1}{q^*}\}}\|\rho(\alpha^{-1}u)\|_1\mathrm{d}u. 
		\end{split}
	\]
	Then, recalling that $\|\rho(\alpha^{-1}\tau)\|_1\leq 2\varepsilon e^{-\frac{\beta}{\alpha}\tau}$ for all $0 \leq \tau \leq \tau_0$, we may apply Lemma~\ref{lem:IntegralLemmaExp} to find that there exists a constant $C_{\min\{\frac{3}{2},2 - \frac{1}{q^*}\},\beta} > 0$ such that
	\[
		\int_0^{\tau}(1+\tau - u)^{-\min\{\frac{3}{2},2 - \frac{1}{q^*}\}}\|\rho(\alpha^{-1}u)\|_1\mathrm{d}u \leq 2\alpha\varepsilon C_{\min\{\frac{3}{2},2 - \frac{1}{q^*}\},\beta}(1 + \tau)^{-\min\{\frac{3}{2},2 - \frac{1}{q^*}\}}, 
	\]
	for all $0 \leq \tau \leq \tau_0$. Then, the bound (\ref{Qq*PhaseBnd}) allows one to apply Lemma~\ref{lem:IntegralLemma} to find that
	\[
		\begin{split}
		\int_0^{\tau}(1+\tau - u)^{-\min\{\frac{3}{2},2 - \frac{1}{q^*}\}}&Q_{q^*}(\psi(u))\mathrm{d}u \\ 
		&\leq 2\sqrt{\alpha}\varepsilon \tilde{C}_L\int_0^{\tau}(1+\tau - u)^{-\min\{\frac{3}{2},2 - \frac{1}{q^*}\}}(1 + u)^{-2 + \frac{1}{q^*}}\mathrm{d}u \\
		&\leq2\sqrt{\alpha}\varepsilon \tilde{C}_LC_{\min\{\frac{3}{2},2 - \frac{1}{q^*}\},-2+\frac{1}{q^*}}(1 + \tau)^{-\min\{\frac{3}{2},2 - \frac{1}{q^*}\}}. 
		\end{split}	
	\]	
	Putting this together gives the bound
	\begin{equation}\label{PhaseQ2Bnd_1}
		\begin{split}
		\int_0^{\tau}&(1+\tau - u)^{-\min\{\frac{3}{2},2 - \frac{1}{q^*}\}}\|\sigma(\psi(u),\alpha)+\rho(\alpha^{-1}u)\|_{q^*}\mathrm{d}u \\ 
		&\leq4\sqrt{\alpha}\varepsilon^2[2\sqrt{\alpha}\delta^*C_{\min\{\frac{3}{2},2 - \frac{1}{q^*}\},\beta} + \tilde{C}_LC_{\min\{\frac{3}{2},2 - \frac{1}{q^*}\},2-\frac{1}{q^*}}] (1 + \tau)^{-\min\{\frac{3}{2},2 - \frac{1}{q^*}\}}.	
		\end{split}
	\end{equation}
	
	Very similar manipulations to those used to bound (\ref{PhaseQ2Bnd_1}) yield the bounds  
	\begin{equation}\label{PhaseQ2Bnd_3}
		\begin{split}
		\int_0^{\tau}&(1+\tau - u)^{-\min\{\frac{3}{2},2 - \frac{1}{q^*}\}}\|\sigma(\psi(u),\alpha)+\rho(\alpha^{-1}u)\|_2 Q_2(\psi(u))\mathrm{d}u \\
		&\leq 2\sqrt{\alpha}\varepsilon\tilde{C}_L(2\tilde{C}_LC_{\min\{\frac{3}{2},2 - \frac{1}{q^*}\},\min\{\frac{3}{2},2 - \frac{1}{q^*}\}} \\ 
		&+ \sqrt{\alpha}\delta^*C_{\min\{\frac{3}{2},2 - \frac{1}{q^*}\},\beta}])(1+\tau)^{-\min\{\frac{3}{2},2 - \frac{1}{q^*}\}},
		\end{split}
	\end{equation}
	and
	\begin{equation}\label{PhaseQ2Bnd_4}
		\int_0^{\tau}(1+\tau - u)^{-\min\{\frac{3}{2},2 - \frac{1}{q^*}\}}Q_2^2(\psi(u))\mathrm{d}u \leq 4\varepsilon^2\tilde{C}_L^2C_{\min\{\frac{3}{2},2 - \frac{1}{q^*}\},\min\{\frac{3}{2},2 - \frac{1}{q^*}\}}(1+\tau)^{-\min\{\frac{3}{2},2 - \frac{1}{q^*}\}}.	
	\end{equation}
	
	Putting (\ref{PhaseQ2Bnd_1})-(\ref{PhaseQ2Bnd_4}) into (\ref{PhaseQ2Bnd}) gives the bound 
	\[
		\begin{split}
		Q_2(T_3\psi(\tau)) &\leq (\tilde{C}_L\|\psi^0\|_1 + \varepsilon h(\alpha,\varepsilon))(1 + \tau)^{-\min\{\frac{3}{2},2 - \frac{1}{q^*}\}} \\ 
		&\leq \varepsilon(\tilde{C}_L + h(\alpha,\varepsilon))(1 + \tau)^{-\min\{\frac{3}{2},2 - \frac{1}{q^*}\}}, 	
		\end{split}
	\]
	where we have introduced the function $h(\alpha,\varepsilon)$ to collect all the constants coming from (\ref{PhaseQ2Bnd}) and (\ref{PhaseQ2Bnd_1})-(\ref{PhaseQ2Bnd_4}). The most important point is that $h(0,0) = 0$ and that $h$ depends continuously on $\alpha,\varepsilon \geq 0$ Hence, for sufficiently small $\alpha,\varepsilon \geq 0$ one can guarantee that $h(\alpha,\varepsilon) \leq \tilde{C}_L$, giving the bound on $Q_2(T_3\psi(\tau))$. The bound on $Q_{q^*}(T_3\psi(\tau))$ follows via nearly identical manipulations and is therefore omitted. 
\end{proof} 

\begin{lem}\label{lem:UFixedPt} 
	For all $\alpha \in (0,\alpha_1]$ and $\varepsilon \in [0,\varepsilon^*]$ and $\|\psi^0\|_1,\|\rho^0\|_1 \leq \varepsilon$ there exists a unique fixed point of $T_3$, denoted $\psi(\tau)$, such that $\psi(0) = \psi^0$ and 
	\begin{equation}\label{FixedPtDecay}
		\begin{split}
			\|\psi(\tau)\|_1 &\leq 2\varepsilon, \\
			Q_{2}(\psi(\tau)) &\leq 2\varepsilon\tilde{C}_L(1 + \tau)^{-\min\{\frac{3}{2},2 - \frac{1}{q^*}\}}, \\
			Q_{q^*}(\psi(\tau)) &\leq 2\varepsilon\tilde{C}_L(1 + \tau)^{-2 + \frac{1}{q^*}}, 
		\end{split}
	\end{equation}
	for all $\tau \geq 0$.
\end{lem}

\begin{proof}
	This proof proceeds in an identical manner to \cite[Theorem~4.5]{Bramburger3}, and therefore we will merely describe the important points which must be considered. The proof is simply a bootstrapping argument via an application of the Banach fixed point theorem. One begins by obtaining a fixed point of $T_3$ for all $\tau \in [0,1]$ by using Lemma~\ref{lem:QPhaseBnd} to show that $T_3$ is a well-defined mapping on a complete metric space of functions satisfying (\ref{Q2PhaseBnd}) and (\ref{Qq*PhaseBnd}) with $\tau_0 = 1$. Upon obtaining such a fixed point with the Banach fixed point theorem, we extend this solution to $\tau \in [0,2]$ by defining an new complete metric space of functions satisfying (\ref{Q2PhaseBnd}) and (\ref{Qq*PhaseBnd}) for $\tau_0 = 2$, such that these functions agree with the fixed point of $T_3$ on $[0,1]$. Again, Lemma~\ref{lem:QPhaseBnd} gives that $T_3$ would be a well-defined mapping in this case. The argument then proceeds inductively to show that if one has a fixed point on $[0,n]$ for any integer $n\geq 1$, then it can be extended to a fixed point on $[0,n+1]$ satisfying the decay rates (\ref{Q2PhaseBnd}) and (\ref{Qq*PhaseBnd}). This argument therefore gives the proof of the lemma.     
\end{proof}

Having now obtained a fixed point of $T_3$ satisfying the decay rates (\ref{FixedPtDecay}), we are now able to prove that the other decay rates of Theorem~\ref{thm:PhaseDecay} follow. The results of Lemma~\ref{lem:UFixedPt} and Corollary~\ref{cor:FixedPtDecay} therefore finish the proof of Theorem~\ref{thm:PhaseDecay}. The important point to note is that the following corollary dictates that it is only the decay rates on $Q_2(\psi(\tau))$ and $Q_{q^*}(\psi(\tau))$ that are required to obtain all other decay rates stated in Theorem~\ref{thm:PhaseDecay}. Of course, this should be apparent to the reader since Corollary~\ref{cor:GBnd} states the nonlinear terms are only bounded by $Q_2(\psi(\tau))$ and $Q_{q^*}(\psi(\tau))$.   

\begin{cor}\label{cor:FixedPtDecay} 
	There exists a constant $C_\psi > 0$ such that for all $\alpha \in (0,\alpha_1]$ and $\varepsilon \in [0,\varepsilon^*]$ and $\|\psi^0\|_1,\|\rho^0\|_1 \leq \varepsilon$ the unique fixed point of $T_3$ satisfying $\psi(0) = \psi^0$ and (\ref{FixedPtDecay}) further satisfies 
	\[
		\begin{split}
			\|\psi(t)\|_p &\leq \varepsilon C_\psi(1 + \alpha t)^{-1 + \frac{1}{p}}, \\
			Q_p(\psi(t)) &\leq \varepsilon C_\psi(1 + \alpha t)^{-\min\{\frac{3}{2},2 - \frac{1}{q^*}\}},
		\end{split}
	\]
	for all $\tau\geq 0$, $\alpha \in [0,\alpha_1]$, and $p \in [1,\infty]$.
\end{cor}

\begin{proof}
	First, since we assume $\psi(\tau)$ is a fixed point of $T_3$, then it necessarily satisfies (\ref{GSlowInt}). Then, for $p \in [1,\infty]$ fixed, from Proposition~\ref{prop:LinearPhaseDecay} and Corollary~\ref{cor:GBnd} 
	\[
		\begin{split}
		\|\psi(\tau)\|_p \leq (1 + \alpha t)^{-1 + \frac{1}{p}}\|\psi^0\|_1 &+ \tilde{C}_LC_G\int_0^{\tau}(1+\tau - u)^{-1 + \frac{1}{p}}\|\sigma(\psi(u),\alpha)+\rho(\alpha^{-1}u)\|_{q^*}\mathrm{d}u \\ 
		&+ \tilde{C}_LC_G\int_0^{\tau}(1+\tau - u)^{-1 + \frac{1}{p}}\|\sigma(\psi(u),\alpha)+\rho(\alpha^{-1}u)\|_2 Q_2(\psi(u))\mathrm{d}u\\ 
		&+ \tilde{C}_LC_G\int_0^{\tau}(1+\tau - u)^{-1 + \frac{1}{p}}Q_2^2(\psi(u))\mathrm{d}u.	
		\end{split} 
	\]
	From here bounding each of the integrals is nearly identical to the bounds (\ref{PhaseQ2Bnd_1})-(\ref{PhaseQ2Bnd_4}), with the notable exceptions being that the rate of decay will now be $-1 + \frac{1}{p}$ and the constants obtained from Lemmas~\ref{lem:IntegralLemma} and \ref{lem:IntegralLemmaExp} will reflect these different decay rates. By the continuity of the constants from Lemmas~\ref{lem:IntegralLemma} and \ref{lem:IntegralLemmaExp}, they are uniformly bounded in $p\in[1,\infty]$ and hence one obtains the desired bound on $\|\psi(t)\|_p$. A similar manipulation yields the bounds on $Q_p(\psi(t))$.     
\end{proof} 

\section{Discussion}\label{sec:Discussion} 

In this work we have provided a series of sufficient conditions that demonstrate the local asymptotic stability of periodic solutions to our Lambda-Omega lattice dynamical system. Recall that our first assumption, Hypotheses~\ref{hyp:LambdaOmega}, simply states that our Lambda-Omega system should generalized a spatially-discretized Ginzburg-Landau equation and that Hypothesis~\ref{hyp:PolarSoln} assumes the existence of a periodic solution that exists for sufficiently small positive coupling values $\alpha \geq 0$. These two assumptions alone are all that are required to demonstrate the existence of an invariant slow manifold which is locally asymptotically stable with uniform exponential rate of decay. Although this result is certainly not surprising, it is necessary for our understanding of local asymptotic stability in this Lambda-Omega setting, as well as demonstrates a useful extension of Hale's integral manifold theorems to the infinite-dimensional lattice dynamical system context.  

Upon proving the existence of a locally asymptotically stable invariant manifold to our Lambda-Omega system, we turned our attention to investigating the behaviour of trajectories starting near our periodic orbit on this manifold. In order to do so we required Hypothesis~\ref{hyp:LinearPhase} which assumed the decay of a certain semigroup with infinitesimal generator closely related to the linearization about the periodic solution on the invariant manifold. We recall that these decay rates were not arbitrary and were confirmed in Section~\ref{sec:Applications} for a wide variety of infinitesimal generators. Our final assumption, Hypothesis~\ref{hyp:pStar} requires that the radial solutions asymptotically become constant as one moves away from the centre of the lattice. Again, this assumption is quite technical, but we saw that it is satisfied for a number of interesting examples of periodic solutions to our Lambda-Omega lattice dynamical system. Furthermore, this assumption appears to be necessary since it allows one to achieve the result Lemma~\ref{lem:AsymBnd}, which in turn gave the linearized decay estimates of Proposition~\ref{prop:LinearPhaseDecay}. From the proofs of these auxiliary results it should be clear to the reader that simple boundedness of the radial components is not sufficient to obtain the nonlinear stability results of Theorem~\ref{thm:PhaseDecay} and therefore we necessitate these radial components asymptotically approach the same value at a fast enough rate. 

With regards to weakening hypotheses, a quick comparison between Hypothesis~\ref{hyp:LambdaOmega} and the assumptions put on the Lambda-Omega lattice dynamical system studied in \cite{Bramburger2} reveals a slight, but important, difference in the form of the function $\omega$. That is, our main result Theorem~\ref{thm:PhaseDecay} assumed that $\omega$ was a constant function, whereas this was not a necessary assumption to obtain rotating wave solutions to system (\ref{LambdaOmegaLDS}). Hence, it would be of interest to determine whether the class of functions $\omega$ for which stability results analogous to Theorem~\ref{thm:PhaseDecay} can be obtained can be expanded beyond that which is studied in this work. Hence, an interesting problem going forward would be determining non-constant functions $\omega$ that may still guarantee similar results to those of Theorem~\ref{thm:PhaseDecay}.

Another point to consider is that all initial conditions are taken from the Banach space $\ell^1$. The reason for this is that Hypothesis~\ref{hyp:LinearPhase} is stated in terms of initial conditions in $\ell^1$. It would be interesting to investigate how the results of Theorem~\ref{thm:PhaseDecay} change as one considers initial conditions in $\ell^p$, for various $p > 1$. What can be stated immediately is that the stability results fail when considering the full range of initial conditions in $\ell^\infty$. This is quite easy to see since $\ell^\infty$ contains the constant sequences indexed by $\mathbb{Z}^2$. To see this, denote $\mathbbm{1} = \{1\}_{(i,j)\in\mathbb{Z}^2}\in\ell^\infty$ to be the constant sequence of all ones in $\ell^\infty$ and note that for all $C \in \mathbb{R}$ we have that $(\bar{r}(\alpha),\bar{\theta}(\alpha) + C\mathbbm{1})$ is also a steady-state solution of (\ref{FullPolarLattice}) due to the fact that only the difference of nearest-neighbours in $\theta$ are present in the polar system. Hence, taking the initial condition $(s^0,\psi^0) = (0,C\mathbbm{1}) \in \ell^\infty\times\ell^\infty$ results in $\dot{s} = 0$ and $\dot{\psi} = 0$ for all $t \geq 0$, $C \in \mathbb{R}$ and $\alpha \in [0,\alpha^*]$, thus giving no decay back to the equilibrium $(s,\psi) = (0,0)$, regardless of how small $|C|$ is chosen. Therefore, it would be of great interest to see if the decay rates gradually weaken for initial conditions belonging to $\ell^p$ as $p$ increases, resulting in no decay in the limiting case of $p = \infty$. 

This work merely provides a necessary initial step in the study of stability in lattice dynamical systems through the use of traditional dynamical systems techniques. As this discussion has eluded to, there still remains a number of open problems related to this work. Away from the Lambda-Omega systems investigated here, the nonlinear stability of solutions to lattice dynamical systems remains one that is largely unexplored without the use of comparison theorems. Therefore, the hypotheses and techniques put forth in this manuscript could also be taken to inform future investigations into the local asymptotic stability of solutions to lattice dynamical systems.

\section*{Acknowledgements} 

This work was supported by an NSERC PDF held at Brown University.

\end{document}